\pdfoutput=1
\RequirePackage{rotating}
\documentclass{amsart}
\usepackage{amsmath}
\usepackage{amsfonts}
\usepackage{amssymb}
\usepackage{amsthm}
\usepackage{bm}
\usepackage{mathrsfs}
\usepackage{stmaryrd}
\usepackage{tikz,tikz-cd}
\usepackage[pagebackref=true]{hyperref}
\usetikzlibrary{arrows,chains,matrix,positioning,scopes,decorations.pathmorphing}
\usepackage[numbers]{natbib}
\usepackage{slashed}
\usepackage{todonotes}
\usepackage{graphicx}
\usepackage{microtype}
\usepackage{bbm}
\usepackage{hhline}

\renewcommand*{\backrefalt}[4]{%
	\ifcase #1 \footnotesize{(Not cited.)}%
	\or	\footnotesize{(Cited on page~#2)}
	\else	\footnotesize{(Cited on pages~#2)}%
	\fi}

\newtheorem{theorem}{Theorem}[section]
\newtheorem{lemma}[theorem]{Lemma}
\newtheorem{proposition}[theorem]{Proposition}
\newtheorem{corollary}[theorem]{Corollary}
{\theoremstyle{definition} \newtheorem{definition}{Definition}[section]}
{\theoremstyle{definition} }
{\theoremstyle{definition} \newtheorem*{exercise*}{Exercise}}
{\theoremstyle{definition}\newtheorem{example}{Example}}
{\theoremstyle{definition} }
{\theoremstyle{remark}\newtheorem*{remark}{Remark}}
{\theoremstyle{remark}}

{\theoremstyle{remark}\newtheorem{question}{Question}}
{\theoremstyle{definition}}
{\theoremstyle{definition}}

\newcommand{\F}{\mathbb{F}}
\newcommand{\R}{\mathbb{R}}

\newcommand{\Z}{\mathbb{Z}}

\newcommand{\id}{\mathrm{id}}

\newcommand{\CF}{\mathit{CF}}
\newcommand{\HF}{\mathit{HF}}
\newcommand{\CFA}{\mathit{CFA}}
\newcommand{\CFD}{\mathit{CFD}}

\newcommand{\CFDA}{\mathit{CFDA}}

\newcommand{\Mor}{\mathrm{Mor}}

\newcommand{\CKh}{{\mathcal{C}_{\mathit{Kh}}}}

\newcommand{\End}{\mathrm{End}}

\newcommand{\gr}{\mathrm{gr}}

\newcommand{\CH}{\mathit{CH}}
\newcommand{\HH}{\mathit{HH}}

\newcommand{\ACH}{\mathcal{CH}}
\newcommand{\AHH}{\mathcal{HH}}

\newcommand{\RW}{\mathcal{H}_{\mathit{RW}}}

\newcommand{\CKhred}{\widetilde{\mathcal{C}}_{\mathit{Kh}}}
\newcommand{\Kh}{\mathit{Kh}}
\newcommand{\Khred}{\widetilde{\mathit{Kh}}}

\newcommand{\reeb}[1]{\rho_{\scriptscriptstyle #1}}
\newcommand{\reebII}[2]{\rho_{\scriptscriptstyle #1}^{\scriptscriptstyle #2}}

\newcommand{\fn}[3]{#1_{\scriptscriptstyle #2}^{\scriptscriptstyle #3}}

\newcommand{\amatch}{\,\raisebox{-0.35cm}{\includegraphics[scale=0.5]{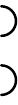}}\,}
\newcommand{\bmatch}{\,\raisebox{-0.35cm}{\includegraphics[scale=0.5]{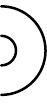}}\,}
\newcommand{\aaii}{\,\raisebox{-0.3cm}{\includegraphics[scale=0.5]{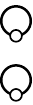}}\,}
\newcommand{\aaxi}{\,\raisebox{-0.3cm}{\includegraphics[scale=0.5]{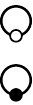}}\,}
\newcommand{\aaix}{\,\raisebox{-0.3cm}{\includegraphics[scale=0.5]{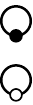}}\,}
\newcommand{\aaxx}{\,\raisebox{-0.3cm}{\includegraphics[scale=0.5]{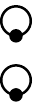}}\,}
\newcommand{\abi}{\,\raisebox{-0.3cm}{\includegraphics[scale=0.5]{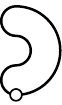}}\,}
\newcommand{\abx}{\,\raisebox{-0.3cm}{\includegraphics[scale=0.5]{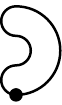}}\,}
\newcommand{\bai}{\,\raisebox{-0.3cm}{\includegraphics[scale=0.5]{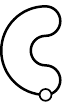}}\,}
\newcommand{\bax}{\,\raisebox{-0.3cm}{\includegraphics[scale=0.5]{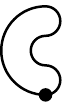}}\,}
\newcommand{\bbii}{\,\raisebox{-0.3cm}{\includegraphics[scale=0.5]{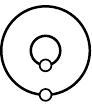}}\,}
\newcommand{\bbxi}{\,\raisebox{-0.3cm}{\includegraphics[scale=0.5]{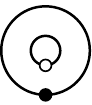}}\,}
\newcommand{\bbix}{\,\raisebox{-0.3cm}{\includegraphics[scale=0.5]{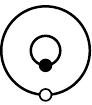}}\,}
\newcommand{\bbxx}{\,\raisebox{-0.3cm}{\includegraphics[scale=0.5]{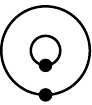}}\,}

\newcommand\scalemath[2]{\scalebox{#1}{\mbox{\ensuremath{\displaystyle #2}}}}

\title{An Ozsv\'{a}th--Szab\'{o}-type spectral sequence for links in $S^1\times S^2$}
\author{Jesse Cohen}
\address{Fachbereich Mathematik (AZ)\\Universit\"{a}t Hamburg\\
	Bundesstra\ss{}e 55\\
	20146 Hamburg, Germany}
\email{jesse.cohen@uni-hamburg.de}

\begin{document}
	\begin{abstract}
		We show that there is a spectral sequence with $E^2$-page given by the Khovanov homology of a link in $S^1\times S^2$, as defined by Rozansky in \cite{Rozansky}, which converges to the Hochschild homology of an $A_\infty$-bimodule defined in terms of bordered Floer invariants. We also show that the homology algebras $H_*\mathfrak{h}_n$ of the algebras $\mathfrak{h}_n$ over which these bimodules are defined give nontrivial $A_\infty$-deformations of Khovanov's arc algebras $H_n$ for $n>1$.
	\end{abstract}
	\maketitle
	\section{Introduction}
	Khovanov homology is a functorial invariant of links in $S^3$ introduced by Mikhail Khovanov in \cite{KhovanovCategorificationJones2000} and later extended in \cite{KhovanovFunctorTangle2002} to a functorial invariant of \emph{even tangles} --- i.e. tangles with an even number of left- and right-endpoints --- valued in a 2-category of complexes of bimodules over the \emph{arc algebras} $\{H_n\}_{n\geq 0}$. We also refer the reader to \cite{Bar-Natan2002} and \cite{Bar-Natan2005} for an alternate perspective on these invariants. It was shown by Lev Rozansky in \cite{Rozansky} that if $T$ is an $(n,n)$-tangle diagram, i.e. a tangle diagram with $2n$ left- and right-endpoints, then the Hochschild homology $\RW(L):=\HH(\CKh(T))$ of $H_n$ with coefficients in the complex $\CKh(T)$ of $(H_n,H_n)$-bimodules associated to $T$ is an isotopy invariant of the link $L\subset S^1\times S^2$ given by taking the closure of $T$, regarded as a tangle in $I\times D^2\subset I\times S^2$, in $S^1\times S^2$. This invariant was later recast in the language of Bar-Natan's diagrammatic category and extended to an invariant of links in $\#^rS^1\times S^2$ by Michael Willis in \cite{Willis}.

In \cite{OzsSzBranched}, Peter Ozsv\'{a}th and Zolt\'{a}n Szab\'{o} showed that there is a (singly-graded) spectral sequence with $E^2$-page isomorphic to the \emph{reduced} Khovanov homology $\Khred(mL;\F)$ of the mirror $mL$ of a link $L\subset S^3$, with $\F=\F_2$ coefficients, which converges to $\widehat{\HF}(\Sigma(L))$, the Heegaard Floer homology of the branched double cover $\Sigma(L)$ of $L$. In this paper, we extend the results of \cite{OzsSzBranched} to construct a spectral sequence relating the Rozansky--Willis homology of a link $L\subset S^1\times S^2$ and the Hochschild homology of an $A_\infty$-bimodule constructed using (bordered) Heegaard Floer homology.
\subsection{Organization and Results}
In Section 2, we give (minimal) background on $A_\infty$-algebras and modules, $A_\infty$-Hochschild homology, type-$D$ structures, and Khovanov bimodules. Along the way, we prove the following crucial algebraic lemma.
\begin{lemma}
	Let $\mathcal{A}$ be an $A_\infty$-algebra over a commutative ring $\Bbbk$ and suppose that $\mathcal{M}$ is an $A_\infty$-bimodule over $\mathcal{A}$ equipped with a finite filtration, then there is a spectral sequence with $E^1$-page given by the Hochschild complex $\CH(H_*\mathcal{A}[1],H_*(\gr_*\mathcal{M}))$, where we regard $H_*\mathcal{A}[1]$ as an associative algebra with trivial higher operations and $H_*(\gr_*\mathcal{M})$ as a complex of bimodules over $H_*\mathcal{A}[1]$, which converges to the $A_\infty$-Hochschild homology $\AHH(\mathcal{M}):=\AHH(\mathcal{A},\mathcal{M})$.
\end{lemma}
In Section 3, we detail a general construction of a differential algebra $\mathfrak{h}_{\bm{\nu}}$ associated to any collection $\bm{\nu}=\{(Y_i,\phi_i:F\stackrel{\cong}{\to}\partial Y_i)\}$ of bordered 3-manifolds --- i.e. 3-manifolds $Y_i$ with boundary equipped with orientation-preserving diffeomorphisms $\phi_i$ --- where $F$ is a fixed parametrized surface, and, for $\bm{\nu}'$ a second collection of 3-manifolds with boundary parameterized by $F'$, an $A_\infty$-bimodule $\CF_{\bm{\nu},\bm{\nu}'}(Y)$ over $(\mathfrak{h}_{\bm{\nu}},\mathfrak{h}_{\bm{\nu}'})$ associated to any cobordism $Y:F\to F'$.

In Section 4, we specialize the constructions of Section 3 to the case that $\bm{\nu}$ is the set of branched double covers of crossingless matchings on $2n+2$ points of the form $a_+$, where $a$ is a crossingless matching on $2n$ points, $n\geq 1$, and $a_+$ is the result of adding a single arc below $a$, regarded as tangles in $B^3$ with endpoints on the equator. This gives us a collection of differential algebras $\{\mathfrak{h}_n\}_{n\geq 0}$ (taking the convention that $\mathfrak{h}_0=\F$), which we call \emph{branched arc algebras}. The main result of this section is the following theorem.
\begin{theorem}
	There is an isomorphism $H_*\mathfrak{h}_n\cong H_n$ of associative algebras.
\end{theorem}
We show in Section 6 that $\mathfrak{h}_n$ is not \emph{formal}, i.e. that $H_*\mathfrak{h}_n$ and $H_n$ are not quasi-isomorphic as $A_\infty$-algebras, for $n>1$. In order to do so, we show that there are homologically injective embeddings $\mathfrak{h}_n\hookrightarrow\mathfrak{h}_{n+1}$ for all $n$. In the appendix, we compute $\mathfrak{h}_2$ and $H_*\mathfrak{h}_2$ explicitly and show that $H_*\mathfrak{h}_2$ has a nontrivial $m_3$ operation (independent of choice of retract) --- in fact, we show that $H_*\mathfrak{h}_2$ is an unbounded $A_\infty$-algebra --- thus completing the proof that $\mathfrak{h}_n$ is non-formal.

In Section 5, we specialize the construction of our bimodules to the case that the cobordism $Y$ is the branched double cover of an $(n+1,n+1)$-tangle associated to an $(n,n)$-tangle $T\subset I\times S^2$, giving us an $A_\infty$-bimodule $\CF^\beta(T)$ over $\mathfrak{h}_n$ which we call the \emph{branched Floer complex} of $T$. Using a result of Lipshitz--Ozsv\'{a}th--Thurston from \cite{LOTSpectral2}, we then prove our main theorem.
\begin{theorem}
	$\CF^\beta(T)$ admits a filtration such that $H_*(\gr_*\CF^\beta(T))$ is isomorphic to $\CKh(mT)$ as a complex of bimodules over $H_n\cong H_*\mathfrak{h}_n$. As a consequence, there is a spectral sequence with $E^2$-page $\HH(\CKh(mT))$ which converges to $\AHH(\CF^\beta(T))$. In other words, if $mL\subset S^1\times S^2$ is the closure of $mT$ in $S^1\times S^2$, there is a spectral sequence $\RW(mL;\F)\Rightarrow\AHH(\CF^\beta(T))$.
\end{theorem}
\subsection*{Acknowledgements}
The author would like to thank Tobias Dyckerhoff, Gary Guth, Robert Lipshitz, Ikshu Neithalath, Paul Wedrich, and Michael Willis for several helpful conversation. The material in Section 4 and the appendices make up Sections 3.1 and 3.2 of the author's Ph.D thesis.

This research was supported in part by the National Science Foundation under Grant Nos. DMS-2204214 and DMS-1928930 and in part by the Deutsche Forschungsgemeinschaft (DFG, German Research Foundation) under Germany’s Excellence Strategy - EXC 2121 ``Quantum Universe'' - 390833306.
	\section{Background}
	\subsection{$A_\infty$-algebras and modules}
We collect here a few basic definitions and facts regarding $A_\infty$-algebras and their (bi)modules. Since it will suffice for our purposes, we will work exclusively over a semisimple commutative ring $\Bbbk$ of characteristic 2, i.e. $\Bbbk=\F e_1\oplus\cdots\oplus\F e_\ell$, where the $e_i$ are orthogonal idempotents and $\F=\F_2$. We refer the reader to \cite{Keller2001,Seidel2008} for more thorough treatments of the subject of $A_\infty$-algebras and modules in general and to \cite{LOTBimodules2015,mescher2016primer} for more details on Hochschild (co)homology for $A_\infty$-bimodules, which we will discuss briefly in Subsection \ref{subsec:HochschildHomology}.
\begin{definition}
	An $A_\infty$-algebra $\mathcal{A}$ over $\Bbbk$ is a graded $\Bbbk$-bimodule $\mathcal{A}$ equipped with a collection of degree 0 $\Bbbk$-linear maps $\mu_i:\mathcal{A}^{\otimes i}\to\mathcal{A}[2-i]$ for $i\geq 1$, where $\mathcal{A}[k]$ is $\mathcal{A}$ shifted in degree by $k$, satisfying the \emph{$A_\infty$-relations}:
	\begin{align}
		\sum_{i+j=n+1}\sum_{k=1}^{n-j+1}\mu_i\circ(\id^{\otimes k-1}\otimes \mu_j\otimes\id^{\otimes n-k-j+1})=0
	\end{align}
	for every $n\geq 1$. An $A_\infty$-algebra $\mathcal{A}$ is \emph{strictly unital} if there is an element $\mathbbm{1}\in\mathcal{A}$ such that $\mu_2\circ(\mathbbm{1}\otimes\id)=\mu_2\circ(\id\otimes\mathbbm{1})=\id$ and $\mu_j\circ(\id^{\otimes i-1}\otimes\mathbbm{1}\otimes\id^{\otimes j-i})=0$ for any $j\neq 2$ and $1\leq i\leq j$. We say $\mathcal{A}$ is \emph{operationally bounded} if there exists some $N\geq 1$ such that $\mu_n=0$ whenever $n\geq N$, and \emph{unbounded} otherwise.
\end{definition}
It is often convenient to assemble the maps $\mu_i$ into a single map $\mu:T(\mathcal{A}[1])\to\mathcal{A}[2]$, where $T(\mathcal{A}[1])=\bigoplus\limits_{k=0}^\infty\mathcal{A}^{\otimes_\Bbbk k}[k]$ is the tensor algebra of $\mathcal{A}[1]$ and we set $\mu_0=0$. If we further define a map $\overline{D}^{\mathcal{A}}:T(\mathcal{A}[1])\to T(\mathcal{A}[1])$ by
\begin{align}
	\overline{D}^{\mathcal{A}}|_{\mathcal{A}^{\otimes n}[n]}=\sum_{j=1}^n\sum_{k=1}^{n-j+1}\id^{\otimes k-1}\otimes\mu_j\otimes\id^{n-j-k+1},
\end{align}
then the $A_\infty$-relations are equivalent to either of the equations $\mu\circ\overline{D}^{\mathcal{A}}=0$ or $\overline{D}^{\mathcal{A}}\circ\overline{D}^{\mathcal{A}}$, which we may interpret in the graphical language of \cite{LOTBHFH} as
\begin{align}
	\begin{tikzcd}[ampersand replacement=\&,row sep=0.5cm]
		{}\arrow[d,Rightarrow]\\\overline{D}^{\mathcal{A}}\arrow[d,Rightarrow]\\\mu\arrow[d]\\{}
	\end{tikzcd}\!\!=0\hspace{0.5cm}\textup{or}\hspace{0.5cm}\begin{tikzcd}[ampersand replacement=\&,row sep=0.5cm]
	{}\arrow[d,Rightarrow]\\\overline{D}^{\mathcal{A}}\arrow[d,Rightarrow]\\\overline{D}^{\mathcal{A}}\arrow[d,Rightarrow]\\{}
\end{tikzcd}\!\!=0.
\end{align}
As a special case, any dg-algebra $\mathcal{A}$ is automatically an $A_\infty$-algebra with $\mu_1$ given by the differential, $\mu_2$ the multiplication on $\mathcal{A}$, and $\mu_i=0$ for $i>2$.
\begin{remark}
	In Section \ref{sec:formality}, we provide a brief exposition on the homological perturbation lemma which, given a choice of retract, will allow us to regard the homology $H_*\mathcal{A}$ of a dg-algebra as an $A_\infty$-algebra. If one forgets the higher structure maps, we may regard $H_*\mathcal{A}$ as a strictly associative algebra, with multiplication given by the map on homology induced by $\mu_2$.
\end{remark}
\begin{definition}
	Let $\mathcal{A}$ and $\mathcal{B}$ be strictly unital $A_\infty$-algebras over $\Bbbk$ and $\mathbbm{j}$, respectively. An $A_\infty$-bimodule over $(\mathcal{A},\mathcal{B})$ is a graded $(\Bbbk,\mathbbm{j})$-bimodule $\mathcal{M}$ equipped with degree 0 maps $m_{i|1|j}:\mathcal{A}[1]^{\otimes i}\otimes\mathcal{M}\otimes\mathcal{B}[1]^{\otimes j}\to\mathcal{M}[1]$ satisfying the compatibility relation
	\begin{align}
		\begin{tikzcd}[ampersand replacement=\&,row sep=0.5cm,column sep=0.35cm]
			{}\arrow[d,Rightarrow] \& {}\arrow[dd,dashed] \& {}\arrow[d,Rightarrow]\\
			\Delta\arrow[dr,Rightarrow]\arrow[ddr,Rightarrow,bend right=20] \& \& \Delta\arrow[dl,Rightarrow]\arrow[ddl,Rightarrow,bend left=20]\\
			\& m\arrow[d,dashed] \& \\
			\& m\arrow[d,dashed] \& \\
			\& {} \&
		\end{tikzcd}+\begin{tikzcd}[ampersand replacement=\&,row sep=0.5cm,column sep=0.35cm]
		{}\arrow[d,Rightarrow] \& {}\arrow[dd,dashed] \& {}\arrow[ddl,Rightarrow,bend left=20]\\
		\overline{D}^{\mathcal{A}}\arrow[dr,Rightarrow] \& \& \\
		\& m\arrow[dd,dashed] \& \\
		\& \& \\
		\& {} \&
	\end{tikzcd}+\begin{tikzcd}[ampersand replacement=\&,row sep=0.5cm,column sep=0.35cm]
	{}\arrow[ddr,bend right=20,Rightarrow] \& {}\arrow[dd,dashed] \& {}\arrow[d,Rightarrow]\\
	\& \& \overline{D}^{\mathcal{B}}\arrow[dl,Rightarrow]\\
	\& m\arrow[dd,dashed] \& \\
	\& \& \\
	\& {} \&
\end{tikzcd}=0,
	\end{align}
	where $m=\sum\limits_{i,j}m_{i|1|j}$ and $\Delta$ on the left and right of the first term denote the canonical comultiplication maps on $T(\mathcal{A}[1])$ and $T(\mathcal{B}[1])$, respectively. As is typical in the literature, we will use solid arrows to denote elements of an algebra, doubled arrows to denote elements of its tensor algebra, and dashed arrows to denote elements of modules.
\end{definition}
\subsection{Hochschild homology}\label{subsec:HochschildHomology}
We now give a brief introduction to Hochschild homology for $A_\infty$-bimodules. Fix an $A_\infty$-algebra $\mathcal{A}$ over $\Bbbk$ and an $A_\infty$-bimodule $\mathcal{M}$ over $\mathcal{A}$ with structure maps $m_{i|1|j}:\mathcal{A}[1]^{\otimes i}\otimes_\Bbbk\mathcal{M}\otimes_\Bbbk\mathcal{A}[1]^{\otimes j}\to\mathcal{M}[1]$. The \emph{Hochschild complex} $\ACH(\mathcal{M})$ of $\mathcal{M}$ is the quotient $\F$-vector space
\begin{align}
	\ACH(\mathcal{M})=\mathcal{M}\otimes_\Bbbk T(\mathcal{A}[1])/\sim,
\end{align}
where $T(\mathcal{A}[1])$ is the tensor algebra of $\mathcal{A}[1]$ in the category of $\Bbbk$-modules and $\sim$ is the equivalence relation generated by
\begin{align}
	(e\cdot x)\otimes a_1\otimes\cdots\otimes a_k\sim x\otimes a_1\otimes\cdots\otimes (a_k\cdot e),
\end{align}
where $x\in\mathcal{M}$ and e ranges over $\Bbbk$. We will use bar notation $x|a_1|\cdots|a_k$ for the class $[x\otimes a_1\otimes\cdots\otimes a_k]$. The differential $D$ on $\ACH(\mathcal{M})$ is given by
\begin{align}
	\begin{split}
		&D(x|a_1|\cdots|a_\ell)\\&=\sum_{i+j\leq\ell}m_{i|1|j}(a_{\ell-i+1},\dots,a_\ell,x,a_1,\dots,a_j)|a_{j+1}|\cdots|a_{\ell-i}\\
		&+\sum_{1\leq i<j\leq\ell}x|a_1|\cdots|\mu_{j-i}(a_i,\dots,a_{j-1})|a_j|\cdots|a_\ell
	\end{split}
\end{align}
which we may interpret graphically as
\begin{align}
	D=\begin{tikzcd}[ampersand replacement=\&,row sep=0.35cm,column sep=0.35cm]
		\& \arrow[ddd,dashed] \& \arrow[dd,Rightarrow] \& \\
		\& \& \& \\
		{}\arrow[dr,Rightarrow]\& \& \Delta\arrow[ddd,Rightarrow]\arrow[dl,Rightarrow]\arrow[dr,Rightarrow] \&\\
		\& m\arrow[dd,dashed] \& \& {} \\
		\& \& \& \\
		\& {} \& {} \&
	\end{tikzcd}+\begin{tikzcd}[ampersand replacement=\&,row sep=0.35cm,column sep=0.35cm]
		\arrow[ddddd,dashed] \& \arrow[dd,Rightarrow] \\
		\& \\
		\& \overline{D}^{\mathcal{A}}\arrow[ddd,Rightarrow]\\
		\& \\
		\& \\
		{} \& {}
	\end{tikzcd}
\end{align}
where, in the left-hand diagram, the arrow which runs off to the right out of $\Delta$ is identified with the arrow that comes into $m$ from the left. The homology $\AHH(\mathcal{M})$ of this complex is called the \emph{Hochschild homology} of $\mathcal{M}$. We will need to distinguish between $A_\infty$-algebra and associative algebra structures, and between $A_\infty$- and ordinary differential bimodule structures, on the same vector spaces $\mathcal{A}$ and $\mathcal{M}$, respectively. Consequently, we will denote the Hochschild complex and homology by $\CH(\mathcal{M})$ and $\HH(\mathcal{M})$ when $\mathcal{A}$ is an associative algebra and $\mathcal{M}$ is a differential bimodule over $\mathcal{A}$.
\subsubsection{Filtrations on $\ACH(\mathcal{M})$ and their associated spectral sequences}
Suppose that $\mathcal{M}$ is an $A_\infty$-bimodule over $\mathcal{A}$ equipped with a (possibly-trivial) finite ascending filtration $\cdots\supseteq\mathcal{M}^2\supseteq\mathcal{M}^1\supseteq\mathcal{M}^0$ by $A_\infty$-submodules. Let $\ACH_n(\mathcal{M})$ be the image of the summand $\mathcal{M}\otimes_\Bbbk\mathcal{A}[1]\otimes_\Bbbk\stackrel{n}{\cdots}\otimes_\Bbbk\mathcal{A}[1]$ of $\mathcal{M}\otimes_\Bbbk T(\mathcal{A}[1])$ in $\ACH(\mathcal{M})$. The \emph{word length} filtration on $\ACH(\mathcal{M})$ is then the ascending filtration $\cdots\supseteq\mathcal{L}^2\supseteq\mathcal{L}^1\supseteq\mathcal{L}^0$ (taking the convention that $\mathcal{A}[1]^{\otimes 0}=\Bbbk$) defined by
\begin{align}
	\mathcal{L}^n=\bigoplus_{m\leq n}\ACH_m(\mathcal{M}).
\end{align}
\begin{remark}
	Note that, since we are working in the setting of $A_\infty$-algebras and bimodules, word length is not in general a grading with respect to the differential $D$. However, when $\mathcal{M}$ and $\mathcal{A}$ both have integral gradings, there is a grading on $\ACH(\mathcal{M})$ given on basis elements by
	\begin{align}
		\gr(x|a_1|\cdots|a_n)=n+\gr(x)+\gr(a_1)+\cdots+\gr(a_n).
	\end{align}
\end{remark}
\begin{lemma}[{\cite[Lemma 5.5]{mescher2016primer}}]
	The length filtration is weakly convergent.
\end{lemma}
The \emph{total filtration} on $\ACH(\mathcal{M})$ is the tensor product filtration $\mathcal{F}^\bullet=\mathcal{M}^\bullet\otimes\mathcal{L}^\bullet$ given by
\begin{align}
	\mathcal{F}^k=\sum_{i+j=k}\mathcal{M}^i\otimes\mathcal{L}^j.
\end{align}
\begin{corollary}
	The total filtration $\mathcal{F}^\bullet$ on $\ACH(\mathcal{M})$ is weakly convergent.
\end{corollary}
\begin{proof}
	This is an immediate consequence of the fact that $\mathcal{L}^\bullet$ is weakly convergent and $\mathcal{M}^\bullet$ is a finite filtration.
\end{proof}
\begin{remark}
	If $\mathcal{M}^\bullet$ is a descending filtration, we can prove analogous results by replacing $\mathcal{L}$ with the reflected word length filtration $\overline{\mathcal{L}}$ on $\ACH(\mathcal{M})$, defined by $\overline{\mathcal{L}}^n=\mathcal{L}^{-n}$.
\end{remark}
\begin{theorem}
	The spectral sequence of the total filtration converges to $\AHH(\mathcal{M})$.
\end{theorem}
\begin{proof}
	The total filtration is weakly convergent and exhaustive so this follows from \cite[Theorem 3.2]{McCleary2000}.
\end{proof}
We denote the $E^r$-page of this spectral sequence by $\ACH^r(\mathcal{M})$. We will now describe $\ACH^0(\mathcal{M})$ and $\ACH^1(\mathcal{M})$ in detail. As a vector space, we have
\begin{align}
	\ACH^0(\mathcal{M})=\ACH(\gr_*\mathcal{M})=\gr_*\mathcal{M}\otimes_\Bbbk T(\mathcal{A}[1])/\sim,
\end{align}
where $\gr_*\mathcal{M}$ is the associated graded object for the filtration $\mathcal{M}^\bullet$. Note that the filtration non-decreasing part $D_0$ of the differential $D$ with respect to the filtration $\mathcal{F}^\bullet$ is given on basis elements by
\begin{align}
	\begin{split}
		D_0(x|a_1|\cdots|a_k)=&m_{0|1|0}^0(x)|a_1|\cdots|a_k\\&+\sum_{1\leq i\leq k}x|a_1|\cdots|\mu_1(a_i)|\cdots|a_k,
	\end{split}
\end{align}
where $m_{0|1|0}^0$ is the filtration non-decreasing part of $m_{0|1|0}$ with respect to the filtration $\mathcal{M}^\bullet$. As a consequence, since $\Bbbk$ is a finite direct product of copies of $\F$, the $E^1$-page of the spectral sequence is given as a vector space by
\begin{align}
	\ACH^1(\mathcal{M})=\CH(H_*(\gr_*\mathcal{M}))=H_*(\gr_*\mathcal{M})\otimes_\Bbbk T(H_*\mathcal{A}[1])/\sim.
\end{align}
Note that the differential $D_1$ on $\ACH^1(\mathcal{M})$ is given by the ordinary Hochschild differential, forgetting all higher operations, regarding $H_*\mathcal{A}[1]$ as an associative algebra with multiplication given by $(\mu_2)_*$, and $H_*(\gr_*\mathcal{M})$ as a complex of bimodules over $H_*\mathcal{A}[1]$ with differential induced by the filtration degree -1 part $m_{0|1|0}^{-1}$ of $m_{0|1|0}$, i.e.
\begin{align}
	\begin{split}
		&D_1([x]|[a_1]|\cdots|[a_k])\\&=(m_{0|1|0}^{-1})_*[x]|[a_1]|\cdots|[a_k]\\&+(m_{1|1|0})_*([a_k],[x])|[a_1]|\cdots|[a_{k-1}]+(m_{0|1|1})_*([x],[a_1])|[a_2]|\cdots|[a_k]\\&+\sum_{1\leq i\leq k-1}[x]|[a_1]|\cdots|(\mu_2)_*([a_i],[a_{i+1}])|\cdots|[a_k].
	\end{split}
\end{align}
As a consequence, the $E^2$-page of the spectral sequence of the total filtration is $\ACH^2(\mathcal{M})=\HH(H_*(\gr_*\mathcal{M}))$. Putting the above results and remarks together, we have the following.
\begin{theorem}\label{thm:GeneralSpectralSequence}
	There is a spectral sequence $\HH(H_*(\gr_*\mathcal{M}))\Rightarrow\AHH(\mathcal{M})$.\hfill\qedsymbol
\end{theorem}
\subsection{Type-$D$ structures}
\begin{definition}
	Let $\mathcal{A}$ be a differential algebra over a ring $\Bbbk$ with differential $\mu_1$ and multiplication map $\mu_2$. A (left) type $D$ structure over $\mathcal{A}$ consists of a graded $\Bbbk$-module $N$ equipped with a $\Bbbk$-linear morphism $\delta^1:N\to(\mathcal{A}\otimes_\Bbbk N)[1]$ satisfying the compatibility condition
	\begin{align}
		(\mu_2\otimes\id_N)\circ(\id_{\mathcal{A}}\otimes\delta^1)\circ\delta^1+(\mu_1\otimes\id_N)\circ\delta^1=0,
	\end{align}
	which can be represented graphically as
	\begin{align}
		\begin{tikzcd}[ampersand replacement=\&,column sep=0.35cm]
			\& {}\arrow[d,dashed]\\
			\& \delta^1\arrow[dd,dashed]\arrow[dl,bend right]\\
			\mu_1\arrow[d]\& \\
			{}\& {}\\
		\end{tikzcd}+\begin{tikzcd}[ampersand replacement=\&,column sep=0.35cm]
			\& \arrow[d,dashed]\\
			\& \delta^1\arrow[d,dashed]\arrow[ddl,bend right=23]\\
			\& \delta^1\arrow[dl]\arrow[dd,dashed]\\
			\mu_2\arrow[d]\& \\
			{}\& {}\\
		\end{tikzcd}=0.
	\end{align}
\end{definition}
A type-$D$ structure homomorphism is a $\Bbbk$-module map $f:N_1\to\mathcal{A}\otimes N_2$ satisfying the equation
\begin{align}
	(\mu_2\otimes\id_{N_2})\circ(\id_\mathcal{A}\otimes f)\circ\delta_{N_1}^1+(\mu_2\otimes\id_{N_2})\circ(\id_\mathcal{A}\otimes\delta_{N_2}^1)\circ f+(\mu_1\otimes\id_{N_2})\circ f=0
\end{align}
and a homotopy between type-$D$ structure homomorphisms $f,g:N_1\to\mathcal{A}\otimes N_2$ is a $\Bbbk$-module homomorphism $h:N_1\to(\mathcal{A}\otimes N_2)[-1]$ satisfying
\begin{align}
	\begin{split}
		f-g=&(\mu_2\otimes\id_{N_2})\circ(\id_{\mathcal{A}}\otimes h)\circ\delta_{N_1}^1+(\mu_2\otimes\id_{N_2})\circ(\id_{\mathcal{A}}\otimes\delta_{N_2}^1)\circ h\\&+(\mu_1\otimes\id_{N_2})\circ h.
	\end{split}
\end{align}
\begin{example}
	Suppose that $X$ is a differential $\mathcal{A}$-module which is free with basis $\{x_i\}$ and has differential determined by
	\begin{align}
		\partial x_i=\sum_ja_{ij}x_j.
	\end{align}
	Let $N=\mathrm{span}_\Bbbk\{x_i\}$.	Then the map $\delta^1:N\to(\mathcal{A}\otimes_\Bbbk N)[1]$ defined on basis elements by
	\begin{align}
		\delta^1(x_i)=\sum_ja_{ij}\otimes x_j
	\end{align}
	makes the pair $(N,\delta^1)$ into a type $D$ structure. Any dg-module homomorphism $f:X_1\to X_2$ induces a corresponding map of type-$D$ structures $f^1:N_1\to\mathcal{A}\otimes N_n$ and the converse is true for type-$D$ structures obtained in the above manner. Similarly, homotopies of such maps are equivalent to homotopies of type-$D$ structures.
\end{example}
On the other hand, if $(N,\delta^1)$ is a left type-$D$ structure over $\mathcal{A}$, then $\mathcal{A}\otimes_\Bbbk N$ is a left differential $\mathcal{A}$-module with differential
\begin{align}
	m_1=(\mu_2\otimes\id_N)\circ(\id_{\mathcal{A}}\otimes\delta^1)+\mu_1\otimes\id_N
\end{align}
and module structure map $m_2=\mu_2\otimes\id_N$. As in the above example, type-$D$ homomorphisms and homotopies induce chain homomorphisms and chain homotopies, respectively.
\begin{definition}
	Given a left type $D$ structure $(N,\delta^1)$ over a dg-algebra $\mathcal{A}$, there are higher structure maps $\delta^k:N\to(\mathcal{A}^{\otimes k}\otimes_\Bbbk N)[k]$ defined recursively by
	\begin{align}
		\delta^k=(\id_{\mathcal{A}^{\otimes(k-1)}}\otimes\delta^1)\circ\delta^{k-1}.
	\end{align}
	A type-$D$ structure $(N,\delta^1)$ is \emph{operationally bounded} --- or just \emph{bounded} --- if $\delta^k=0$ for all $k$ sufficiently large and \emph{unbounded} otherwise.
\end{definition}
\begin{remark}
	It is frequently useful, in the case that a type $D$ structure $(N,\delta^1)$ is constructed from a differential $\mathcal{A}$-module with a finite $\mathcal{A}$-basis $\{x_i\}$, to represent it as a directed graph $\Gamma:=\Gamma_{(N,\delta^1)}$ with vertices $x_i$ and one edge $x_i\to x_j$ labeled by $a_{ij}\in\mathcal{A}$ for each $i$ and $j$ --- where, by convention, an unlabeled arrow corresponds to $a_{ij}=1$. Framed in this way, it is easy to see that $(N,\delta^1)$ is operationally bounded if and only if the corresponding graph $\Gamma$ contains no directed cycles, except possibly those in which there are successive edges $x_i\stackrel{a_{ij}}{\to}x_j$ and $x_j\stackrel{a_{jk}}{\to}x_k$ such that $a_{ij}\otimes a_{jk}=0\in \mathcal{A}\otimes_\Bbbk \mathcal{A}$.
\end{remark}
\begin{example}
	Let $\mathcal{A}$ be the associative $\F$-algebra $\F[a]/(a^2)$ with zero differential and consider the free $\mathcal{A}$-module $X=\mathcal{A}\langle x\rangle$ with differential given by $\partial x=ax$. Then the corresponding type $D$ structure $(N,\delta^1)$ with $N=\F\langle x\rangle$ is the one whose associated directed graph is
	\begin{align*}
		\Gamma=\begin{tikzcd}
			x\arrow[loop right,out=35,in=-35,looseness=10,"a"]
		\end{tikzcd}
	\end{align*}
	i.e. $\delta^k(x)=a\otimes\stackrel{k}{\cdots}\otimes a\otimes x$. In particular, this an example of an unbounded type $D$ structure.
\end{example}
This graphical interpretation of type-$D$ structures is especially useful for efficiently computing complexes of module homomorphisms between them: if $(N_1,\partial_1)$ and $(N_2,\partial_2)$ are dg-modules over a differential algebra $\mathcal{A}$, then the chain complex $\Mor^{\mathcal{A}}(N_1,N_2)$ of $\mathcal{A}$-module homomorphisms $N_1\to N_2$ is naturally a complex of modules over the ground ring $\Bbbk$ when equipped with the differential $\partial f=\partial_2\circ f+f\circ\partial_1$.

If we instead regard $N_1$ and $N_2$ as type-$D$ structures by specifying a preferred $\mathcal{A}$-basis, we may equivalently treat $\Mor^{\mathcal{A}}(N_1,N_2)$ as the space of $\Bbbk$-linear maps $N_1\to\mathcal{A}\otimes_\Bbbk N_2$ with differential as given in (\ref{line:morphismdifferential}).
\begin{definition}\label{MorphismGraph}
	If $\Gamma_i$, $i=1,2$, is the graph for the type-$D$ structure associated to $(N_i,\partial_i)$ and $f:N_1\to N_2$ is a module homomorphism with
	\begin{align}
		f(x_i)=\sum_{j}f_{ij}y_j,
	\end{align}
	we may form a new graph $\Gamma_f$ from $\Gamma_1\sqcup\Gamma_2$ by adding a new edge $x_i\stackrel{f_{ij}}{\to}y_j$ for each nonzero term in $f(x_i)$ for all $i$.
\end{definition}
This new graph is the graph for the mapping cone of $f$ and represents a type-$D$ structure if and only if $f$ is a chain map. The morphism $\partial f$ can then be computed by summing over all length 2 paths in $\Gamma_f$ which contain one of these new edges, in the sense that a path of the form $x_{h}\stackrel{a}{\to}x_{i}\stackrel{f_{ij}}{\to}y_j$ contributes a summand of $(\partial f)(x_h)$ of the form $af_{ij}y_j$ and a path of the form $x_i\stackrel{f_{ij}}{\to}y_j\stackrel{b}{\to}y_k$ contributes a summand of the form $f_{ij}y_k$.
\begin{example}
	Consider the algebra $\mathcal{A}(\mathbb{T}^2)$ associated to the torus by bordered Floer homology and let $(N,\delta^1)$ be the left type-$D$ structure over $\mathcal{A}(\mathbb{T}^2)$ with $N=\F\langle x,y,z\rangle$, idempotents $x=\iota_1x$, $y=\iota_0y$, $z=\iota_0z$, and associated graph
	\begin{align}
		\begin{tikzcd}[ampersand replacement=\&,column sep=0.3cm]
			y\arrow[rr]\arrow[dr,"\rho_3"'] \& \& z\\
			\& x\arrow[ur,"\rho_2"'] \&
		\end{tikzcd}.
	\end{align}
	Consider the endomorphism $f:N\to N$ given by $f(z)=\rho_{12}y$ and $f(x)=f(y)=0$. Then
	\begin{align}
		\Gamma_f=\,\,\begin{tikzcd}[ampersand replacement=\&,column sep=0.3cm]
			y\arrow[rr]\arrow[dr,"\rho_3"'] \& \& z\arrow[rr,"\rho_{12}"] \& \& y\arrow[rr]\arrow[dr,"\rho_3"'] \& \& z\\
			\& x\arrow[ur,"\rho_2"'] \& \& \& \& x\arrow[ur,"\rho_2"'] \&
		\end{tikzcd}
	\end{align}
	and we can read off $\partial f$ from this graph as
	\begin{align}
		\partial f=[y\mapsto\rho_{12}y]+[z\mapsto\rho_{123}x]+[z\mapsto\rho_{12}z].
	\end{align}
	Note that the path $x\stackrel{\rho_2}{\to}z\stackrel{\rho_{12}}{\to}y$ does not contribute to $\partial f$ since $\rho_2\rho_{12}=0$.
\end{example}
\subsubsection{Pairing $A_\infty$-modules and type-$D$ structures}
Let $\mathcal{A}$ be a differential algebra, and suppose that $M$ is a right $A_\infty$-module and $N$ is a left type-$D$ structure, both over $\mathcal{A}$. Then the box tensor product $M\boxtimes N:=M\otimes_{\Bbbk}N$ becomes a chain complex when endowed with the box tensor differential
\begin{align}
	\partial^\boxtimes=\begin{tikzcd}[column sep=0.35cm,ampersand replacement=\&]
		{}\arrow[dd,densely dashed] \& {}\arrow[d,densely dashed]\\
		{} \& \delta_N\arrow[dl,Rightarrow]\arrow[dd,densely dashed]\\
		m_M\arrow[d,densely dashed] \& {}\\
		{} \& {}
	\end{tikzcd}
\end{align}
We recall here for later use that there is an isomorphism of chain complexes $\Mor^{\mathcal{A}}(N_1,N_2)\cong\overline{N}_1\boxtimes\mathcal{A}\boxtimes N_2$, where $\overline{N}_1$ is the linear dual of $N_1$ and we regard $\mathcal{A}$ as an $A_\infty$-bimodule.
\subsection{Filtered type-$\mathit{DA}$ bimodules}
Let $\mathcal{A}$ and $\mathcal{B}$ be differential algebras over ground rings $\Bbbk$ and $\mathbbm{j}$. A type-$\mathit{DA}$ bimodule over $(\mathcal{A},\mathcal{B})$ is a graded $(\Bbbk,\mathbbm{j})$-bimodule $M$ equipped with degree 0 $(\Bbbk,\Bbbk')$-linear maps $\delta^1_{1+j}:M\otimes\mathcal{B}[1]^{\otimes j}\to\mathcal{A}[1]\otimes M$ which satisfy the following structure relation. Assemble the $\delta_{1+j}^1$ into a single map $\delta^1_M=\sum_j\delta_j^1$ and inductively define $\delta_M^i:M\otimes T(\mathcal{B}[1])\to\mathcal{A}[1]^{\otimes i}\otimes M$ by $\delta_M^0=\id_M$ and $\delta_M^{i+1}=(\id\otimes\delta^1)\circ(\delta^i\otimes\id)\circ(\id_M\otimes\Delta)$, where $\Delta:T(\mathcal{B}[1])\to T(\mathcal{B}[1])\otimes T(\mathcal{B}[1])$ is the canonical comultiplication map. Now let $\delta_M=\sum\limits_{i=0}^\infty\delta_M^i$ --- or, graphically,
\begin{align}
	\delta_M=\begin{tikzcd}[column sep=0.35cm,ampersand replacement=\&]
		{} \& {}\arrow[d,densely dashed] \& {}\arrow[dl,Rightarrow]\\
		\& \delta_M\arrow[d,densely dashed]\arrow[dl,Rightarrow] \&\\
		{} \& {} \& {}
	\end{tikzcd}=\begin{tikzcd}[column sep=0.35cm,ampersand replacement=\&]
	{}\arrow[dd,densely dashed]\\ \\ {}
\end{tikzcd}\oplus\begin{tikzcd}[column sep=0.35cm,ampersand replacement=\&]
{} \& {}\arrow[d,densely dashed] \& {}\arrow[dl,Rightarrow]\\
\& \delta_M\arrow[d,densely dashed]\arrow[dl] \&\\
{} \& {} \& {}
\end{tikzcd}\oplus\begin{tikzcd}[column sep=0.35cm,ampersand replacement=\&]
{} \& {} \& {}\arrow[dd,densely dashed] \& {}\arrow[d,Rightarrow]\\
{} \& {} \& {} \& \Delta\arrow[dl,Rightarrow]\arrow[ddl,Rightarrow]\\
\& \& \delta_M^1\arrow[d,densely dashed]\arrow[ddll] \&\\
\& \& \delta_M^1\arrow[d,densely dashed]\arrow[dl] \& \\
{} \& {} \& {} \& {}
\end{tikzcd}\oplus\cdots
\end{align}
--- then the structure relation for $\delta_M$ is given graphically by
\begin{align}
	\begin{tikzcd}[column sep=0.35cm,ampersand replacement=\&]
		{} \& {}\arrow[dd,densely dashed] \& {}\arrow[d,Rightarrow]\\
		{} \& {} \& \overline{D}^{\mathcal{B}}\arrow[dl,Rightarrow]\\
		{} \& \delta_M\arrow[d,densely dashed]\arrow[dl,Rightarrow] \& {}\\
		{} \& {} \& {}
	\end{tikzcd}+\begin{tikzcd}[column sep=0.35cm,ampersand replacement=\&]
	{} \& {}\arrow[d,densely dashed] \& {}\arrow[dl,Rightarrow]\\
	{} \& \delta_M\arrow[dl,Rightarrow]\arrow[dd,densely dashed] \& {}\\
	\overline{D}^{\mathcal{A}}\arrow[d,Rightarrow] \& {} \& {}\\
	{} \& {} \& {}
\end{tikzcd}=0.
\end{align}
If $M$ is a type-$\mathit{DA}$ bimodule over $(\mathcal{A},\mathcal{B})$ and $N$ is a left type-$D$ (resp. type-$\mathit{DA}$ bimodule) structure over $\mathcal{B}$ (resp. $(\mathcal{B},\mathcal{C})$), then $M\boxtimes N$ is naturally a left type-$D$ structure over $\mathcal{A}$ with structure map
\begin{align}
	\delta^1_{M\boxtimes N}=\begin{tikzcd}[column sep=0.35cm,ampersand replacement=\&]
		{} \& {}\arrow[dd,densely dashed] \& {}\arrow[d,densely dashed]\\
		{} \& {} \& \delta_N\arrow[dl,Rightarrow]\arrow[dd,densely dashed]\\
		{} \& \delta_M^1\arrow[d,densely dashed]\arrow[dl] \& {}\\
		{} \& {} \& {}
	\end{tikzcd}
\end{align}
(resp. the modification of this diagram in which $\delta_N$ has an input arrow $\Leftarrow$ from $T(\mathcal{C}[1])$).
\begin{definition}
	Let $S$ be a finite partially ordered set. An $S$-filtered type-$\mathit{DA}$ bimodule over $(\mathcal{A},\mathcal{B})$ is a collection $\{M^s\}_{s\in S}$ of type-$\mathit{DA}$ bimodules together with a collection of morphisms $\{F^{s<t}:M^s\to M^t[1]\}_{s,t\in S|s<t}$ satisfying
	\begin{align}
		dF^{s<u}=\sum_{s<t<u}F^{t<u}\circ F^{s<t}
	\end{align}
	whenever $s<u$, where $d$ is the morphism spaces differential. Given an $S$-filtered type-$\mathit{DA}$ bimodule $\{M^s,F^{s<t}\}$, its \emph{total bimodule} is the type-$\mathit{DA}$ bimodule $(M,\delta^1_M)$ with $M=\bigoplus_{s\in S}M^s$ and $\delta^1_M$ given by
	\begin{align}
		\delta_M^1|_{M^s\otimes T(\mathcal{A}[1])}=\delta^1_s+\sum_{s<t}F^{s<t}.
	\end{align}
\end{definition}
For example, if $f:M^0\to M^1$ is a bimodule morphism, then the mapping cone
\begin{align}
	\mathrm{Cone}(f)=\left(M^0[1]\oplus M^1,\begin{pmatrix}
		\delta_0^1 & 0\\
		f & \delta_1^1
	\end{pmatrix}\right)
\end{align}
is the total bimodule of the $\{0,1\}$-filtered type-$\mathit{DA}$ bimodule $\{M^0[1],M^1,f^{0<1}=f\}$. We refer the reader to \cite{LOTSpectral1} for further details on filtered $A_\infty$-, type-$\mathit{D}$ and type-$\mathit{DA}$ structures. 
\subsection{Khovanov's arc algebras and bimodules}
In \cite{KhovanovCategorificationJones2000}, Khovanov associates to any oriented link diagram $\mathbb{L}$ a bigraded chain complex $\CKh(\mathbb{L})$ whose homology $\Kh(L)$ is an isotopy invariant of the link $L$ in $S^3$ represented by $\mathbb{L}$ and which is functorial with respect to link cobordisms in $S^3\times I$, and which recovers the Jones polynomial as its graded Euler characteristic
\begin{align}
	J(L)=\chi_q(\Kh(L))=\sum_{i,j\in\Z}(-1)^iq^j\dim\Kh^{i,j}(L),
\end{align}
where $J$ is normalized so that $J(\bigcirc)=q+q^{-1}$. This chain complex may be computed by first numbering the each crossing
\begin{center}
	\includegraphics[scale=0.66666]{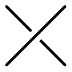}
\end{center}
in the diagram and then replacing it with the formal bigraded complex
\begin{align}
	h^{-n_-}q^{n_+-2n_-}\left(\raisebox{-0.33333cm}{\includegraphics[scale=0.66666]{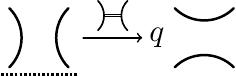}}\,\right)
\end{align}
where the ``differential'' is a saddle cobordism, the underlined term sits in homological grading 0, and a power of $h$ denotes a shift in the homological grading by that power and a power of $q$ represents a shift in the \emph{quantum} (or internal) grading. Here, $n_\pm$ is either 1 or 0 depending on whether the crossing is positive or negative. The first term in this complex is called the 0-resolution of the crossing and the second term is called the 1-resolution. One thus obtains a cubical diagram modeled on $\bm{2}^c$, where $\bm{2}$ is the poset $0\to 1$ and $c$ is the number of crossings in the diagram $\mathbb{L}$, whose vertices are formally $q$-graded planar configurations of circles labeled with a vertex of $\bm{2}^c$ and whose edges are cobordisms consisting of either a single merger of two circles or a split of one circle into two. One then applies to this cube the topological quantum field theory associated to the commutative Frobenius algebra $V=\F[x]/(x^2)$ with comultiplication defined by $1\mapsto 1\otimes x+x\otimes 1$ and $x\mapsto x\otimes x$, with quantum gradings $\gr_q(1)=1$ and $\gr_q(x)=-1$, and then takes direct sums along like homological gradings to obtain $\CKh(\mathbb{L})$.
\begin{remark}
	We follow Khovanov's original bigrading convention, rather than that used in \cite{KhovanovFunctorTangle2002}, in order to remain consistent with \cite{Rozansky}, \cite{Willis}, and \cite{MMSW}.
\end{remark}
\begin{definition}
	Given $k\geq 1$, let $[k]=\{1,\dots,k\}$. A \emph{matching} on $2n$ points is a 2-to-1 function $a:[2n]\to[n]$. A matching $a$ is said to be \emph{crossingless} if $i<j<k<\ell$ and $a(i)=a(k)$ implies that $a(j)\neq a(\ell)$.
\end{definition}
The condition of being crossingless is best understood graphically: we may represent a matching $a$ on $2n$ points diagrammatically by the unique collection of semicircular arcs in $[0,\infty)\times\R$ with ends on $\{0\}\times[2n]$ with the property that two points $(0,i),(0,j)\in\{0\}\times[2n]$ are the endpoints of a single arc if and only if $a(i)=a(j)$. A crossing $a$ is then crossingless if and only if the intersection of any two arcs in the corresponding diagram is empty (see Figure \ref{fig:CrossinglessMatchingsOn4Points} for the case $n=2$).
\begin{figure}
	\begin{center}
		\includegraphics[scale=0.5]{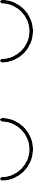}\hspace{1.5cm}\includegraphics[scale=0.5]{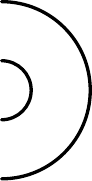}
	\end{center}
	\caption{The two crossingless matchings on 4 points are uniquely determined by the above diagrams.}
	\label{fig:CrossinglessMatchingsOn4Points}
\end{figure}

Let $\mathfrak{C}_n$ be the set of crossingless matchings on $2n$ points. Given $a,b\in\mathfrak{C}_n$, let $a^!b$ be the configuration of circles in the plane obtained by reflecting the semicircular arc diagram for $a$ across $\{0\}\times\R$ to obtain a diagram in $(-\infty,0]\times\R$ and then gluing this to the diagram for $b$ along $\{0\}\times\R$. The \emph{arc algebra on $2n$ points} is the associative algebra $H_n$ given as a vector space by
\begin{align}
	H_n=q^{-n}\bigoplus_{a,b\in\mathfrak{C}_n}\CKh(a^!b),
\end{align}
i.e. $H_n$ is spanned by planar configurations of circles of the form $a^!b$ in which each circle has been equipped with a label by either $1$ or $x$. Multiplication on $H_n$ is determined by the maps $\CKh(a^!b)\otimes_\F\CKh(b^!c)\to\CKh(a^!c)$ given by the cobordism $a^!b\sqcup b^!c\to a^!c$ given by $a^!\times I$ and $c\times I$ on $a$ and $c$, respectively, and on $b\sqcup b^!$ by the \emph{minimal saddle cobordisms} $b\sqcup b^!\to\id_{2n}$, where $\id_{2n}$ is the $2n$-stranded identity braid, consisting of a single 2-dimensional 1-handle addition merging each arc in $b$ with the corresponding arc in $b^!$ (see Figure \ref{fig:saddle} for an example).

\begin{figure}
	\begin{center}
		$b=\,\,\raisebox{-0.675cm}{\includegraphics[scale=0.5]{./figs/C2/C2_2.pdf}}$\hspace{1.5cm}$\raisebox{-0.85cm}{\includegraphics[scale=0.5]{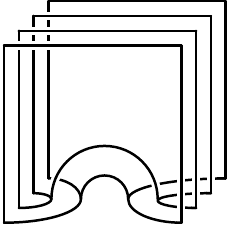}}\,\,:b\sqcup b^!\to\id_4$
	\end{center}
	\caption{A crossingless matching $b$ and the corresponding minimal saddle cobordism.}
	\label{fig:saddle}
\end{figure}

Now, given an $(m,n)$-tangle diagram $\mathbb{T}$ --- a tangle diagram with $2m$ left-endpoints and $2n$ right-endpoints --- Khovanov defines a complex of $(H_m,H_n)$-bimodules $\CKh(\mathbb{T})$ by
\begin{align}
	\CKh(\mathbb{T})=q^{-n}\bigoplus_{a\in\mathfrak{C}_m,b\in\mathfrak{C}_n}\CKh(a^!\mathbb{T}b),
\end{align}
where $a^!\mathbb{T}b$ is the link diagram obtained from $\mathbb{T}$ by gluing the mirrored diagram for $a$ on the left and the diagram for $b$ on the right. The left- and right-actions are given, as before, by maps induced by minimal saddle cobordisms. The homotopy type of this complex is an isotopy invariant of the tangle $T$ in $I\times D^2$ represented by $\mathbb{T}$ --- we refer the reader to \cite{KhovanovFunctorTangle2002} for further details.

In \cite{Rozansky}, Rozansky shows that, if $\mathbb{T}$ is a $(2n,2n)$-tangle diagram representing a tangle $T\subset I\times D^2$, then the (negatively graded) Hochschild homology $\HH_*(\CKh(\mathbb{T}))$ of $H_n$ with coefficients in the bimodule $\CKh(\mathbb{T})$ is an invariant of the link $L\subset S^1\times S^2$ given by taking the closure of $T$ to obtain a link in $S^1\times D^2$ and then regarding $D^2$ as a disk in $S^2$ to allow us to view this closure in $S^1\times S^2$. Further, he shows that if $\bm{\Phi}_n$ is the full-twist braid on $2n$ strands, then the complexes $\CKh(\bm{\Phi}_n^k)$ stabilize to give a projective resolution of $H_n$ as an $H_n\otimes_\F H_n^{\mathrm{op}}$-module as $k\to\infty$, so this invariant can be computed as the homology of a stable colimit of Khovanov complexes. In particular, if $\mathbb{L}_k$ is the link diagram obtained by taking the planar closure of $\bm{\Phi}_n^k\mathbb{T}$, then, for any homological lower bound $a$, there is some $k_a$ such that $\HH_*(\CKh(\mathbb{T}))\cong\Kh^*(\mathbb{L}_k)$ in any homological grading $*\geq a$ --- we refer the reader to \cite{Willis} for more details.
	\section{Endomorphism Algebras of Type-$D$ structures}
	Let $\bm{\nu}=\{(Y_i,\phi_i)\}_{i=1,\dots,n}$ be a finite collection of bordered 3-manifolds, all of which have boundary parametrized by the same pointed matched circle $\mathcal{Z}$. Define a differential algebra $\mathfrak{h}_{\bm{\nu}}$ over $\F$ by
\begin{align}
	\mathfrak{h}_{\bm{\nu}}=\End^{\mathcal{A}(-\mathcal{Z})}\left(\,\bigoplus_{i=1}^n\widehat{\CFD}(Y_i)\right)
\end{align}
with the usual morphism spaces differential and multiplication given by $\circ_2$.
\begin{remark}
	Only the subalgebra $\mathcal{A}(-\mathcal{Z},0)$ of $\mathcal{A}(-\mathcal{Z})$ acts nontrivially on the type-$D$ structures $\widehat{\CFD}(Y_i)$ (see \cite[Section 6.1]{LOTBHFH}) so one may freely replace $\mathcal{A}(-\mathcal{Z})$ with $\mathcal{A}(-\mathcal{Z},0)$ when computing $\mathfrak{h}_{\bm{\nu}}$.
\end{remark}

It will be convenient for us to regard $\mathcal{A}(-\mathcal{Z})$ as an $A_\infty$-algebra and the endomorphism complex $\End^{\mathcal{A}}(N)$ of a type-$D$ structure $N$ over an $A_\infty$-algebra $\mathcal{A}$ as the space of as the space of $\Bbbk$-linear maps $f:N\to\mathcal{A}\otimes_\Bbbk N$. This will allow us to make full use of the graphical notation of \cite{LOTBHFH,LOTBimodules2015,LOTMorphism}. In this notation, the differential of $f\in\mathfrak{h}_{\bm{\nu}}$ is given by
\begin{align}\label{line:morphismdifferential}
	\partial f:=\circ_1(f)=\begin{tikzcd}[ampersand replacement=\&,column sep=0.35cm,row sep=0.35cm]
		\& \arrow[d,dashed] \\
		\& \delta\arrow[d,dashed]\arrow[ddl,Rightarrow,bend right=24]\\
		\& f\arrow[d,dashed]\arrow[dl] \\
		\mu\arrow[d] \& \delta\arrow[d,dashed]\arrow[l,Rightarrow]\\
		{} \& {}
	\end{tikzcd}
\end{align}
and multiplication is the opposite composition map $\circ_2$, given graphically by 
\begin{align}
	fg=\circ_2(f,g)=\begin{tikzcd}[ampersand replacement=\&,column sep=0.35cm,row sep=0.35cm]
		\& \arrow[d,dashed]\\
		\& \delta\arrow[d,dashed]\arrow[dl,Rightarrow,bend right=42]\\
		\otimes\arrow[d,Rightarrow] \& f\arrow[d,dashed]\arrow[l] \\
		\otimes\arrow[d,Rightarrow] \& \delta\arrow[d,dashed]\arrow[l,Rightarrow]
		\\
		\otimes\arrow[d,Rightarrow] \& g\arrow[d,dashed]\arrow[l]\\
		\mu\arrow[d] \& \delta\arrow[d,dashed]\arrow[l,Rightarrow]\\
		{} \& {}
	\end{tikzcd}.
\end{align}
Since $\mathcal{A}(-\mathcal{Z})$ is a differential algebra, all of the higher composition maps
\begin{align}
	\circ_k(f_1,\dots,f_k)=\begin{tikzcd}[ampersand replacement=\&,column sep=0.35cm,row sep=0.35cm]
		\& \arrow[d,dashed]\\
		\& \delta\arrow[d,dashed]\arrow[dl,Rightarrow,bend right=42]\\
		\otimes\arrow[d,Rightarrow] \& f_1\arrow[d,dashed]\arrow[l]\\
		\otimes\arrow[d,Rightarrow] \& \delta\arrow[d,dashed]\arrow[l,Rightarrow]\\
		\otimes\arrow[d,Rightarrow] \& f_2\arrow[d,dashed]\arrow[l]\\
		\raisebox{0.2cm}{\vdots}\arrow[d,Rightarrow] \& \raisebox{0.2cm}{\vdots}\arrow[d,dashed]\\
		\otimes\arrow[d,Rightarrow] \& f_k\arrow[d,dashed]\arrow[l]\\
		\mu\arrow[d] \& \delta\arrow[d,dashed]\arrow[l,Rightarrow]\\
		{} \& {}
	\end{tikzcd}
\end{align}
necessarily vanish so, since the maps $\circ_k$ satisfy the $A_\infty$-relation, it follows that $\mathfrak{h}_{\bm{\nu}}$ is an associative algebra and it is straightforward to check that $\partial(fg)=(\partial f)g+f\partial g$. In other words $\circ_2$ and $\circ_1$ make $\mathfrak{h}_{\bm{\nu}}$ into a differential algebra. Note, additionally, that, in this context, $\circ_2$ reduces to the ordinary opposite composition map $f\otimes g\mapsto g\circ f$ on taking modulifications of our type-$D$ structures.

Note that, a priori, $\mathfrak{h}_{\bm{\nu}}$ depends on a choice of bordered Heegaard diagram for each $(Y_i,\phi_i)$. However, the following result tells us that we may safely ignore this dependence.
\begin{lemma}\label{QIalgLemma}
	Let $N$ and $N'$ be homotopy equivalent left type-$D$ structures over an $A_\infty$-algebra $\mathcal{A}$. Then there is a quasi-isomorphism of $A_\infty$-algebras
	\begin{align}
		\End^{\mathcal{A}}(N)\simeq\End^{\mathcal{A}}(N').
	\end{align}
\end{lemma}
\begin{proof}
	The category of left type-$D$ structures over $\mathcal{A}$ is canonically quasi-equivalent to the category of left $A_\infty$-modules over $\mathcal{A}$ via the $A_\infty$-functor $\mathcal{A}\boxtimes-$ by \cite[Proposition 2.3.18]{LOTBimodules2015} so we have quasi-isomorphisms $\End^{\mathcal{A}}(N)\simeq\End_{\mathcal{A}}(\mathcal{A}\boxtimes N)$ and $\End^{\mathcal{A}}(N')\simeq\End_{\mathcal{A}}(\mathcal{A}\boxtimes N')$. Since $N$ and $N'$ are homotopy equivalent, the same is true of $\mathcal{A}\boxtimes N$ and $\mathcal{A}\boxtimes N'$ and the category of left $A_\infty$-modules over $\mathcal{A}$ is a dg-category so it follows that $\End_{\mathcal{A}}(\mathcal{A}\boxtimes N)\simeq\End_{\mathcal{A}}(\mathcal{A}\boxtimes N')$ and this proves the desired result.
\end{proof}
\begin{corollary}\label{QICorollary}
	The algebra $\mathfrak{h}_{\bm{\nu}}$ depends only up to quasi-isomorphism on a choice of bordered Heegaard diagrams for the bordered 3-manifolds $(Y_i,\phi_i)$.
\end{corollary}
By \cite[Lemma 2.3.13]{LOTBimodules2015}, if $M$ is a type-$\mathit{DA}$ bimodule over $(\mathcal{A},\mathcal{B})$, then $M\boxtimes-$ is an $A_\infty$-functor. More explicitly, if, for $i=0,\dots,j$, $(N_j,\delta_j^1)$ are left type-$D$ structures over $\mathcal{B}$ and $f_i:N_{i-1}\to N_i$ are morphisms, then
\begin{align}
	(M\boxtimes-)_{1|j}(\id_M,f_1,\dots,f_j):=\begin{tikzcd}[ampersand replacement=\&,column sep=0.3cm,row sep=0.35cm]
		\& {}\arrow[d,densely dashed] \& {}\arrow[d,densely dashed] \\
		\& \otimes\arrow[d,Rightarrow] \& \delta_0\arrow[d,densely dashed]\arrow[l,Rightarrow]\\
		\& \otimes\arrow[d,Rightarrow] \& f_1\arrow[d,densely dashed]\arrow[l]\\
		\& \otimes\arrow[d,Rightarrow] \& \delta_1\arrow[d,densely dashed]\arrow[l,Rightarrow]\\
		\& \vdots\arrow[d,Rightarrow] \& \vdots\arrow[d,densely dashed]\\
		\& \otimes\arrow[d,Rightarrow] \& \delta_{j-1}\arrow[d,densely dashed]\arrow[l,Rightarrow]\\
		\& \otimes\arrow[d,Rightarrow] \& f_j\arrow[d,densely dashed]\arrow[l]\\
		\& \delta_M^1\arrow[d,densely dashed]\arrow[dl,bend right=33] \& \delta_j\arrow[d,densely dashed]\arrow[l,Rightarrow]\\
		{} \& {} \& {}
	\end{tikzcd}
\end{align}
makes $M\boxtimes-$ into an $A_\infty$-functor.
\begin{definition}
	Let $Y$ be a (strongly) bordered 3-manifold with left- and right-boundary parameterized by $F(\mathcal{Z})$ and $F(\mathcal{Z}')$, respectively. Let $\bm{\nu}=\{(Y_i,\phi_i)\}_{i=1,\dots,m}$ be a collection of bordered 3-manifolds with boundary parameterized by $F(\mathcal{Z})$ and $\bm{\nu}'=\{(Y_j',\phi_j')\}_{j=1,\dots,n}$ a collection with boundaries parameterized by $F(\mathcal{Z}')$, and let $\mathfrak{h}=\mathfrak{h}_{\bm{\nu}}$ and $\mathfrak{h}'=\mathfrak{h}_{\bm{\nu}'}$. Define an $A_\infty$-bimodule $\CF_{\bm{\nu},\bm{\nu}'}(Y)$ over $(\mathfrak{h},\mathfrak{h}')$ by
	\begin{align}
		\CF_{\bm{\nu},\bm{\nu}'}(Y)=\bigoplus_{i,j}\Mor^{\mathcal{A}(-\mathcal{Z})}(Y_i,Y\boxtimes Y_j'),
	\end{align}
	where, by abuse of notation, we define
	\begin{align}
		\Mor^{\mathcal{A}(-\mathcal{Z})}(Y_i,Y\boxtimes Y_j'):=\Mor^{\mathcal{A}(-\mathcal{Z})}(\widehat{\CFD}(Y_i),\widehat{\CFDA}(Y)\boxtimes\widehat{\CFD}(Y_j')).
	\end{align}
	The structure maps $m_{i|1|j}:\mathfrak{h}[1]^{\otimes i}\otimes_\F\CF_{\bm{\nu},\bm{\nu}'}(Y)\otimes_\F\mathfrak{h}'[1]^{\otimes j}\to\CF_{\bm{\nu},\bm{\nu}'}(Y)[1]$ are given by
	\begin{align}
		m_{i|1|j}(\phi_1,\dots,\phi_i,f,\psi_1,\dots,\psi_j)=\circ_{2+i-\delta(j,0)}(\psi_1,\dots,\psi_i,f,H_j(\psi_1,\dots,\psi_j))
	\end{align}
	where $H_j=(\widehat{\CFDA}(Y)\boxtimes-)_{1|j}(\id,-,\stackrel{j}{\dots},-)$ and
	\begin{align*}
		\delta(m,n)=\begin{cases}
			1 & \textup{if $m=n$}\\
			0 & \textup{else.}
		\end{cases}
	\end{align*}
	More concretely, we have that $m_{1|1|0}(\psi,f)=\circ_2(\psi,f)$, the left-action is strict, and strictly commutes with the right-$A_\infty$-action, since $\circ_{2+i}$ vanishes identically for $i>0$, and $m_{0|1|1}(f,\psi)=\circ_2(f,\id\boxtimes\psi)$. The right action is non-strict in general, but $m_{0|1|2}(f,\psi_1,\psi_2)=\circ_2(f,H_2(\psi_1,\psi_2))$, where $H_2(\psi_1,\psi_2)$ is a homotopy between $\circ_2(\id\boxtimes\psi_1,\id\boxtimes\psi_2)$ and $\id\boxtimes\circ_2(\psi_1,\psi_2)$, and the $m_{0|1|j}$ for $j>2$ form a coherent system of higher homotopies encoding the fact that the right-action is associative up to homotopy.
\end{definition}
\begin{proposition}
	The maps $m_{i|1|j}$ satisfy the $A_\infty$-bimodule structure relations.
\end{proposition}
\begin{proof}
	This follows immediately from the fact that the left action is strict and strictly commutes with the right action, and that $\widehat{\CFDA}(Y)\boxtimes-$ is an $A_\infty$-functor.
\end{proof}
\begin{remark}
	Note that $m_{0|1|j}=0$ for $j$ sufficiently large provided that $\widehat{\CFDA}(Y)$ is bounded.
\end{remark}
\begin{proposition}
	The quasi-isomorphism type of $\CF_{\bm{\nu},\bm{\nu}'}(Y)$ is independent of the choice of strongly bordered Heegaard diagram for $Y$.
\end{proposition}
\begin{proof}
	If $\mathcal{H}_1$ and $\mathcal{H}_2$ are strongly bordered Heegaard diagrams for $Y$, then they differ by a finite sequence of Heegaard moves which induces a homotopy equivalence $F:\widehat{\CFDA}(\mathcal{H}_1)\to\widehat{\CFDA}(\mathcal{H}_2)$. The map
	\begin{align}
		\circ_2(-,F\boxtimes\id):=\sum_j\circ_2(-,F\boxtimes\id_{Y'_j}):\CF_{\bm{\nu},\bm{\nu}'}(\mathcal{H}_1)\to\CF_{\bm{\nu},\bm{\nu}'}(\mathcal{H}_1)
	\end{align}
	is then automatically a quasi-isomorphism of complexes of vector spaces. Note that the left action of $\mathfrak{h}_{\bm{\nu}}$ commutes with $\circ_2(-,F\boxtimes\id)$ on the nose. We refer the reader to the proof of Theorem \ref{thm:InducedActions} for details regarding the homotopy commutation of the right-action of $\mathfrak{h}_{\bm{\nu}'}$ with $\circ_2(-,F\boxtimes\id)$, as the argument given there can be applied here.
\end{proof}
	\section{Branched Arc Algebras}
	\subsection{Branched double covers}
Given a link $L\subset S^3$, one may construct a 3-manifold $\Sigma(L)$, called the \emph{branched double cover of} $L$ as follows: choose a Seifert surface $F$ for $L$ and let $Y_L^0$ be the complement of a tubular open neighborhood of $F\cap(S^3\smallsetminus\mathrm{nbd}(L))$ in $S^3\smallsetminus\mathrm{nbd}(L)$, where $\mathrm{nbd}(L)$ is a tubular open neighborhood of $L$. The (cornered) 3-manifold $Y_L^0$ contains two copies of $F$, call them $F_-$ and $F_+$. Let $Y_L^1$ be the manifold with boundary obtained by taking the quotient of $Y_L^0\sqcup Y_L^0$ obtained by identifying $F_{\pm}$ in the first copy of $Y_L^0$ with $F_{\mp}$ in the second. Note that $Y_L^1$ has one toroidal boundary component for every component of $L$. The closed 3-manifold $\Sigma(L)$ is then obtained by Dehn filling each of these boundary components with respect to the Seifert framing induced by the copies of $F_\pm$ sitting inside of $Y_L^1$.
\begin{example}
	The branched double cover of an unlink with $k$ components is $\#^{k-1}(S^2\times S^1)$. More generally, given two links $L_0$ and $L_1$, there is a diffeomorphism $\Sigma(L_0\sqcup L_1)\cong\Sigma(L_0)\#\Sigma(L_1)\#(S^2\times S^1)$.
\end{example}
\begin{remark}
	A link cobordism $C:L_0\to L_1$ induces a cobordism of 3-manifolds $\Sigma(C):\Sigma(L_0)\to\Sigma(L_1)$, which we call the branched double cover of $C$.
\end{remark}
Note that one may extend this definition to obtain branched double covers $\Sigma(T)$ of tangles $T$ in the 3-ball, or in $S^2\times[0,1]$, which are 3-manifolds with boundary. For simplicity, we will restrict ourselves to the case of tangles with an even number of endpoints on the equator(s) of the boundary of their ambient 3-manifold. A \emph{cornered Seifert surface} for such a tangle $T$ is an orientable surface $F\subset Y$ with corners, where $Y$ is either of $B^3$ or $S^2\times[0,1]$, such that $\partial F$ decomposes as the union of $T$ and a collection of arcs in the equator(s) of $Y$. Such a surface always exists: $T$ has an even number of endpoints on each boundary component of $Y$ so the plat closure $p(T)$ of $T$ embeds in $Y$, smoothly away from the endpoints of $T$. We may then apply Seifert's algorithm to any oriented diagram for $p(T)$ obtained by taking the plat closure of a diagram for $T$, using arcs in the projections of the equators for the closure, and regarding the resulting cornered surface as an embedded surface $F$ in $Y$ (see Figure \ref{fig:Seifert} for an example).
\begin{figure}
	\begin{center}
		\includegraphics[scale=1]{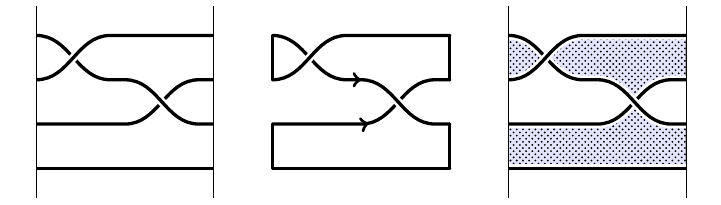}
	\end{center}
	\caption[Construction of a cornered Seifert surface for a tangle in $S^2\times{[}0,1{]}$.]{A diagram for a tangle $T\subset S^2\times[0,1]$ (left), its plat closure $p(T)$ by equatorial arcs (middle), and the cornered Seifert surface obtained from applying Seifert's algorithm to $p(T)$ (right). Here, the vertical lines in the left- and right-hand figures represent the projections of the equators of $S^2\times[0,1]$.}
	\label{fig:Seifert}
\end{figure}
To construct $\Sigma(T)$, we take $Y_T^0$ to be the complement of a tubular open neighborhood of $F\cap(Y\smallsetminus\mathrm{nbd}(T))$ and glue two copies of this space, as we did with $Y_L^0$ above, to obtain a cornered 3-manifold $Y_T^1$ whose codimension 1 stratum decomposes as $\partial_1Y_T^1=\overline{\Sigma}\cup_{\partial}\partial\mathrm{nbd}(T)$, where $\overline{\Sigma}$ is a (possibly disconnected) surface with $\#\partial T$ boundary components. We then fill $Y_T^1$ with $\mathrm{nbd}(T)$ to obtain $\Sigma(T)$. If $T\subset B^3$ has $2n$ endpoints, then $\partial\Sigma(T)$ is an oriented surface of genus $n-1$. Similarly, if $T\subset S^2\times[0,1]$ has $\# T\cap(S^2\times\{0\})=2m$ and $\# T\cap(S^2\cap\{1\})=2n$, then the boundary components of $\Sigma(T)$ have genus $m-1$ and $n-1$. One can see this by considering the branched double cover of the $2n$-stranded identity braid $\id_{2n}$ in $S^2\times[0,1]$, which we may think of as a collar neighborhood of $\partial\Sigma(T)$. This 3-manifold is the product of an interval and the double cover $\Sigma_g$ of $S^2$ branched along $2n$ points. Since the ramification index of each branch point is $2$, the Riemann--Hurwitz formula tells us that $\chi(\Sigma_g)=2\chi(S^2)-2n=2-2(n-1)$ so $g=n-1$ and $\Sigma(\id_{2n})\cong\Sigma_{n-1}\times[0,1]$.
\subsection{The algebras}\label{Algebras}
In \cite{OzsSzBranched}, Ozsv\'{a}th--Szab\'{o} showed that, for any (based) link $L\subset S^3$, there is a spectral sequence $\Khred(mL;\F)\Rightarrow\widehat{\HF}(\Sigma(L))$. They prove this result by constructing a filtration on $\widehat{\CF}(\Sigma(L))$, associated to a diagram for $L$, such that the $E^1$-page of the induced spectral sequence is
\begin{align}
	\bigoplus_{\bm{v}\in\bm{2}^c}\widehat{\HF}(\Sigma(L_{\bm{v}})),
\end{align}
where $c$ is the number of crossings in the diagram, $\bm{2}=\{0,1\}$, and $L_{\bm{v}}$ is the complete resolution of the diagram determined by $\bm{v}$ and an ordering of the crossings. Since each $L_{\bm{v}}$ is a planar unlink, each summand is of the form $\widehat{\HF}(\#^{k-1}(S^2\times S^1))$, where $k$ is the number of components of $L_{\bm{v}}$, which they show is isomorphic to $\CKhred(L_{\bm{v}})$. They then identify the $d^1$-differential, which is given by the maps on Heegaard Floer homology induced by the branched double covers of the saddle cobordisms making up the edges of the cube of resolutions, with the Khovanov differential. In the case that $L$ is a planar unlink, the spectral sequence degenerates on the $E^1$-page, so one should expect there to be a Heegaard Floer analogue of the arc algebra $H_n$. Na\"{i}vely, this algebra might take the form
\begin{align}
	\bigoplus_{a,b\in\mathfrak{C}_n}\widehat{\HF}(\Sigma(a^!b))
\end{align}
with multiplication given by the maps induced by branched double covers of minimal saddle cobordisms. However there are some issues with this construction. First, the arc algebra $H_n$ and its reduced version $\widetilde{H}_n$ have somewhat different properties as algebras --- for example, $\HH_*(H_1)$ is infinite-dimensional while $\widetilde{H}_1\cong\F$ so $\HH_*(H_1)\cong\F$ --- though this difference is only up to a tensor factor of the algebra $V$ (see \cite{CohenSplitting}). Second, and more seriously, it is not immediately clear that this construction yields an algebra, or even a generalized algebra, in a sensible way. We will instead define a chain-level version of this algebra and show that its homology is, in general, a nontrivial $A_\infty$-deformation of $H_n$.
\begin{definition}
	The \emph{genus} $k$ \emph{linear pointed matched circle} $\mathcal{Z}_k$ is the pointed matched circle whose matching $M$ matches the pairs $\{a_1,a_3\}$ and $\{a_{4k-2},a_{4k}\}$ and, for each $n=1,\dots,2k-2$, the pairs $\{a_{2n},a_{2n+3}\}$ (see Figure \ref{fig:LinearPMC2}). Note that $\mathcal{Z}_1$ is the usual pointed matched circle for the torus.
\end{definition}
\begin{figure}
	\begin{center}
		\includegraphics[scale=1]{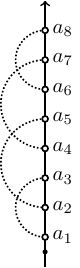}
	\end{center}
	\caption{The genus 2 linear pointed matched circle.}
	\label{fig:LinearPMC2}
\end{figure}
One may naturally view the branched double cover $\Sigma(T)$ of a tangle $T$ in $B^3$ with $2n$ equatorial endpoints as having boundary parametrized by $\mathcal{Z}_{n-1}$ by using the algorithm given in \cite[Section 6.1]{LOTSpectral2} to construct an explicit bordered Heegaard diagram for $\Sigma(T)$. We recall this construction here for crossingless matchings, starting with a diagram $\mathcal{H}$ for the branched double cover of the \emph{plat closure} on $2n$ points, i.e. the matching consisting of $n$ caps stacked vertically. We illustrate the $n=3$ case in Figure \ref{fig:Plat}. First, draw a vertical line segment with a distinguished basepoint near its bottom end and, temporarily denoting the plat closure by $a$, identify $\partial a$ with $[2n]$ by enumerating the endpoints from bottom to top. Step 1: to the right of this line draw $4n-4$ horizontal line segments which each meet it at a single point, two corresponding to each of the endpoints $2$ through $2n-2$ in $\partial a$ and one each corresponding to $1$ and $2n-1$ --- note that there is no line corresponding to the endpoint labeled $2n$ --- and enumerate these from bottom to top. Step 2: draw pairs of labeled circles representing handles at the other ends of the pairs of segments labeled $4k+2$ and $4k+5$ for $k=1,2,\dots,n-2$ and one more pair for the segments labeled $4n-6$ and $4n-4$. Step 3: draw half-circular arcs to the right of the circles added in Step 2 which connect the endpoints of the segments labeled 1 and 3 and the pairs of segments labeled $4k$ and $4k+3$ for $k=1,2,\dots,n-2$. Steps 2 and 3 completely specify the $\alpha$-curves in $\mathcal{H}$. Step 4: draw a $\beta$-circle enclosing all of the circles contained in each region of the diagram bounded by an $\alpha$-arc constructed in Step 3. The result is then a bordered Heegaard diagram for $\Sigma(a)$.
\begin{figure}
	\begin{center}
		\includegraphics[scale=1]{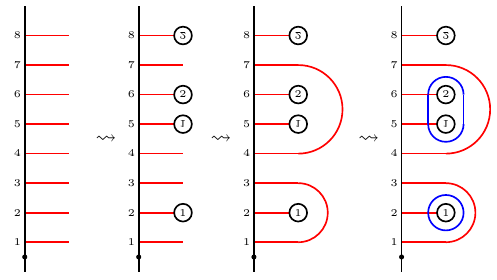}
	\end{center}
	\caption[Construction of a bordered Heegaard diagram for the 6-ended plat closure.]{Construction of a bordered Heegaard diagram for the 6-ended plat closure. Here, steps 1 through 4 are illustrated from left to right.}
	\label{fig:Plat}
\end{figure}

If $b\in\mathfrak{C}_n$ is any other crossingless matching, we may isotope the diagram for $b$ so that it becomes the plat closure (on the right) by $a$ of a product of cap-cup tangles (see Figure \ref{fig:Plat2} for an example) which is minimal in the sense that there is no such presentation of $b$ with fewer caps and cups.
\begin{figure}
	\begin{center}
		\includegraphics[scale=0.8]{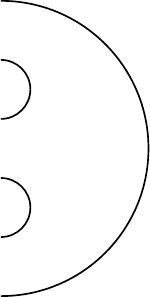}\hspace{2cm}\includegraphics[scale=0.8]{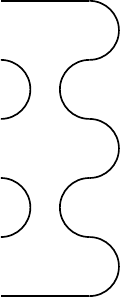}
	\end{center}
	\caption[A crossingless matching and its minimal plat closure-form.]{A crossingless matching on 6 points (left) and its minimal plat closure-form (right).}
	\label{fig:Plat2}
\end{figure}
Note that, by minimality, no cap-cup pair will involve the bottom-most or top-most strands of this diagram for $b$. For each cap-cup pair, we insert a new handle and $\beta$-circle of the form
\begin{center}
	\includegraphics{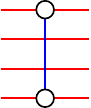}
\end{center}
into the bordered Heegaard diagram for the plat closure, where the four $\alpha$-curves are the arcs corresponding to the strands in which the cap-cup pair occurs, provided these strands are not the ones at heights $2n-2$ and $2n-1$. In the latter case, we instead insert
\begin{center}
	\includegraphics{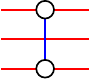}
\end{center}
to modify the plat closure diagram. Inserting these handles and $\beta$-circles will always result in a diagram with some number of configurations of handles, $\alpha$-curves, and $\beta$-circles of the form
\begin{center}
	\includegraphics{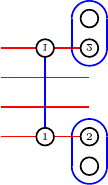}
\end{center}
where the two $\beta$-circles at right come from the original bordered Heegaard diagram for the plat closure. We may then perform a sequence of isotopies and handleslides
\begin{center}\label{CapCupHandle}
	\includegraphics{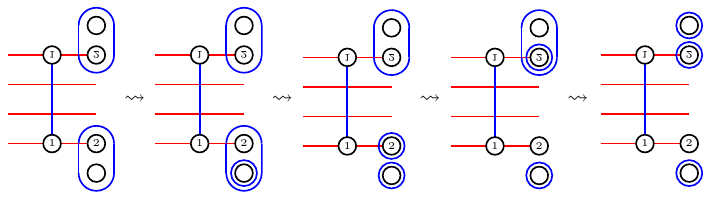}
\end{center}
of the $\beta$-circles coming from the plat closure diagram, starting by isotoping the bottom-most circle --- the one in the region adjacent to the boundary Reeb chord $[1,3]$ --- over the handle it encircles, until any such configuration in the diagram has been changed to be as at right. In the above schematic, the first step is an isotopy of a $\beta$-circle (which is not pictured in the first diagram) over a handle. After this sequence of Heegaard moves, each such resulting configuration contains a connected sum with a standard diagram for $S^3$ --- here given by the handle corresponding to the two circles labeled 2, the $\beta$-circle enclosing the topmost of these circles, and the $\alpha$-circle given by the two red line segments between the circles labeled 1 and the circles labeled 2. We then destabilize the diagram until all of these standard diagrams are removed to obtain the bordered Heegaard diagram $\mathcal{H}_b$ for $\Sigma(b)$ (see Figure \ref{fig:Plat2Heegaard} for an example).
\begin{figure}
	\begin{center}
		\includegraphics[scale=0.9]{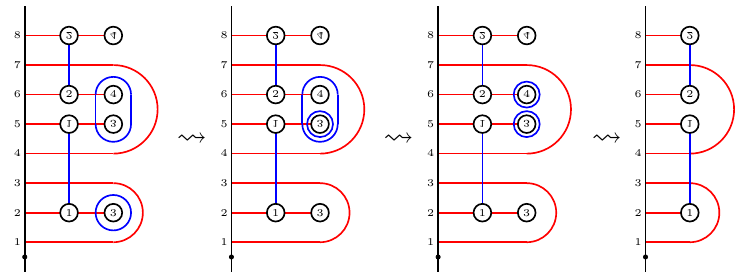}
	\end{center}
	\caption[The bordered Heegaard diagram for the matching in Figure \ref{fig:Plat2}]{The bordered Heegaard diagram for the matching in Figure \ref{fig:Plat2}. Its destabilized form, at right, is obtained by performing an isotopy and a handleslide, shown in the two intermediate steps, followed by two destabilizations, which comprise the last step shown.}
	\label{fig:Plat2Heegaard}
\end{figure}
\begin{definition}
	Given $c\in\mathfrak{C}_n$, let $c_+\in\mathfrak{C}_{n+1}$ be the crossingless matching obtained by adding a single extra arc below $c$. Let $\mathfrak{B}_n=\{c_+|c\in\mathfrak{C}_n\}$. Regarding $c_+$ as a tangle in $B^3$ with endpoints on the equator, define $\widehat{\CFD}(c_+)=\widehat{\CFD}(\mathcal{H}_{c})$, where $\mathcal{H}_c$ is the bordered Heegaard diagram for $\Sigma(c_+)$ constructed as above. The \emph{branched arc algebra} $\mathfrak{h}_n$ on $2n$ points is then the differential algebra
	\begin{align}
		\mathfrak{h}_n=\End^{\mathcal{A}_n}\left(\,\bigoplus_{c\in\mathfrak{C}_n}\widehat{\CFD}(c_+)\right)=\End^{\mathcal{A}_n}\left(\,\bigoplus_{a\in\mathfrak{B}_n}\widehat{\CFD}(a)\right),
	\end{align}
	where $\mathcal{A}_n=\mathcal{A}(-\mathcal{Z}_n,0)$, with algebra operation given by the composition map $\circ_{2}:f\otimes g\mapsto g\circ f$ and the usual morphism space differential. In other words, $\mathfrak{h}_n=\mathfrak{h}_{\bm{\nu}_n}$, where $\bm{\nu}_n=\{\Sigma(a)|a\in\mathfrak{B}_n\}$.
\end{definition}
\begin{remark}
	We will typically avoid writing $c_+$ for $c\in\mathfrak{C}_n$ by summing over the set $\mathfrak{B}_n$, instead of $\mathfrak{C}_n$, in forthcoming sections, but it will be convenient to use that notation here.
\end{remark}
We will show that the algebras $H_*\mathfrak{h}_n$ and $H_n$ agree. We first recall the following propositions from \cite{OzsSzBranched}.
\begin{proposition}[{\cite[Proposition 6.1]{OzsSzBranched}}]
	If $Y\cong\#^k(S^2\times S^1)$, then $\widehat{\HF}(Y)$ is a rank 1 free module over $\Lambda^*H_1(Y)$, generated by the class $\Theta^{\mathrm{top}}\in\widehat{\HF}(Y)$. Moreover, if $K\subset Y$ is a curve representing an $S^1$ fiber in one of the $S^2\times S^1$ summands, then the 3-manifold $Y'=Y_0(K)$ is diffeomorphic to $\#^{k-1}(S^2\times S^1)$, with a natural identification $\pi:H_1(Y)/[K]\to H_1(Y')$. Under the $2$-handle cobordism $W_1:Y\to Y'$, the map $\widehat{F}_{W_1}:\widehat{\HF}(Y)\to\widehat{\HF}(Y')$ is determined by
	\begin{align}
		\widehat{F}_{W_1}(\xi\cdot\Theta^{\mathrm{top}})=\pi(\xi)\cdot\Theta^{\mathrm{top}'},
	\end{align}
	where $\Theta^{\mathrm{top}'}\in\widehat{\HF}(Y')$ is the generator of $\widehat{\HF}(Y')$ as a free $\Lambda^*H_1(Y')$-module and $\xi\in\Lambda^*H_1(Y)$. Dually, if $K\subset Y$ is a local unknot, then the manifold $Y''(K)=Y_0(K)$ is diffeomorphic to $\#^{k+1}(S^2\times S^1)$, and there is a natural inclusion $i:H_1(Y)\to H_1(Y'')$. The map $\widehat{F}_{W_2}:\widehat{\HF}(Y)\to\widehat{\HF}(Y'')$ induced by the 2-handle cobordism $W_2:Y\to Y''$ is then determined by
	\begin{align}
		\widehat{F}_{W_2}(\xi\cdot\Theta^{\mathrm{top}})=i(\xi)\wedge[K'']\cdot\Theta^{\mathrm{top}''},
	\end{align}
	where $[K'']\in H_1(Y'')$ is a generator of $\ker(H_1(Y'')\to H_1(W_2))$.
\end{proposition}
In the case that $Y$ is given as the branched double cover $\Sigma(\mathcal{D})=\#^k(S^2\times S^1)$ of a planar unlink $\mathcal{D}=S_0\cup\cdots\cup S_k$, where $S_0$ is a distinguished component with a basepoint, this proposition furnishes us with the following variation of \cite[Proposition 6.2]{OzsSzBranched}.
\begin{proposition}[{\cite[Proposition 6.2]{OzsSzBranched}}]\label{OzsSzProp6-2}
	If $\mathcal{D}$ is a planar unlink with one based component, then there is an isomorphism $\psi_{\mathcal{D}}:\CKhred(\mathcal{D})\stackrel{\cong}{\longrightarrow}\widehat{\HF}(\Sigma(\mathcal{D}))$ which is natural under cobordisms in the sense that if $s:\mathcal{D}\to\mathcal{D}'$ is either a single merge or split cobordism, then the diagram
	\begin{align}
		\begin{tikzcd}[ampersand replacement=\&]
			\CKhred(\mathcal{D})\arrow[r,"\CKhred(s)"]\arrow[d,"\psi_{\mathcal{D}}"'] \& \CKhred(\mathcal{D}')\arrow[d,"\psi_{\mathcal{D}'}"]\\
			\widehat{\HF}(\Sigma(\mathcal{D}))\arrow[r,"\widehat{F}_{\Sigma(s)}"] \& \widehat{\HF}(\Sigma(\mathcal{D}'))
		\end{tikzcd}
	\end{align}
	commutes.
\end{proposition}
We recall the proof of this statement in the case that $s$ does not involve the marked component since we will only require this case in our proof that the algebras agree.
\begin{proof}
	For $i>0$, let $\gamma_i$ be an arc in $S^3$ from $S_0$ to $S_i$ which is disjoint from $\mathcal{D}$ away from its endpoints and let $\tilde{\gamma}_i$ be the preimage of $\gamma_i$ in $\Sigma(\mathcal{D})$. Note that the preimages of any two choices of $\gamma_i$ are homologous in $\Sigma(\mathcal{D})$. Then, by construction, $\{[\tilde{\gamma}_i]\}_{i=1}^k$ is a basis for $H_1(\Sigma(\mathcal{D}))$. Using \cite[Proposition 6.1]{OzsSzBranched} and the identification, given in \cite[Section 5]{OzsSzBranched}, of $\CKhred(\mathcal{D})$ with the exterior algebra $\Lambda^*\widetilde{Z}(\mathcal{D})$, where $\widetilde{Z}(\mathcal{D})$ is the vector space formally spanned by the unmarked components $[S_1],\dots,[S_k]$ of $\mathcal{D}$, the map $\psi_{\mathcal{D}}$ is then given by the isomorphism $\Lambda^*\widetilde{Z}(\mathcal{D})\stackrel{\cong}{\longrightarrow}\Lambda^*H_1(\Sigma(\mathcal{D}))$ determined by $[S_i]\mapsto[\tilde{\gamma}_i]$.
	
	If $s$ merges two circles $S_1$ and $S_2$ into a single circle $T$, then, in the cobordism $\Sigma(s)$, the curves $\tilde{\gamma}_1$ and $\tilde{\gamma}_2$ become homologous to the lift of the curve from $S_0$ to $T$ in $\Sigma(\mathcal{D}')$. Commutativity of the above square then follows from \cite[Proposition 6.1]{OzsSzBranched} and the definition of $\CKhred(s)$. Dually, if $s$ splits a circle $T$ into a disjoint union $S_1\sqcup S_2$ of two circles, then the curve $\tilde{\delta}=\tilde{\gamma}_2-\tilde{\gamma}_1$ is nullhomologous in $\Sigma(s)$ and commutativity of the square follows similarly.
	\begin{figure}
		\begin{center}
			\includegraphics[scale=1]{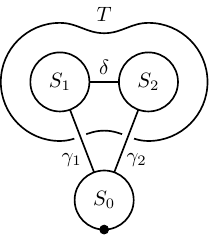}
		\end{center}
		\caption{After merging $S_1$ and $S_2$, the curves $\tilde{\gamma}_1$ and $\tilde{\gamma}_2$ become homologous. Dually, if $T$ is split into $S_1\sqcup S_2$, the curve $\tilde{\delta}=\tilde{\gamma}_2-\tilde{\gamma}_1$ becomes nullhomologous.}
	\end{figure}
\end{proof}
Note that if $\mathcal{D}_0$ and $\mathcal{D}'_0$ are two planar unlinks, $\mathcal{D}$ and $\mathcal{D}'$ are the based unlink diagrams obtained by placing a based circle below each diagram, and $\mathcal{D}''$ is the diagram obtained from $\mathcal{D}_0\sqcup\mathcal{D}_0'$ in the same manner, then there is automatically an isomorphism $\widetilde{Z}(\mathcal{D})\oplus\widetilde{Z}(\mathcal{D}')\stackrel{\cong}{\longrightarrow}\widetilde{Z}(\mathcal{D}'')$ because there is a canonical bijection between the set of unmarked components of $\mathcal{D}\sqcup\mathcal{D}'$, regarded as a single diagram with two marked components, and the unmarked components of $\mathcal{D}''$ which sends an unmarked component to itself (see Figure \ref{Identify} for an example). This then induces an isomorphism $\Lambda^*\widetilde{Z}(\mathcal{D})\otimes\Lambda^*\widetilde{Z}(\mathcal{D}')\stackrel{\cong}{\longrightarrow}\Lambda^*\widetilde{Z}(\mathcal{D}'')$.
\begin{figure}
	\begin{center}
		\includegraphics[scale=1]{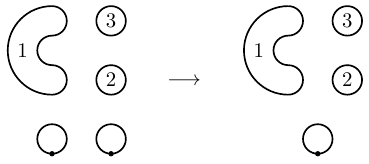}
	\end{center}
	\caption[The canonical identification between the unmarked components of $\mathcal{D}\sqcup\mathcal{D}'$ and $\mathcal{D}''$.]{The canonical identification between the unmarked components of a diagram of the form $\mathcal{D}\sqcup\mathcal{D}'$ (left) and the corresponding diagram $\mathcal{D}''$ (right).}
	\label{Identify}
\end{figure}
We are now ready to prove that $H_*\mathfrak{h}_n$ and $H_n$ are isomorphic.
\begin{theorem}\label{thm:AlgebrasAgree}
	Let $\overline{\circ}_2:H_*\mathfrak{h}_n\otimes_{\F}H_*\mathfrak{h}_n\to H_*\mathfrak{h}_n$ denote the operation induced by $\circ_{2}$ on homology. Then $(H_*\mathfrak{h}_n,\overline{\circ}_2)\cong H_n$ as associative algebras.
\end{theorem}
\begin{proof}
	Note that we may regard $H_n$ as the algebra
	\begin{align}
		H_n=\bigoplus_{a,b\in\mathfrak{C}_n}\CKhred(a_+^!b_+),
	\end{align}
	where we place a basepoint on the bottom-most circle of $a_+^!b_+$ and regard $\CKhred(a_+^!b_+)$ as the quotient complex of $\CKh(a_+^!b_+)$ wherein the marked component is labeled $1$. The multiplication $m$ on $H_n$ is then given by
	\begin{align}
		m=\sum_{a,b,c\in\mathfrak{C}_n}\CKhred(C_{abc}\sqcup\id_{\bigcirc}),
	\end{align}
	where $C_{abc}:a^!b\sqcup b^!c\to a^!c$ is the minimal saddle cobordism.
	
	Note that the pair-of-pants cobordism $W:\Sigma(a_+^!b_+)\sqcup\Sigma(b_+^!c_+)\to\Sigma(a_+^!c_+)$ decomposes as $W=\Sigma(C_{abc}\sqcup\id_{\bigcirc})\circ W_\#$, where
	\begin{align}
		W_\#:\Sigma(a_+^!b_+)\sqcup\Sigma(b_+^!c_+)\to\Sigma((a^!b\sqcup b^!c)\sqcup\bigcirc)
	\end{align}
	is the connected sum cobordism given by taking the connected sum at the preimages of the basepoints on the bottom-most circles in $a_+^!b_+$ and $b_+^!c_+$. We may decompose $C_{abc}\sqcup\id_{\bigcirc}$ as a movie $P_1\stackrel{s_1}{\to}\cdots\stackrel{s_{k-1}}{\to}P_k$ of planar unlinks, where $P_1=a^!b\sqcup b^!c\sqcup\bigcirc$, $P_k=a_+^!c_+$ and $P_i\stackrel{s_i}{\to}P_{i+1}$ is a single saddle cobordism, so that $\Sigma(C_{abc}\sqcup\id_{\bigcirc})=\Sigma(s_{k-1})\circ\cdots\circ\Sigma(s_1)$. This then allows us to further decompose $W$ as
	\begin{align}
		W=\Sigma(s_{k-1})\circ\cdots\circ\Sigma(s_1)\circ W_\#
	\end{align}
	
	Regarding $P_i$ and $P_{i+1}$ as successive resolutions $P_i=\mathcal{D}_i(0)$ and $P_{i+1}=\mathcal{D}_i(1)$ of a link diagram $\mathcal{D}_i$ with a single crossing, there is a commutative square
	\begin{align}
		\begin{tikzcd}[ampersand replacement=\&]
			\CKhred(P_i)\arrow[r,"\CKhred(s_i)"]\arrow[d,"\psi_i"'] \& \CKhred (P_{i+1})\arrow[d,"\psi_{i+1}"]\\
			\widehat{\HF}(\Sigma(P_i))\arrow[r,"\widehat{F}_{\Sigma(s_i)}"] \& \widehat{\HF}(\Sigma(P_{i+1})),
		\end{tikzcd}
	\end{align}
	where $\psi_i=\psi_{\mathcal{D}_i(0)}:\CKhred(P_i)\to\widehat{\HF}(\Sigma(P_i))$ is the isomorphism constructed in the proof of \cite[Proposition 6.2]{OzsSzBranched} (see page \pageref{OzsSzProp6-2}). Note that, since the construction of each $\psi_{i}$ depends only on the diagram $P_i$, we have $\psi_{i+1}=\psi_{\mathcal{D}_{i+1}(0)}=\psi_{\mathcal{D}_i(1)}$. For $a,b\in\mathfrak{C}_n$, let $\psi_{ab}=\psi_{a_+^!b_+}$. We claim that the diagrams
	\begin{align}
		\begin{tikzcd}[ampersand replacement=\&,column sep=1.25cm]
			\CKhred(a_+^!b_+)\otimes\CKhred(b_+^!c_+)\arrow[d,"\psi_{ab}\otimes\psi_{bc}"']\arrow[r,"f_{abc}"] \& \CKhred(P_1)\arrow[d,"\psi_1"]\\
			\widehat{\HF}(\Sigma(a_+^!b_+))\otimes\widehat{\HF}(\Sigma(b_+^!c_+))\arrow[r,"\widehat{F}_{W_\#}"] \& \widehat{\HF}(\Sigma(P_1))
		\end{tikzcd}
	\end{align}
	and
	\begin{align}
		\begin{tikzcd}[ampersand replacement=\&,column sep=2cm]
			\CKhred(P_1)\arrow[r,"\CKhred(C_{abc}\sqcup\id_{\bigcirc})"]\arrow[d,"\psi_1"'] \& \CKhred(P_k)\arrow[d,"\psi_k"]\\
			\widehat{\HF}(\Sigma(P_1))\arrow[r,"\widehat{F}_{\Sigma(C_{abc}\sqcup\id_{\bigcirc})}"] \& \widehat{\HF}(\Sigma(P_k))
		\end{tikzcd}
	\end{align}
	commute, where $f_{abc}:\CKhred(a_+^!b_+)\otimes\CKhred(b_+^!c_+)\to\CKhred(a_+^!c_+)$ is the isomorphism given by
	\begin{align}
		(a^!b\sqcup\bigcirc_1,\bm{v})\otimes(b^!c\sqcup\bigcirc_1,\bm{w})\mapsto((a^!b\sqcup b^!c)\sqcup\bigcirc_1,\bm{v}\sqcup\bm{w})
	\end{align}
	for any labelings $\bm{v}$ and $\bm{w}$ of $a^!b$ and $b^!c$. In other words, $f_{abc}$ is the composite of the isomorphisms $\Lambda^*\widetilde{Z}(a_+^!b_+)\otimes\Lambda^*\widetilde{Z}(b_+^!c_+)\stackrel{\cong}{\longrightarrow}\Lambda^*(\widetilde{Z}(a_+^!b_+)\oplus\widetilde{Z}(b_+^!c_+))$ and $\Lambda^*(\widetilde{Z}(a_+^!b_+)\oplus\widetilde{Z}(b_+^!c_+))\stackrel{\cong}{\longrightarrow}\Lambda^*\widetilde{Z}((a^!b\sqcup b^!c)\sqcup\bigcirc)$. Here, $\widehat{F}_{W_\#}$ is the map associated to $W_\#$, regarded as a graph cobordism $(\Sigma(a_+^!b_+)\sqcup\Sigma(b_+^!c_+),\{w_1,w_2\})\to(\Sigma(P_1),w)$, as in \cite[Proposition 5.2]{HendricksManolescuZemke2017}. By \cite[Proposition 8.1]{Zemke2021}, this map computes the connected sum isomorphism of \cite[Proposition 6.1]{OzsvathSzabo2004Properties} given on generators at the chain level by the identification
	\begin{align}
		\mathbb{T}_\gamma\cap\mathbb{T}_\delta=(\mathbb{T}_{\alpha_1}\cap\mathbb{T}_{\beta_1})\times(\mathbb{T}_{\alpha_2}\cap\mathbb{T}_{\beta_2}),
	\end{align}
	where $(\Sigma,\bm{\gamma},\bm{\delta},z)=(\Sigma_1,\bm{\alpha}_1,\bm{\beta}_1,z_1)\#(\Sigma_2,\bm{\alpha}_2,\bm{\beta}_2,z_2)$ is the connected sum of the Heegaard diagrams $(\Sigma_1,\bm{\alpha}_1,\bm{\beta}_1,z_1)$ and $(\Sigma_2,\bm{\alpha}_2,\bm{\beta}_2,z_2)$ for $\Sigma(a_+^!b_+)$ and $\Sigma(b_+^!c_+)$, respectively, with the connected sum taken at the basepoints $z_1$ and $z_2$, and $z$ a basepoint in the connected sum region of $\Sigma$. More explicitly, $\widehat{F}_{W_\#}$ is given on basis elements by
	\begin{align}
		\widehat{F}_{W_\#}(\xi\cdot\Theta^{\mathrm{top}}_{ab}\otimes\xi'\cdot\Theta^{\mathrm{top}}_{bc})=\xi\otimes\xi'\cdot\Theta^{\mathrm{top}}_{ac},
	\end{align}
	where we identify $\xi\otimes\xi'$ with its image under the isomorphism
	\begin{align}
		\Lambda^*H_1(\Sigma(a_+^!b_+))\otimes\Lambda^*H_1(\Sigma(b_+^!c_+))\to\Lambda^*H_1(\Sigma(a_+^!c_+))
	\end{align}
	induced by the identification of $H_1(\Sigma(a_+^!b_+))\oplus H_1(\Sigma(b_+^!c_+))$ with $H_1(\Sigma(a_+^!c_+))$, which we outline as follows. Note that $P_1$ is obtained from the doubly-pointed diagram $a_+^!b_+\sqcup b_+^!c_+$ by merging the two marked components into one. If $\gamma_1^1,\gamma_2^1,\dots,\gamma_k^1$ are the arcs from the marked component of $a_+^!b_+$ to the remaining components and $\gamma_1^2,\gamma_2^2,\dots,\gamma_\ell^2$ are the arcs for $b_+^!c$, then there is a natural choice of bijection $\sigma$ between $\{\gamma_1^1,\gamma_2^1,\dots,\gamma_k^1\}\sqcup\{\gamma_1^2,\gamma_2^2,\dots,\gamma_\ell^2\}$ and the set of arcs for $P_1$ as illustrated in Figure \ref{Merger}.
	\begin{figure}
		\begin{center}
			\includegraphics[scale=1]{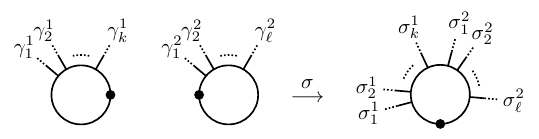}
		\end{center}
		\caption[The bijection between the arcs for the diagrams $a_+^!b_+\sqcup b_+^!c_+$ and $P_1$.]{The bijection $\sigma$ between the arcs for the diagrams $a_+^!b_+\sqcup b_+^!c_+$ and $P_1$. Here, $\sigma_i^j=\sigma(\gamma_i^j)$.}
		\label{Merger}
	\end{figure}
	We then have an explicit isomorphism
	\begin{align}
		H_1(\Sigma(a_+^!b_+))\oplus H_1(\Sigma(b_+^!c_+))\cong H_1(\Sigma(P_1))
	\end{align}
	given by $[\tilde{\gamma}_i^j]\mapsto[\tilde{\sigma}_i^j]$, where $\tilde{\sigma}_i^j$ is the preimage of $\sigma(\gamma_i^j)$ in $\Sigma(P_1)$.
	
	Now, since $\widehat{F}_{W_\#}$ agrees with the map of modules induced by the isomorphism
	\begin{align}
		\Lambda^*H_1(\Sigma(a_+^!b_+))\otimes\Lambda^*H_1(\Sigma(b_+^!c_+))\cong\Lambda^*H_1(\Sigma(P_1)),
	\end{align}
	this tells us that $\widehat{F}_{W_\#}\circ(\psi_{ab}\otimes\psi_{bc})=\psi_{ac}\circ f_{abc}$.
	
	The fact that the second diagram commutes follows immediately from functoriality of reduced Khovanov and Heegaard Floer homology and the fact that the diagram
	\begin{align}
		\begin{tikzcd}[ampersand replacement=\&,column sep=1.25cm]
			\CKhred(P_1)\arrow[r,"\CKhred(s_1)"]\arrow[d,"\psi_1"'] \& \CKhred (P_2)\arrow[d,"\psi_2"]\arrow[r,"\CKhred(s_2)"] \& \cdots\arrow[r,"\CKhred(s_{k-1})"] \& \CKhred(P_k)\arrow[d,"\psi_k"]\\
			\widehat{\HF}(\Sigma(P_1))\arrow[r,"\widehat{F}_{\Sigma(s_1)}"] \& \widehat{\HF}(\Sigma(P_2))\arrow[r,"\widehat{F}_{\Sigma(s_2)}"] \& \cdots\arrow[r,"\widehat{F}_{\Sigma(s_{k-1})}"] \& \widehat{\HF}(\Sigma(P_k))
		\end{tikzcd}
	\end{align}
	commutes, which in turn follows from the fact that each individual square in this diagram commutes.
	
	For the sake of brevity, define $\Mor^{\mathcal{A}_n}(a_+,b_+)=\Mor^{\mathcal{A}_n}(\widehat{\CFD}(a_+),\widehat{\CFD}(b_+))$. Since $\circ_{2}:\Mor^{\mathcal{A}_n}(a_+,b_+)\otimes\Mor^{\mathcal{A}_n}(b_+,c_+)\to\Mor^{\mathcal{A}_n}(a_+,c_+)$ induces the cobordism map $\widehat{F}_W=\widehat{F}_{\Sigma(s_{k-1})}\circ\cdots\circ\widehat{F}_{\Sigma(s_1)}\circ\widehat{F}_{W_\#}$ on homology, it then follows that there is an isomorphism $(H_*\mathfrak{h}_n,\overline{\circ}_2)\cong H_n$ of associative algebras since the square
	\begin{align}
		\begin{tikzcd}[ampersand replacement=\&,column sep=2cm]
			H_n\otimes H_n\arrow[r,"m"]\arrow[d,"\psi\otimes\psi"'] \& H_n\arrow[d,"\psi"]\\
			H_*\mathfrak{h}_n\otimes H_*\mathfrak{h}_n\arrow[r,"\overline{\circ}_2"] \& H_*\mathfrak{h}_n
		\end{tikzcd}
	\end{align}
	commutes, where $\psi=\sum\limits_{a,b\in\mathfrak{C}_n}\psi_{ab}:H_n\to H_*\mathfrak{h}_n$ is the linear isomorphism assembled from the $\psi_{ab}$.
\end{proof}
We will prove the following result in Section 6.
\begin{theorem}\label{Nonformal}
	The differential algebras $\mathfrak{h}_n$ are not formal for $n>1$. In other words, whenever $n>1$, there are nontrivial higher $A_\infty$-operations on $H_*\mathfrak{h}_n$ under the canonical quasi-isomorphism $q:H_*\mathfrak{h}_n\to\mathfrak{h}_n$ of \textup{\cite[Proposition 7]{KONTSEVICH2001}}. In fact, $H_*\mathfrak{h}_n$ is an unbounded $A_\infty$-algebra for all $n>1$.
\end{theorem}
\begin{remark}
	It is an easy exercise in using the homological perturbation lemma to show that $\mathfrak{h}_1$ is formal.
\end{remark}
Compare with \cite[Theorem 1.1]{Abouzaid2016}, which tells us that the \emph{symplectic} arc algebras are formal in characteristic 0. Compare also with \cite[Theorem 7.3]{Klamt11} (see also \cite[Lemma 3.12]{Auroux2012}), which says that the quiver algebras of \cite{KhovanovSeidel2001} --- which are described in \cite{Klamt11} via Ext algebras of direct sums of Verma modules in the parabolic category $\mathcal{O}$ --- are formal. Both of these (families of) algebras are of the form $\bigoplus_{i,j}\HF(L_i,L_j)$, where $\{L_i\}$ is some collection of Lagrangian submanifolds in a symplectic manifold and $\HF(L_i,L_j)$ is the Lagrangian intersection Floer homology of $L_i$ and $L_j$. In the former case, the ambient symplectic manifolds are transverse slices $\mathscr{Y}_n$ to the adjoint quotient maps $\mathfrak{sl}_{2n}(\mathbb{C})\to\mathbb{C}^{2n-1}$ at a nilpotent matrix with two identical Jordan blocks and the $L_i$ are compact exact Lagrangian submanifolds which are diffeomorphic to $S^2\times\cdots\times S^2$ and determined by crossingless matchings. In the latter, the ambient manifold is the Milnor fiber $M$ of a complex $n$-dimensional $A_m$-singularity and the $L_i$ are Lagrangian $n$-spheres, save for one which is a Lagrangian $n$-ball with boundary on $\partial M$, determined by embedded curves in a surface with boundary. In \cite[Theorems 7.7]{Klamt11}, Klamt shows that the Ext algebras of certain related Verma modules are non-formal but bounded.

In view of the results of \cite{Auroux2010}, which interprets the right $A_\infty$-modules $\widehat{\CFA}(Y)$ of bordered Floer homology in terms of Lagrangian correspondences in the extended partially wrapped Fukaya category of symmetric products of a surface and Heegaard Floer complexes $\widehat{\CF}(-Y_1\cup_\partial Y_2)\simeq\Mor_{\mathcal{A}}(\widehat{\CFA}(Y_1),\widehat{\CFA}(Y_2))$ as Hom complexes in the same extended Fukaya category, Theorem \ref{Nonformal} can be thought of as providing an example of a an arc algebra-like structure defined in terms of symplectic data which is intrinsically different from previously known examples, all of which are formal or have bounded minimal models. Though we will not discuss it here, since it is beyond the scope of this article, it would be interesting to give a purely symplectic description of the algebras $H_*\mathfrak{h}_n$ in terms of explicit Lagrangian correspondences via the machinery of \cite{Auroux2010}. It would also be interesting to compare the results of the next section of this paper with those of \cite{Auroux2012,Auroux2015}, which concern similar spectral sequences relating Khovanov--Seidel braid invariants and bordered Floer bimodules (as well as Hochschild homologies thereof).
	\section{Branched Bimodules and the Spectral Sequence}
	Let $T$ be an $(m,n)$-tangle diagram, i.e. a tangle diagram with $2m$ left-endpoints and $2n$ right-endpoints. We are now ready to define the $A_\infty$-$(\mathfrak{h}_m,\mathfrak{h}_n)$-bimodule $\CF^\beta(T)$, the \emph{branched Floer bimodule}, associated to $T$. As a first definition, we take $\CF^\beta(T)=\CF_{\bm{\nu}_m,\bm{\nu}_n}(\Sigma(T^+))$, where $T^+$ is the result of placing a 2-stranded identity braid beneath $T$, i.e.
\begin{align}
	T^+=\,\,\raisebox{-0.775cm}{\includegraphics[scale=1]{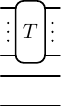}}
\end{align}
However, we will construct this object more explicitly using a choice of strongly bordered Heegaard diagrams for branched double covers of braid generators in $I\times S^2$ due to Lipshitz--Ozsv\'{a}th--Thurston \cite[Section 6.1]{LOTSpectral2} that is adapted to the construction of the Ozsv\'{a}th-Szab\'{o} spectral sequence. We focus first on the case of a single positive braid generator $s_i$ on $2n$ strands. Let $\hat{s}_i$ be its \emph{braidlike resolution}, i.e. the identity braid thought of as a resolution of $s_i$, and $\check{s}_i$ its \emph{anti-braidlike resolution}, i.e. the resolution of $s_i$ obtained by replacing the crossing with a cap followed by a cup --- see Figure \ref{fig:BraidlikeAntibraidlike}.
\begin{figure}
	\begin{align*}
		s_i=\,\raisebox{-1.675cm}{\includegraphics[scale=1]{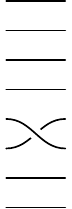}}\hspace{1cm}\hat{s}_i=\,\raisebox{-1.675cm}{\includegraphics[scale=1]{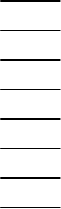}}\hspace{1cm}\check{s}_i=\,\raisebox{-1.675cm}{\includegraphics[scale=1]{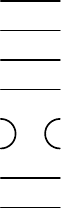}}
	\end{align*}
	\caption{A braid generator $s_i$ and its braidlike and anti-braidlike resolutions, $\hat{s}_i$ and $\check{s}_i$.}\label{fig:BraidlikeAntibraidlike}
\end{figure}
We simultaneously describe the constructions of $\mathcal{H}_{s_i}$, $\mathcal{H}_{\hat{s}_i}$, and $\mathcal{H}_{\check{s}_i}$: each begins with $4n-4$ horizontal segments, corresponding to the ends of the braid, placed between two vertical line segments --- analogously to the construction of the diagram for the plat closure. One then inserts a handle and one of $\beta^0$, $\beta^1$, or $\beta^\infty$ in the crossing region --- where the specific curve is given by the multidiagram
\begin{align}
	\raisebox{-0.75cm}{\includegraphics[scale=1]{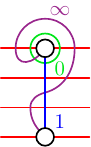}}
\end{align}
--- depending on whether one is constructing $\mathcal{H}_{\hat{s}_i}$, $\mathcal{H}_{\check{s}_i}$, or $\mathcal{H}_{s_i}$, respectively. The reason for this notation is that the strongly bordered 3-manifolds $\Sigma(\hat{s}_i)$, $\Sigma(\check{s}_i)$, and $\Sigma(s_i)$ are the $0$-, $1$-, and $\infty$-surgeries, respectively, of the knot $\kappa_i^+$ in $\Sigma(s_i)$ given by the preimage of the arc $\alpha_i^+$
\begin{align}
	\raisebox{-0.75cm}{\includegraphics[scale=1]{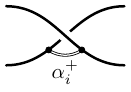}}
\end{align}
--- depicted here in slightly perturbed form --- given by a line segment joining the two arcs of the crossing at the double point. Similarly, if we had started with $s_i^{-1}$, the corresponding multidiagram is
\begin{align}
	\raisebox{-0.75cm}{\includegraphics[scale=1]{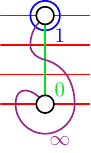}}
\end{align}
and the curves $\beta^0$, $\beta^1$, and $\beta^\infty$ correspond to $\check{s}_i^{-1}$, $\hat{s}_i^{-1}$, and $s_i^{-1}$, respectively, since $\Sigma(\check{s}_i^{-1})$, $\Sigma(\hat{s}_i^{-1})$, and $\Sigma(s_i^{-1})$ are the 0-, 1-, and $\infty$-surgeries of the knot $\kappa_i^-$ in $\Sigma(s_i^{-1})$ obtained analogously to $\kappa_i^+$.

Now, if $T$ is an elementary cap or cup tangle, we choose any strongly bordered Heegaard diagram for $\Sigma(T^+)$. If $T$ is an arbitrary $(m,n)$-tangle diagram, we may decompose $T^+$ into elementary tangles $T^+=T_1T_2\cdots T_k$ --- i.e. each $T_i$ is $s_i$, $s_i^{-1}$, a cap, or a cup --- and we obtain $\mathcal{H}_T$ by gluing the constituent diagrams $\mathcal{H}_{T_i}$, so, if $\widehat{\CFDA}(T^+):=\widehat{\CFDA}(\mathcal{H}_T)$, we have
\begin{align}
	\widehat{\CFDA}(T^+)\simeq\widehat{\CFDA}(\mathcal{H}_{T_1})\boxtimes\widehat{\CFDA}(\mathcal{H}_{T_2})\boxtimes\cdots\boxtimes\widehat{\CFDA}(\mathcal{H}_{T_k})
\end{align}
and we take the right-hand side of this homotopy equivalence as our preferred model for $\widehat{\CFDA}(T^+)$. We recall the following results of Lipshitz--Ozsv\'{a}th--Thurston.
\begin{proposition}[{\cite[Corollary 1.3]{LOTSpectral1}}]
	There are distinguished type-$\mathit{DA}$ bimodule morphisms $F^-:\widehat{\CFDA}(\hat{s}_i)\to\widehat{\CFDA}(\check{s}_i)$ and $F^+:\widehat{\CFDA}(\check{s}_i)\to\widehat{\CFDA}(\hat{s}_i)$ which are uniquely characterized up to homotopy by $\mathrm{Cone}(F^-)\simeq\widehat{\CFDA}(s_i^{-1})$ and $\mathrm{Cone}(F^+)\simeq\widehat{\CFDA}(s_i)$.
\end{proposition}
\begin{lemma}[{\cite[Lemma 2.7]{LOTSpectral1}}]
	Let $S$ and $T$ be finite partially ordered sets and suppose that $M$ and $N$ are $S$- and $T$-filtered type-$\mathit{DA}$ bimodules over $(\mathcal{A},\mathcal{B})$ and $(\mathcal{B},\mathcal{C})$, respectively. Then $M\boxtimes N$ is naturally an $(S\times T)$-filtered type-$\mathit{DA}$ bimodule over $(\mathcal{A},\mathcal{C})$. The analogous result holds when $M$ is a right $A_\infty$-module or $N$ is a left type-$D$ structure.
\end{lemma}
Consequently, if $c$ is the number of positive or negative braid generators in the decomposition of $T$ into elementary tangles, $\widehat{\CFDA}(T)$ becomes a $\{0,1\}^c$-filtered type-$\mathit{DA}$ bimodule by regarding the box tensorands of the form $\widehat{\CFDA}(s_i^{\pm 1})$ as filtered type-$\mathit{DA}$ bimodules.

Now, let $L$ be a link diagram obtained as $L=a^!Tb$ for some tangle diagram $T$ and crossingless matchings $a$ and $b$.
\begin{theorem}[{\cite[Theorem 3]{LOTSpectral1} and \cite[Theorem 2]{LOTSpectral2}}]
	The above discussion endows $\widehat{\CF}(\Sigma(L))$ with a filtration by $\{0,1\}^c$ whose associated spectral sequence has $E_2$-page identified with the reduced Khovanov homology $\Khred(mL)$ and $E_\infty$-page $\widehat{\HF}(\Sigma(L))$. Moreover, this spectral sequence agrees with the one given by \cite[Theorem 1.1]{OzsSzBranched}.
\end{theorem}
\begin{corollary}\label{CFbetaFiltration}
	If $T$ is an $(m,n)$-tangle diagram, there is a filtration on $\CF^\beta(T)$ induced by the one on $\widehat{\CFDA}(T^+)$ such that $H_*(\gr_*\CF^\beta(T))\cong\CKh(mT)$ as vector spaces.
\end{corollary}
More explicitly, we may regard $\CF^\beta(T)$ as the direct sum
\begin{align}
	\bigoplus_{a,b\in\mathfrak{B}_n}\overline{\widehat{\CFD}(a)}\boxtimes\mathcal{A}\boxtimes\widehat{\CFDA}(T_1^+)\boxtimes\cdots\boxtimes\widehat{\CFDA}(T_k^+)\boxtimes\widehat{\CFD}(b),
\end{align}
where $T=T_1\cdots T_k$ is a decomposition of $T$ into elementary tangles. $\CF^\beta(T)$ then inherits a filtration from the braid generator tensorands by regarding the $\widehat{\CFD}(c)$, for $c=a,b$, and those $\widehat{\CFDA}(T_i^+)$ for crossingless $T_i$ as trivially filtered (bi)modules.
\begin{lemma}[{\cite{CohenGuth}}]\label{PairingLemma}
	Let $(N_1,\delta_1^1)$ and $(N_2,\delta_2^1)$ be left type-$D$ structures over a dg-algebra $\mathcal{A}$ over $\Bbbk$, and let $M$ be a bounded right $A_\infty$-module over $\mathcal{A}$. Assume that $N_1$ and $N_2$ are homotopy equivalent to bounded type-$D$ structures. Then there is a homotopy commutative square
	\begin{align}
		\begin{tikzcd}[ampersand replacement=\&]
			M\boxtimes N_1\otimes\Mor^{\mathcal{A}}(N_1,N_2)\arrow[r,"\id\boxtimes\mathit{ev}^1"]\arrow[d,"\simeq"'] \& M\boxtimes N_2\\
			\Mor^{\mathcal{A}}(N_0,N_1)\otimes\Mor^{\mathcal{A}}(N_1,N_2)\arrow[r,"\circ_2"] \& \Mor^{\mathcal{A}}(N_0,N_2)\arrow[u,"\simeq"']
		\end{tikzcd},
	\end{align}
	where $N_0$ is any left type-$D$ structure such that $M\simeq\overline{N}_0\boxtimes\mathcal{A}$ as right $A_\infty$-modules and $\mathit{ev}^1:N_1\otimes\Mor^{\mathcal{A}}(N_1,N_2)\to\mathcal{A}\otimes N_2$ is the evaluation map.
\end{lemma}
\begin{remark}
	It is always possible to choose such an $N_0$ by \cite[Lemma 2.15]{LOTMorphism} and \cite[Proposition 2.3.22]{LOTBimodules2015} --- for example $\overline{\mathrm{Bar}(\mathcal{A})}\boxtimes\overline{M}$.
\end{remark}
The proof of the above lemma is straightforward and will appear in \cite{CohenGuth}.
\begin{corollary}\label{cor:CircTensorGluing}
	Suppose that $Y_i$ are bordered 3-manifolds with boundaries parameterized by $F_i=F(\mathcal{Z}_i)$ for $i=1,2,3$ and that $Y$ and $Y'$ are strongly bordered 3-manifolds with $\partial_LY\cong F_1$, $\partial_RY\cong\partial_LY'\cong F_2$, and $\partial_RY'\cong F_3$. Then there is a homotopy commutative square
	\begin{align*}
		\begin{tikzcd}[ampersand replacement=\&,column sep=1.35cm]
			\Mor^{\mathcal{A}_1}(Y_1,Y\boxtimes Y_2)\otimes_\F\Mor^{\mathcal{A}_2}(Y_2,Y'\boxtimes Y_3)\arrow[r,"{\circ_2(-,\id\boxtimes-)}"]\arrow[d,"\simeq"'] \& \Mor^{\mathcal{A}_2}(Y_1,Y\boxtimes Y'\boxtimes Y_3)\arrow[d,"\simeq"]\\
			\widehat{\CF}(-Y_1\!\cup_{\partial_1}\!\! Y\!\cup_{\partial_2}\!\!Y_2)\otimes_\F\widehat{\CF}(-Y_2\!\cup_{\partial_2}\!\!Y'\!\cup_{\partial_3}\!\!Y_3)\arrow[r,"\widehat{F}_W"] \& \widehat{\CF}(-Y_1\!\cup_{\partial_1}\!\!Y\!\cup_{\partial_2}\!\!Y'\!\cup_{\partial_3}\!\!Y_3)
		\end{tikzcd},
	\end{align*}
	where, by abuse of notation, we drop $\widehat{\CFD}$ and $\widehat{\CFDA}$ for readability, and $W$ is the cobordism $(-Y_1\!\cup_{\partial_1}\!\! Y\!\cup_{\partial_2}\!\!Y_2)\sqcup(-Y_2\!\cup_{\partial_2}\!\!Y'\!\cup_{\partial_3}\!\!Y_3)\to-Y_1\!\cup_{\partial_1}\!\!Y\!\cup_{\partial_2}\!\!Y'\!\cup_{\partial_3}\!\!Y_3$  given by
	\begin{align*}
		W=(\triangle\times F_2)\cup_{e_1\times F_2}e_1\times(-Y\cup_{F_1}Y_1)\cup_{e_2\times F_2}e_2\times Y_2\cup_{e_3\times F_2}e_3\times(-Y'\cup_{F_3}Y_3).
	\end{align*}
	Here $\triangle$ is a triangle with edges $e_1$, $e_2$, and $e_3$, endowed with the clockwise cyclic ordering.
\end{corollary}
\begin{proof}
	As in the statement above, we will abuse notation throughout the proof and denote (bi)modules by their corresponding bordered 3-manifolds whenever they appear in commutative diagrams in order to save space. By Lemma \ref{PairingLemma}, we have a homotopy commutative square
	\begin{align*}
		\begin{tikzcd}[ampersand replacement=\&,column sep=1.35cm]
			\Mor^{\mathcal{A}_1}(Y_1,Y)\boxtimes Y_2\otimes_\F\Mor^{\mathcal{A}_2}(Y_2,Y'\boxtimes Y_3)\arrow[r,"{\id\boxtimes\mathit{ev}^1}"]\arrow[d,"\simeq"'] \& \Mor^{\mathcal{A}_1}(Y_1,Y)\boxtimes Y'\boxtimes Y_3\\
			\Mor^{\mathcal{A}_2}(-Y\!\cup_{\partial_1}\!\! Y_1,Y_2)\otimes_\F\Mor^{\mathcal{A}_2}(Y_2,Y'\boxtimes Y_3)\arrow[r,"\circ_2"] \& \Mor^{\mathcal{A}_2}(-Y_1\!\cup_{\partial_1}\!\! Y,Y'\boxtimes Y_3)\arrow[u,"\simeq"']
		\end{tikzcd}
	\end{align*}
	and there are isomorphisms of complexes  $\Mor^{\mathcal{A}_1}(Y_1,Y)\boxtimes Y_2\cong\Mor^{\mathcal{A}_1}(Y_1,Y\boxtimes Y_2)$ and $\Mor^{\mathcal{A}_2}(Y_1,Y)\boxtimes Y'\boxtimes Y_3\cong\Mor^{\mathcal{A}_2}(Y_1,Y\boxtimes Y'\boxtimes Y_3)$, as well as a homotopy equivalence $\widehat{\CFDA}(Y')\boxtimes\widehat{\CFD}(Y_3)\simeq\widehat{\CFD}(Y'\!\cup_{\partial_3}\! Y_3)$, so we may extend this to a homotopy commutative diagram of the form
	\begin{align*}
		\begin{tikzcd}[ampersand replacement=\&,column sep=1.35cm]
			\Mor^{\mathcal{A}_1}(Y_1,Y\boxtimes Y_2)\otimes_\F\Mor^{\mathcal{A}_2}(Y_2,Y'\boxtimes Y_3)\arrow[r,"\Gamma"]\arrow[d,"\cong"'] \& \Mor^{\mathcal{A}_1}(Y_1,Y\boxtimes Y'\boxtimes Y_3)\\
			\Mor^{\mathcal{A}_1}(Y_1,Y)\boxtimes Y_2\otimes_\F\Mor^{\mathcal{A}_2}(Y_2,Y'\boxtimes Y_3)\arrow[r,"{\id\boxtimes\mathit{ev}^1}"]\arrow[d,"\simeq"'] \& \Mor^{\mathcal{A}_1}(Y_1,Y)\boxtimes Y'\boxtimes Y_3\arrow[u,"\cong"']\\
			\Mor^{\mathcal{A}_2}(-Y\!\cup_{\partial_1}\!\! Y_1,Y_2)\otimes_\F\Mor^{\mathcal{A}_2}(Y_2,Y'\boxtimes Y_3)\arrow[r,"\circ_2"]\arrow[d,"\simeq"'] \& \Mor^{\mathcal{A}_2}(-Y\!\cup_{\partial_1}\!\! Y_1,Y'\boxtimes Y_3)\arrow[u,"\simeq"']\\
			\Mor^{\mathcal{A}_2}(-Y\!\cup_{\partial_1}\!\! Y_1,Y_2)\otimes_\F\Mor^{\mathcal{A}_2}(Y_2,Y'\!\cup_{\partial_3}\! Y_3)\arrow[r,"\circ_2"] \& \Mor^{\mathcal{A}_2}(-Y\!\cup_{\partial_1}\!\! Y_1,Y'\!\cup_{\partial_3}\! Y_3)\arrow[u,"\simeq"']
		\end{tikzcd}
	\end{align*}
	and we claim that $\Gamma\simeq\circ_2(-,\id\boxtimes-)$.
	
	In graphical notation, we have
	\begin{align}
		\circ_2(-,\id\boxtimes-)=\begin{tikzcd}[ampersand replacement=\&,column sep=0.35cm]
			{} \& {}\arrow[d,dashed] \& {}\arrow[dl,bend left,rightsquigarrow] \& {}\arrow[dddl,bend left=15,rightsquigarrow] \\
			{} \& (-)\arrow[dr,dashed,bend left]\arrow[d,dashed]\arrow[dddl,bend right=15] \& {} \& {} \\
			{} \& \otimes\arrow[d,Rightarrow,densely dashed] \& \delta_2\arrow[d,dashed]\arrow[l,Rightarrow] \& {} \\
			{} \& \otimes\arrow[d,Rightarrow,densely dashed] \& (-)\arrow[d,dashed]\arrow[dr,dashed,bend left]\arrow[l] \& {} \\
			\mu_2\arrow[d] \& \delta_Y^1\arrow[d,dashed]\arrow[l] \& \delta_{Y'}\arrow[d,dashed]\arrow[l,Rightarrow] \& \delta_3\arrow[d,dashed]\arrow[l,Rightarrow] \\
			{} \& {} \& {} \& {} \\
		\end{tikzcd},
	\end{align}
	where squiggly arrows $\rightsquigarrow$ denote elements of morphism spaces in the sense that
	\begin{align}
		\circ_2(f,\id\boxtimes g)=\begin{tikzcd}[ampersand replacement=\&,column sep=0.35cm]
			{} \& {}\arrow[d,dashed] \& f\arrow[dl,bend left,rightsquigarrow] \& g\arrow[dddl,bend left=15,rightsquigarrow] \\
			{} \& (-)\arrow[dr,dashed,bend left]\arrow[d,dashed]\arrow[dddl,bend right=15] \& {} \& {} \\
			{} \& \otimes\arrow[d,Rightarrow,densely dashed] \& \delta_2\arrow[d,dashed]\arrow[l,Rightarrow] \& {} \\
			{} \& \otimes\arrow[d,Rightarrow,densely dashed] \& (-)\arrow[d,dashed]\arrow[dr,dashed,bend left]\arrow[l] \& {} \\
			\mu_2\arrow[d] \& \delta_Y^1\arrow[d,dashed]\arrow[l] \& \delta_{Y'}\arrow[d,dashed]\arrow[l,Rightarrow] \& \delta_3\arrow[d,dashed]\arrow[l,Rightarrow] \\
			{} \& {} \& {} \& {} \\
		\end{tikzcd}=\begin{tikzcd}[ampersand replacement=\&,column sep=0.35cm]
		{} \& {}\arrow[d,dashed] \& {} \& {} \\
		{} \& f\arrow[dr,dashed,bend left]\arrow[d,dashed]\arrow[dddl,bend right=15] \& {} \& {} \\
		{} \& \otimes\arrow[d,Rightarrow,densely dashed] \& \delta_2\arrow[d,dashed]\arrow[l,Rightarrow] \& {} \\
		{} \& \otimes\arrow[d,Rightarrow,densely dashed] \& g\arrow[d,dashed]\arrow[dr,dashed,bend left]\arrow[l] \& {} \\
		\mu_2\arrow[d] \& \delta_Y^1\arrow[d,dashed]\arrow[l] \& \delta_{Y'}\arrow[d,dashed]\arrow[l,Rightarrow] \& \delta_3\arrow[d,dashed]\arrow[l,Rightarrow] \\
		{} \& {} \& {} \& {} \\
	\end{tikzcd}
	\end{align}
	as a linear map $\widehat{\CFD}(Y_1)\to\mathcal{A}_1\otimes\widehat{\CFDA}(Y)\boxtimes\widehat{\CFDA}(Y')\boxtimes\widehat{\CFD}(Y_3)$, and a dashed doubled arrow is to be interpreted as a dashed arrow next to a doubled arrow, as in
	\begin{align}
		\begin{tikzcd}
			{}\arrow[d,Rightarrow,densely dashed]\\{}
		\end{tikzcd}=\begin{tikzcd}[ampersand replacement=\&,column sep=0cm]
			{}\arrow[d,densely dashed] \& {}\arrow[d,Rightarrow] \\ {} \& {}
		\end{tikzcd},
	\end{align}
	to represent an element of $\widehat{\CFDA}(Y)\otimes T(\mathcal{A}_2)$.
	
	On the other hand, the right $A_\infty$-module structure on $\Mor^{\mathcal{A}_1}(Y_1,Y)$ is given by
	\begin{align}
		m=\begin{tikzcd}[ampersand replacement=\&,column sep=0.35cm]
			{} \& {}\arrow[d,densely dashed] \& {}\arrow[ddl,bend left=15,rightsquigarrow] \& {}\arrow[dddll,Rightarrow,bend left=20]\\
			{} \& \delta_1\arrow[dl,bend right,Rightarrow]\arrow[d,densely dashed] \& {} \& {}\\
			\otimes\arrow[d,Rightarrow] \& (-)\arrow[d,densely dashed]\arrow[l] \& {} \& {}\\
			\mu\arrow[d] \& \delta_Y\arrow[d,densely dashed]\arrow[l,Rightarrow] \& {} \& {}\\
			{} \& {} \& {} \& {}\\
		\end{tikzcd},
	\end{align}
	i.e.
	\begin{align}
		m(f,-)=\begin{tikzcd}[ampersand replacement=\&,column sep=0.35cm]
			{} \& {}\arrow[d,densely dashed] \& f\arrow[ddl,bend left=15,rightsquigarrow] \& {}\arrow[dddll,Rightarrow,bend left=20]\\
			{} \& \delta_1\arrow[dl,bend right,Rightarrow]\arrow[d,densely dashed] \& {} \& {}\\
			\otimes\arrow[d,Rightarrow] \& (-)\arrow[d,densely dashed]\arrow[l] \& {} \& {}\\
			\mu\arrow[d] \& \delta_Y\arrow[d,densely dashed]\arrow[l,Rightarrow] \& {} \& {}\\
			{} \& {} \& {} \& {}\\
		\end{tikzcd}=\begin{tikzcd}[ampersand replacement=\&,column sep=0.35cm]
		{} \& {}\arrow[d,densely dashed] \& {} \& {}\arrow[dddll,Rightarrow,bend left=20]\\
		{} \& \delta_1\arrow[dl,bend right,Rightarrow]\arrow[d,densely dashed] \& {} \& {}\\
		\otimes\arrow[d,Rightarrow] \& f\arrow[d,densely dashed]\arrow[l] \& {} \& {}\\
		\mu\arrow[d] \& \delta_Y\arrow[d,densely dashed]\arrow[l,Rightarrow] \& {} \& {}\\
		{} \& {} \& {} \& {}\\
	\end{tikzcd}
	\end{align}
	as a linear map $\widehat{\CFD}(Y_1)\otimes T(\mathcal{A}_2)\to\mathcal{A}_1\otimes\widehat{\CFDA}(Y)$. Consequently, the map $\id\boxtimes\mathit{ev}^1$ is given graphically in the present application by
	\begin{align}
		\id\boxtimes\mathit{ev}^1=\begin{tikzcd}[ampersand replacement=\&,column sep=0.35cm]
			{} \& {}\arrow[d,densely dashed] \& {}\arrow[ddl,rightsquigarrow,bend left=15] \& {}\arrow[d,densely dashed] \& {}\arrow[ddl,rightsquigarrow,bend left=15] \\
			{} \& \delta_1\arrow[dl,bend right,Rightarrow]\arrow[d,densely dashed] \& {} \& \delta_2\arrow[dl,bend right,Rightarrow]\arrow[d,densely dashed] \& {} \\
			\otimes\arrow[d,Rightarrow] \& (-)\arrow[l]\arrow[d,densely dashed] \& \otimes\arrow[d,Rightarrow] \& (-)\arrow[l]\arrow[d,densely dashed]\arrow[dr,bend left,densely dashed] \& {} \\
			\mu\arrow[d] \& \delta_Y\arrow[l,Rightarrow]\arrow[d,densely dashed] \& \otimes\arrow[l,Rightarrow] \& \delta_{Y'}\arrow[l,Rightarrow]\arrow[d,densely dashed] \& \delta_3\arrow[l,Rightarrow]\arrow[d,densely dashed] \\
			{} \& {} \& {} \& {} \& {} \\
		\end{tikzcd}.
	\end{align}
	However, $\mathcal{A}_1$ is a differential algebra and $\mu$ necessarily has two inputs, one coming from the first morphism and one from $\delta_Y$, so this reduces to
	\begin{align}
		\id\boxtimes\mathit{ev}^1=\begin{tikzcd}[ampersand replacement=\&,column sep=0.35cm]
			{} \& {}\arrow[dd,densely dashed] \& {}\arrow[ddl,rightsquigarrow,bend left=15] \& {}\arrow[d,densely dashed] \& {}\arrow[ddl,rightsquigarrow,bend left=15] \\
			{} \& {} \& {} \& \delta_2\arrow[dl,bend right,Rightarrow]\arrow[d,densely dashed] \& {} \\
			{} \& (-)\arrow[dl,bend right]\arrow[d,densely dashed] \& \otimes\arrow[d,Rightarrow] \& (-)\arrow[l]\arrow[d,densely dashed]\arrow[dr,bend left,densely dashed] \& {} \\
			\mu_2\arrow[d] \& \delta_Y^1\arrow[l]\arrow[d,densely dashed] \& \otimes\arrow[l,Rightarrow] \& \delta_{Y'}\arrow[l,Rightarrow]\arrow[d,densely dashed] \& \delta_3\arrow[l,Rightarrow]\arrow[d,densely dashed] \\
			{} \& {} \& {} \& {} \& {} \\
		\end{tikzcd}
	\end{align}
	which we may rewrite as
	\begin{align}
		\id\boxtimes\mathit{ev}^1=\begin{tikzcd}[ampersand replacement=\&,column sep=0.35cm]
			{} \& {}\arrow[d,densely dashed] \& {}\arrow[dl,bend left,rightsquigarrow] \& {}\arrow[ddl,densely dashed,bend left=15] \& {}\arrow[dddll,bend left=20,rightsquigarrow] \\
			{} \& (-)\arrow[dddl,bend right=15]\arrow[d,densely dashed] \& {} \& {} \& {} \\
			{} \& \otimes\arrow[d,Rightarrow,densely dashed] \& \delta_2\arrow[l,Rightarrow]\arrow[d,densely dashed] \& {} \& {} \\
			{} \& \otimes\arrow[d,Rightarrow,densely dashed] \& (-)\arrow[l]\arrow[d,densely dashed]\arrow[dr,bend left,densely dashed] \& {} \& {} \\
			\mu_2\arrow[d] \& \delta_Y^1\arrow[l]\arrow[d,densely dashed] \& \delta_{Y'}\arrow[l,Rightarrow]\arrow[d,densely dashed] \& \delta_3\arrow[l,Rightarrow]\arrow[d,densely dashed] \& {} \\
			{} \& {} \& {} \& {} \& {} 
		\end{tikzcd}
	\end{align}
	so, since identifying $\Mor^{\mathcal{A}_1}(Y_1,Y)\boxtimes Y_2$ with $\Mor^{\mathcal{A}_1}(Y_1,Y\boxtimes Y_2)$ has the effect of changing this diagram so that the input from $Y_2$ --- the downward-left-pointing dashed arrow in the top-right of the diagram --- is changed into an output of the first morphism, $\Gamma$ is in fact equal to $\circ_2(-,\id\boxtimes-)$. The desired result now follows from \cite[Theorem 1.1]{CohenComposition}.
\end{proof}
\begin{corollary}\label{cor:ActionsAgree}
	Let $T$ be a planar $(m,n)$-tangle diagram, then there are commutative squares
	\begin{align}\label{line:LeftAction}
		\begin{tikzcd}[ampersand replacement=\&]
			H_m\otimes_{\F}\CKh(T)\arrow[r,"m_L"]\arrow[d,"\cong"'] \& \CKh(T)\arrow[d,"\cong"]\\
			H_*\mathfrak{h}_m\otimes_{\F}H_*\CF^\beta(T)\arrow[r,"\overline{\circ}_2"] \& H_*\CF^\beta(T)
		\end{tikzcd}
	\end{align}
	and
	\begin{align}\label{line:RightAction}
		\begin{tikzcd}[ampersand replacement=\&,column sep=1.75cm]
			\CKh(T)\otimes_{\F}H_n\arrow[r,"m_R"]\arrow[d,"\cong"'] \& \CKh(T)\arrow[d,"\cong"]\\
			H_*\CF^\beta(T)\otimes_{\F}H_*\mathfrak{h}_n\arrow[r,"\overline{\circ_2(-,\id\boxtimes-)}"] \& H_*\CF^\beta(T)
		\end{tikzcd},
	\end{align}
	where $m_L$ and $m_R$ are the left- and right-module structure maps on $\CKh(T)$, respectively.
\end{corollary}
\begin{remark}
	We omit mirrors here since, for a planar tangle diagram $T$, we have $mT=T$.
\end{remark}
\begin{proof}
	The proof is identical to the proof of Theorem \ref{thm:AlgebrasAgree}, with the sole exceptions that we replace the minimal saddle cobordism $C_{abc}:a^!b\sqcup b^!c\to a^!c$ with left and right saddle cobordisms $C_{abc}^L:a^!b\sqcup b^!Tc\to a^!Tc$ and $C_{abc}^R:a^!Tb\sqcup b^!c\to a^!Tc$ for Lines \ref{line:LeftAction} and \ref{line:RightAction}, respectively, and use Corollary \ref{cor:CircTensorGluing} to show that the bottom arrow of the diagram in line \ref{line:RightAction} agrees with the on given by the maps on Heegaard Floer homology induced by the branched double covers of the minimal saddle cobordisms.
\end{proof}
\begin{theorem}\label{thm:InducedActions}
	The $(\mathfrak{h}_m,\mathfrak{h}_n)$-bimodule actions on $\CF^\beta(T)$ induce the usual $(H_m,H_n)$-bimodule actions on $H_*(\gr_*\CF^\beta(T))\cong\CKh(mT)$. 
\end{theorem}
\begin{proof}
	By \cite[Lemma 2.3.13(2)]{LOTBimodules2015}, if, for $i=0,1$, $M_i$ are type-$\mathit{DA}$ bimodules over $(\mathcal{A},\mathcal{B})$ and $N_i$ are left type-$D$ structures over $\mathcal{B}$, then there is a homotopy commutative square
	\begin{align}
		\begin{tikzcd}[ampersand replacement=\&]
			\Mor_{\mathcal{B}}^{\mathcal{A}}(M_0,M_1)\otimes_{\Bbbk}\Mor^{\mathcal{B}}(N_0,N_1)\arrow[r,"A"]\arrow[d,"D"'] \& \overunderset{\displaystyle\Mor^{\mathcal{A}}(M_0\boxtimes N_0,M_0\boxtimes N_1)}{\displaystyle\Mor^{\mathcal{A}}(M_0\boxtimes N_1,M_1\boxtimes N_1)}{\displaystyle\otimes_{\Bbbk}}\arrow[d,"B"]\\
			\overunderset{\displaystyle\Mor^{\mathcal{A}}(M_0\boxtimes N_0,M_1\boxtimes N_0)}{\displaystyle\Mor^{\mathcal{A}}(M_1\boxtimes N_0,M_1\boxtimes N_1)}{\displaystyle\otimes_{\Bbbk}}\arrow[r,"C"] \& \Mor^{\mathcal{A}}(M_0\boxtimes N_0,M_1\boxtimes N_1)
		\end{tikzcd},
	\end{align}
	where $A:f\otimes g^1\mapsto(\id_{M_0}\boxtimes g^1)\otimes(f\boxtimes\id_{N_1})$, $B:k\otimes\ell\mapsto\ell\otimes k$, $C:k\otimes\ell\mapsto\ell\circ k$, and $D:f\otimes g^1\mapsto(f\boxtimes\id_{N_0})\otimes(\id_{M_1}\boxtimes g^1)$, and we read vertical tensor products from top to bottom --- i.e., in this context, $\overunderset{X}{Y}{\otimes}$ means $X\otimes Y$. Consequently the left- and right-actions of $\mathfrak{h}_m$ and $\mathfrak{h}_n$ on $\CF^\beta(T)$ commute up to homotopy with the maps $\circ_2(-,F_i^\pm\boxtimes\id)$ on the $E_0$-page of the spectral sequence associated to the filtration on $\CF^\beta(T)$ which induce the differential on the $E_1$-page. This $E^1$-page is identified as a complex of vector spaces with $\CKh(mT)$ by Corollary \ref{CFbetaFiltration}. Corollary \ref{cor:ActionsAgree} tells us that the $(\mathfrak{h}_m,\mathfrak{h}_n)$-bimodule actions induce the usual bimodule actions on each summand $H_*\CF^\beta(mT_{\bm{v}})\cong\CKh(mT_{\bm{v}})$ corresponding to some complete resolution $mT_{\bm{v}}$ of $mT$.
\end{proof}
We now have all of the ingredients necessary to prove the main result.
\begin{theorem}
	Let $T$ be a $(n,n)$-tangle diagram and let $mL$ be the link in $S^1\times S^2$ represented by the mirror tangle diagram $mT$. Then there is a spectral sequence with $E^2$-page isomorphic to the Rozansky-Willis homology $\RW(mL)$ and converging to $\AHH(\CF^\beta(T))$.
\end{theorem}
\begin{proof}
	This is an immediate consequence of Theorems \ref{thm:AlgebrasAgree}, \ref{thm:InducedActions}, and \ref{thm:GeneralSpectralSequence}.
\end{proof}
\begin{lemma}\label{lemma:QuasiIsoComparisonLemma}
	Let $f:\mathcal{M}_0\to\mathcal{M}_1$ is a map of filtered $A_\infty$-bimodules with bounded filtrations such that the induced map $\gr f:\gr_*\mathcal{M}_0\to\gr_*\mathcal{M}_1$ is a quasi-isomorphism, then $f$ is a quasi-isomorphism.
\end{lemma}
\begin{proof}
	Such an $f$ induces a map of spectral sequences
	\begin{align}
		\begin{tikzcd}[ampersand replacement=\&]
			H_*(\gr_*\mathcal{M}_0)\arrow[r,Rightarrow]\arrow[d,"(\gr f)_*"'] \& \gr_*H_*(\mathcal{M}_0)\arrow[d,"\gr(f_*)"]\\
			 H_*(\gr_*\mathcal{M}_1)\arrow[r,Rightarrow] \& \gr_*H_*(\mathcal{M}_1)
		\end{tikzcd}
	\end{align}
	which is an isomorphism on the $E^r$-page for $r\geq 2$ by \cite[Theorem 3.5]{McCleary2000} since $(\gr f)_*$ is an isomorphism. Since the filtrations on the $\mathcal{M}_i$ are bounded, the spectral sequences converge, so the map $\gr(f_*)$ of $E^\infty$-pages is an isomorphism, and, therefore, $f$ is a quasi-isomorphism.
\end{proof}
\begin{theorem}\label{thm:dg-gluing}
	Given an $(\ell,m)$-tangle $T_1$ and an $(m,n)$-tangle $T_2$, there is a quasi-isomorphism
	\begin{align}
		\CF^\beta(T_1)\widetilde{\otimes}\CF^\beta(T_2)\simeq\CF^\beta(T_1T_2)
	\end{align}
	of $A_\infty$-bimodules.
\end{theorem}
\begin{proof}
	Recall that
	\begin{align}
		\CF^\beta(T_1)\widetilde{\otimes}\CF^\beta(T_2)=\CF^\beta(T_1)\boxtimes\mathrm{Bar}(\mathfrak{h}_m)\boxtimes\CF^\beta(T_2)
	\end{align}
	and define a map $\sigma:\CF^\beta(T_1)\widetilde{\otimes}\CF^\beta(T_2)\to\CF^\beta(T_1T_2)$ of $A_\infty$-bimodules by
	\begin{align}
		\sigma_{i|j|k}(\bm{\varphi},f,\bm{\psi},g,\bm{\omega})=\circ_{4+i-\delta(j,0)-\delta(k,0)}(\bm{\varphi},f,H(\bm{\psi}),\id\boxtimes g,H(\bm{\omega})),
	\end{align}
	where, for $\ell=i,j,$ or $k$, as appropriate, $\pi=\varphi,\psi,\omega$, $\bm{\pi}=\pi_1\otimes\cdots\otimes\pi_\ell$, and $H=\sum_jH_j$. Note that $\sigma_{i|j|k}=0$ unless $i=j=k=0$, in which case we have $\sigma_{0|0|0}(f,g)=\circ_2(f,\id\boxtimes g)$, since $\mathcal{A}_n$ is a differential algebra, so all higher composition maps vanish. Let $\mathcal{L}_i^\bullet$ be the filtration on $\CF^\beta(T_i)$ induced by the filtration on $\widehat{\CFDA}(\Sigma(T_i^+))$ defined in \cite{LOTSpectral1}, $\mathcal{W}^\bullet$ the word-length filtration on $\mathrm{Bar}(\mathfrak{h}_m)$, and let $\mathcal{F}^\bullet=\mathcal{L}_1^\bullet\otimes\mathcal{W}^\bullet\otimes\mathcal{L}_2^\bullet$ be the tensor product filtration. Note that $\sigma$ is automatically filtered with respect to $\mathcal{F}^\bullet$. Note that the filtration-preserving part of the differential is $\partial_1^0\otimes\id\otimes\id+\id\otimes\partial_{\mathrm{Bar}}^0\otimes\id+\id\otimes\id\otimes\partial_2^0$, where $\partial_i^0$ is the filtration-preserving part of the differential on $\CF^\beta(T_i)$, and $\partial_{\mathrm{Bar}}^0$ is the part of the bar differential given by the differential on $\mathfrak{h}_m$, so it follows that $H_*(\gr_*\CF^\beta(T_1)\,\widetilde{\otimes}\,\CF^\beta(T_2))\cong\CKh(T_1)\,\widetilde{\otimes}\,\CKh(T_2)$ and, by Theorem \ref{thm:InducedActions}, the map
	\begin{align}
		(\gr\sigma)_*=\sum_{i,j,k}(\gr\sigma)_{*,i|j|k}
	\end{align}
	is given by the gluing isomorphism $\CKh(T_1)\otimes_{H_m}\!\CKh(T_2)\to\CKh(T_1T_2)$ if we have $i=j=k=0$ and by zero otherwise. Since there is a retract of $\CKh(T_1)\,\widetilde{\otimes}\,\CKh(T_2)$ onto $\CKh(T_1)\otimes_{H_m}\!\CKh(T_2)$, this tells us that $(\gr\sigma)_*$ is a quasi-isomorphism. Therefore, by Lemma \ref{lemma:QuasiIsoComparisonLemma}, $\sigma$ is a quasi-isomorphism.
\end{proof}
\begin{corollary}
	The Hochschild homology $\AHH(\CF^\beta(T))$ is an isotopy invariant of the annular closure of $T$.
\end{corollary}
\begin{proof}
	This follows immediately from Theorem \ref{thm:dg-gluing} and the trace property of Hochschild homology for $A_\infty$-bimodules (see \cite[Lemma 3.7]{Auroux2015}).
\end{proof}
\begin{question}
	Is $\AHH(\CF^\beta(T))$ an isotopy invariant of the link $L\subset S^1\times S^2$ represented by $T$?
\end{question}
If $T_1$ and $T_2$ differ by a lightbulb move, then $\AHH(\CF^\beta(T_1))$ and $\AHH(\CF^\beta(T_2))$ are the $E^\infty$-pages of spectral sequences with isomorphic $E^2$-pages since Rozansky-Willis homology is an invariant of links in $S^1\times S^2$. However, an isomorphism $\RW(mL_1)\cong\RW(mL_2)$ need not, a priori, be induced by a filtered $A_\infty$-bimodule homomorphism $\CF^\beta(T_1)\to\CF^\beta(T_2)$, so it may be too much to expect that the two $E^\infty$-pages are isomorphic.
\begin{question}
	Rozansky--Willis homology $\RW(L)$ can be approximated in any range of homological gradings by the Khovanov homology $\Kh(\mathbb{L}_k)$, where $\mathbb{L}_k$ is the planar closure of the tangle diagram $\bm{\Phi}^kT$, where $T$ is any tangle diagram whose closure in $S^1\times S^2$ is $L$. Is there a similar approximation of $\AHH(\CF^\beta(T))$ by Heegaard Floer homologies of the form $\widehat{\HF}(\Sigma(\mathbb{L}_k\sqcup\bigcirc))$, where $\bigcirc$ is an unknot disjoint from $\mathbb{L}_k$?
\end{question}
\begin{question}
	Can $\AHH(\CF^\beta(T))$ be related to other Floer-theoretic invariants, such as knot Floer homology?
\end{question}
We plan to revisit these questions in future work.
	\section{The Branched Arc Algebras Are Not Formal}
	\subsection{Formality for $A_\infty$-algebras}\label{sec:formality}
Homological perturbation theory allows one to transfer homotopy-algebraic structures on chain complexes along certain types of morphisms. In particular, it allows one to construct a canonical $A_\infty$-algebra structure on the homology of an $A_\infty$-algebra.
\begin{proposition}[Homological perturbation lemma for $A_\infty$-algebras, \cite{KONTSEVICH2001}]
	Let $\mathcal{A}=(A,\{m_i^A\})$ be an $A_\infty$-algebra and let
	\begin{align}
		\begin{tikzcd}[ampersand replacement=\&]
			\mathcal{A}\arrow[loop,out=145,in=215,looseness=8,"h"']\arrow[r,bend right,"p"'] \& H_*\mathcal{A}\arrow[l,bend right,"\iota"']
		\end{tikzcd}
	\end{align}
	be a retract of $\mathcal{A}$ onto its homology $H_*\mathcal{A}$, regarding $(A,m_1)$ as a chain complex. That is to say chain maps $p:\mathcal{A}\to H_*\mathcal{A}$ and $\iota:H_*\mathcal{A}\to\mathcal{A}$, regarding $H_*\mathcal{A}$ as a complex with trivial differential, and a chain homotopy $h:\mathcal{A}\to\mathcal{A}$ such that
	\begin{align}
		\iota p=\id+\partial h+h\partial
	\end{align}
	and
	\begin{align}
		p\iota=\id.
	\end{align}
	Then $H_*\mathcal{A}$ admits an $A_\infty$-algebra structure $\{m_i\}$ such that\textup{\begin{enumerate}
			\item $m_1=0$ and $m_2=(m_2^A)_*$ and
			\item there are $A_\infty$ quasi-isomorphisms $p':\mathcal{A}\to H_*\mathcal{A}$ and $\iota:H_*\mathcal{A}\to\mathcal{A}$ and an $A_\infty$-homotopy $h':\mathcal{A}\to \mathcal{A}$ which extend $p$, $\iota$, and $h$.
	\end{enumerate}}
	The structure maps $m_i:(H_*\mathcal{A})^{\otimes i}\to H_*\mathcal{A}[2-i]$ are given by
	\begin{align}
		m_i=\sum_{T\in\mathscr{P}_i}m_i^T,
	\end{align}
	where $\mathscr{P}_i$ is the set of planar rooted trees with $i$ leaves such that each internal vertex has degree at least 3, and $m_i^T$ is given by labeling the leaves of $T$ by $\iota$, interior edges by $h$, vertices by the $A_\infty$-operations $m_j^A$, and the root by $p$ and regarding this labeled tree as a composition of morphisms $(H_*\mathcal{A})^{\otimes i}\to H_*\mathcal{A}$. This $A_\infty$-algebra structure on $H_*\mathcal{A}$ is independent of the choice of $p$, $\iota$, and $h$ up to $A_\infty$-isomorphism.
\end{proposition}
Note, in particular, that if $\mathcal{A}$ is a genuine differential algebra, then the only trees $T$ contributing to the $A_\infty$-operations on $H_*\mathcal{A}$ are those whose internal vertices are all trivalent, i.e. the binary trees. For instance, there are two trees contributing to $m_3$ and the trees contributing to $m_4$ are those shown in Figure \ref{Trees}.
\begin{figure}
	\begin{center}
		\includegraphics[scale=0.6]{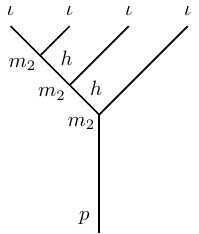}\hspace{0.25cm}\includegraphics[scale=0.6]{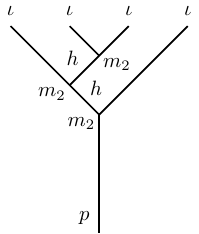}\hspace{0.25cm}\includegraphics[scale=0.6]{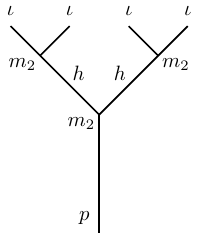}\hspace{0.25cm}\includegraphics[scale=0.6]{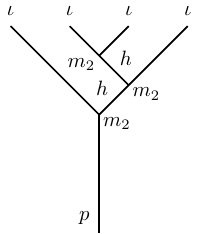}\hspace{0.25cm}\includegraphics[scale=0.6]{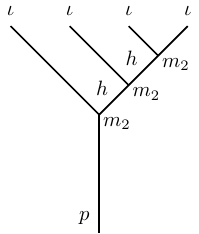}
	\end{center}
	\caption[Trees contributing to $m_4$.]{Trees contributing to the $m_4$ operation on the homology $H_*\mathcal{A}$ of a differential algebra $\mathcal{A}$.}
	\label{Trees}
\end{figure}
\begin{proposition}[\cite{KONTSEVICH2001}, Proposition 7\label{CanonicalAinfQuasiIso}]
	There is a canonical $A_\infty$ quasi-isomorphism $q:H_*\mathcal{A}\to \mathcal{A}$.
\end{proposition}
\begin{proof}[Sketch]
	The map $q_1:H_*\mathcal{A}\to\mathcal{A}$ is defined to be the chain map $\iota$ while the higher $q_i$ are defined by
	\begin{align}
		q_i=\sum_{T\in\mathscr{P}_i}q_i^T,
	\end{align}
	where $q_i^T$ is defined precisely as is $m_i^T$ except that, instead of $p$, we label the root of each tree $T$ by the homotopy $h$. One may then verify that $q=\{q_i\}$ is such a map.
\end{proof}
\begin{definition}
	An $A_\infty$-algebra $(\mathcal{A},m)$ is called \emph{formal} if there is an $A_\infty$-algebra structure $\{\mu_i\}$ on $H_*\mathcal{A}$ with $\mu_i=0$ for $i=1$, $i>2$, and $\mu_2=\overline{m}_2$, together with an $A_\infty$ quasi-isomorphism $i:(H_*\mathcal{A},\mu)\to(\mathcal{A},m)$ such that $i_1$ induces the identity on homology. In other words, if $\mathcal{A}$ is formal, then the higher operations on $\mathcal{A}$ are trivial up to a (canonical) quasi-isomorphism.
\end{definition}
It is easy to show that $\mathfrak{h}_1$ is formal, but we will show that this is not the case for $\mathfrak{h}_n$ with $n>1$. We first need a couple of technical propositions.
\begin{proposition}
	Let $\mathcal{A}_n$ be the weight-0 algebra for the genus $n$ linear pointed matched circle. There are injective differential algebra homomorphisms $L_n:\mathcal{A}_n\hookrightarrow\mathcal{A}_{n+1}$.
\end{proposition}
\begin{proof}
	Consider the injective map $\iota_n:[4n]\hookrightarrow[4n+4]$ given by
	\begin{align}
		\iota_n:i\mapsto\begin{cases}
			4\,\,\,\,\textup{if $i=1$}\\
			i+4\,\,\,\,\,\textup{else.}
		\end{cases}
	\end{align}
	Given any partial permutation $(S,T,\sigma)\in\mathcal{A}(n+1,-1)$ and $h\in[4n+4]\smallsetminus(S\cup T)$, define $(S,T,\sigma)_h\in\mathcal{A}(n+1,0)$ by
	\begin{align}
		(S,T,\sigma)_h=(S\cup\{h\},T\cup\{h\},\sigma_h)
	\end{align}
	where $\sigma_h$ is the extension of $\sigma$ to $S\cup\{h\}$ such that $\sigma_h(h)=h$. Suppose that $a\in\mathcal{A}_n$ is a basis element which decomposes into partial permutations as
	\begin{align}
		a=\sum_{j=1}^m(S_j,T_j,\sigma_j)
	\end{align}
	and define
	\begin{align*}
		L_n(a)=\sum_{j=1}^m\,\,\sum_{h=1,3}(\iota_n(S),\iota_n(T),\iota_n\circ\sigma_j\circ\iota_n^{-1})_h,
	\end{align*}
	where $\iota_n^{-1}:\iota_n([4n])\to[4n]$ is the inverse of the bijection $\iota_n:[4n]\to\iota_n([4n])$, extending linearly to obtain a map $L_n:\mathcal{A}_n\to\mathcal{A}_{n+1}$. Since the height of each inserted strand is at most 3, it follows immediately that $L_n$ is injective and that $L_n(\partial a)=\partial L_n(a)$ and $L_n(ab)=L_n(a)L_n(b)$ for all algebra elements $a,b\in\mathcal{A}_n$.
\end{proof}
\begin{proposition}
	Given a crossingless matching $a\in\mathfrak{C}_n$, there is an $\F$-vector space isomorphism $\lambda_a:\widehat{\CFD}(a_+)\to\widehat{\CFD}(a_{++})$ such that $\bm{x}_i\stackrel{A_{ij}}{\longrightarrow}\bm{x}_j$ is an arrow in the graph $\Gamma_{\widehat{\CFD}(a_+)}$ if and only if $\lambda_a(\bm{x}_i)\stackrel{L_n(A_{ij})+\delta_{ij}\rho_{1,3}}{\longrightarrow}\lambda_a(\bm{x}_j)$ is an arrow in the graph $\Gamma_{\widehat{\CFD}(a_{++})}$.
\end{proposition}
\begin{proof}
	The diagrams $\mathcal{H}_{a_+}$ and $\mathcal{H}_{a_{++}}$ are of the form shown in Figure \ref{UpDiagrams}
	\begin{figure}
		\begin{center}
			\includegraphics[scale=1]{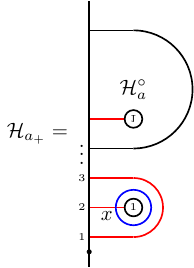}\hspace{1cm}\includegraphics[scale=1]{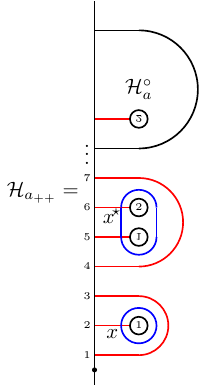}
		\end{center}
		\caption[The diagrams $\mathcal{H}_{a_+}$ and $\mathcal{H}_{a_{++}}$.]{The diagrams $\mathcal{H}_{a_+}$ and $\mathcal{H}_{a_{++}}$.}
		\label{UpDiagrams}
	\end{figure}
	and we claim that there is a bijection between the sets of generators for $\mathcal{H}_{a_+}$ and $\mathcal{H}_{a_{++}}$. To see this, note that if $\bm{x}$ is a generator for $\mathcal{H}_{a_{++}}$, then $x\in\bm{x}$ since $x$ is the only intersection point on the bottom-most $\beta$-circle. Since there is exactly one intersection point in $\bm{x}$ lying on the next lowest $\beta$-circle and this point cannot lie on the same $\alpha$-arc as $x$, we also have $x^\star\in\bm{x}$. Therefore, we have a decomposition $\bm{x}=\bm{x}^\circ\cup\{x,x^\star\}$, where $\bm{x}^\circ$ is a collection of intersection points in $\mathcal{H}_a^\circ$. The desired bijection is then given by $\bm{x}^\circ\cup\{x\}\mapsto\bm{x}^\circ\cup\{x,x^\star\}$ and $\lambda_a$ is given by extending this bijection linearly. Note that the labels of the ends of the $\alpha$-arcs in $\mathcal{H}_a^\circ\subset\mathcal{H}_{a_{++}}$ are obtained from the labels of the $\alpha$-arcs in $\mathcal{H}_a^\circ\subset\mathcal{H}_{a_+}$ by applying $\iota_n$ and $\lambda_a(\bm{x})$ necessarily occupies the $\alpha$-arc labeled 2 and 5 but not the arcs labeled 1 and 3 or 4 and 7 so $I_D(\lambda_a(\bm{x}))=L_n(I_D(\bm{x}))$ for all $\bm{x}\in\widehat{\CFD}(a_+)$. By construction, the regions
	\begin{figure}
		\begin{center}
			\raisebox{0.25cm}{\includegraphics{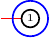}}\hspace{1cm}\includegraphics{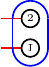}
		\end{center}
		\caption[Regions adjacent to the basepoint in $\mathcal{H}_{a_+}$ and $\mathcal{H}_{a_{++}}$.]{Regions adjacent to the basepoint in $\mathcal{H}_{a_+}$ (both) and $\mathcal{H}_{a_{++}}$ (right).}
		\label{BasepointRegions}
	\end{figure}
	inside the $\beta$-circles shown in Figure \ref{BasepointRegions} are adjacent to the basepoint in both $\mathcal{H}_{a_+}$ and $\mathcal{H}_{a_{++}}$ so the only domains contributing to the structure maps $\delta^1_{a_+}$ and $\delta^1_{a_{++}}$ are those supported in $\mathcal{H}_a^\circ$ and the annular domain $x\to x$ asymptotic to $\rho_{1,3}$ in both diagrams plus the annular domain $x^\star\to x^\star$ asymptotic to $\rho_{4,7}$ in $\mathcal{H}_{a_{++}}$ (see Figure \ref{fig:a-plus-plus-domains}). Now, there is a bijection between the sets of domains for index 1 holomorphic disks supported in $\mathcal{H}_a^\circ\subset\mathcal{H}_{a_+}$ and domains for index 1 holomorphic disks supported in $\mathcal{H}_a^\circ\subset\mathcal{H}_{a_{++}}$.
	This tells us that if $i\neq j$, then $\bm{x}_i\stackrel{A_{ij}}{\longrightarrow}\bm{x}_j$ is an arrow in $\Gamma_{\widehat{\CFD}(a_+)}$ if and only if $\lambda_a(\bm{x}_i)\stackrel{L_n(A_{ij})}{\longrightarrow}\lambda_a(\bm{x}_j)$ is an arrow in $\Gamma_{\widehat{\CFD}(a_{++})}$.
	\begin{figure}
		\begin{center}
			\includegraphics[scale=1]{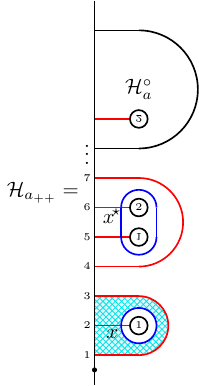}\hspace{1cm}\includegraphics[scale=1]{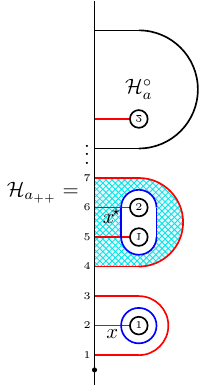}
		\end{center}
		\caption[The domains asymptotic to $\rho_{1,3}$ and $\rho_{4,7}$ in $\mathcal{H}_{a_{++}}$.]{The domains asymptotic to $\rho_{1,3}$ (left) and $\rho_{4,7}$ (right) in $\mathcal{H}_{a_{++}}$.}
		\label{fig:a-plus-plus-domains}
	\end{figure}
	This bijection, taken together with the existence of the annular domains $x\to x$ and $x^\star\to x^\star$, tells us that $\bm{x}_i\stackrel{A_i+\rho_{1,3}}{\longrightarrow}\bm{x}_i$ is an arrow in $\Gamma_{\widehat{\CFD}(a_+)}$ if and only if $\lambda_a(\bm{x}_i)\stackrel{L_n(A_i)+\rho_{1,3}+\rho_{4,7}}{\longrightarrow}\lambda_a(\bm{x}_j)$ is an arrow in $\Gamma_{\widehat{\CFD}(a_{++})}$. Since $\rho_{4,7}=L_n(\rho_{1,3})$, this proves the desired result.
\end{proof}
\begin{lemma}\label{Injective}
	Suppose that $\phi:C\to D$ is an injective chain map such that if $z\in\mathrm{im}(\phi)\cap\mathrm{im}(\partial_D)$, then $z=\partial_D y$ for some $y\in\mathrm{im}(\phi)$. Then the induced map $\phi_*:H_*(C)\to H_*(D)$ is injective.
\end{lemma}
\begin{proof}
	Suppose that $[x]\in\ker(\phi_*)$, then $0=\phi_*([x])=[\phi(x)]$ so $\phi(x)\in\mathrm{im}(\partial_D)$ and, hence, $\phi(x)=\partial_D w$ for some $w\in\mathrm{im}(\phi)$. Therefore, we have $\phi(x)=\partial_D(\phi(u))=\phi(\partial_C u)$ for some $u\in C$. Since $\phi$ is injective, we then have that $x=\partial_C u$ so $[x]=0$. Therefore, $\phi_*$ is injective.
\end{proof}
\begin{corollary}
	There is a homologically injective embedding $\Lambda_n:\mathfrak{h}_n\hookrightarrow\mathfrak{h}_{n+1}$ of differential algebras. Moreover, there is a direct sum decomposition
	\begin{align}
		\End^{\mathcal{A}_{n+1}}\left(\,\bigoplus_{a\in\mathfrak{C}_n}\widehat{\CFD}(a_{++})\right)=\mathrm{im}(\Lambda_n)\oplus\mathrm{im}(\rho_{1,3}\Lambda_n)
	\end{align}
	of vector spaces with respect to which the restriction of the differential on $\mathfrak{h}_{n+1}$ is block-diagonal. As a consequence, the map $(\Lambda_n)_*:H_*\mathfrak{h}_n\to H_*\mathfrak{h}_{n+1}$ is injective.
\end{corollary}
\begin{proof}
	We claim that the injective linear map $\Lambda_n:\mathfrak{h}_n\hookrightarrow\mathfrak{h}_{n+1}$ given on a basic morphism $f:\bm{x}\mapsto\rho\bm{y}$ by the morphism $\Lambda_nf:\lambda_a(\bm{x})\mapsto L_n(\rho)\lambda_a(\bm{y})$ is a differential algebra homomorphism. Note that if $g:\bm{y}\mapsto\sigma\bm{z}$ is another basic morphism with $\Lambda_n g:\lambda_b(\bm{y})\mapsto L_n(\sigma)\lambda_c(\bm{z})$, then, by construction, we have $\Lambda_n(f\circ_{2}g):\lambda_a(\bm{x})\mapsto L_n(\rho\sigma)\lambda_c(\bm{z})$ and $L_n(\rho\sigma)\lambda_c(\bm{z})=L_n(\rho)L_n(\sigma)\lambda_c(\bm{z})$ since $L_n$ is an algebra homomorphism so $\Lambda_n(f\circ_{2}g)=\Lambda_nf\circ_{2}\Lambda_ng$. Therefore, $\Lambda_n$ is an algebra homomorphism. Now consider the part
	\begin{center}
		\includegraphics[scale=1]{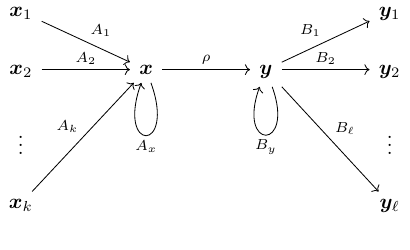}
	\end{center}
	of the graph $\Gamma_{f}$ (see Definition \ref{MorphismGraph}) contributing to $\partial f$ and the corresponding part
	\begin{center}
		\includegraphics[scale=1]{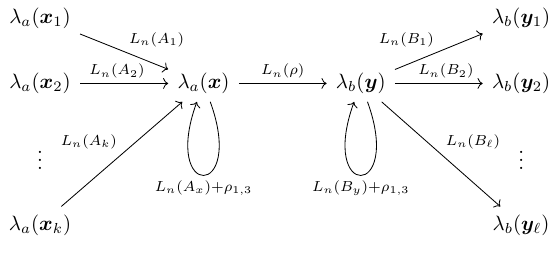}
	\end{center}
	of the graph $\Gamma_{\Lambda_nf}$. We compute
	\begin{align}
		\begin{split}
			\partial(\Lambda_nf)=&[\lambda_a(\bm{x})\mapsto (L_n(A_x)+\rho_{1,3})L_n(\rho)\lambda_b(\bm{y})]\\&+\sum_{i=1}^k[\lambda_a(\bm{x}_i)\mapsto L_n(A_i)L_n(\rho)\lambda_b(\bm{y})]\\&+[\lambda_a(\bm{x})\mapsto L_n(\rho)(L_n(A_y)+\rho_{1,3})\lambda_b(\bm{y})]\\&+\sum_{j=1}^\ell[\lambda_a(\bm{x})\mapsto L_n(\rho)L_n(B_j)\lambda_b(\bm{y}_j)]\\=&[\lambda_a(\bm{x})\mapsto L_n(A_x\rho)\lambda_b(\bm{y})]+\sum_{i=1}^k[\lambda_a(\bm{x}_i)\mapsto L_n(A_i\rho)\lambda_b(\bm{y})]\\&+[\lambda_a(\bm{x})\mapsto L_n(\rho A_y)\lambda_b(\bm{y})]+\sum_{j=1}^\ell[\lambda_a(\bm{x})\mapsto L_n(\rho B_j)\lambda_b(\bm{y}_j)]\\&+[\lambda_a(\bm{x})\mapsto \rho_{1,3}L_n(\rho)\lambda_b(\bm{y})]+[\lambda_a(\bm{x})\mapsto L_n(\rho)\rho_{1,3}\lambda_b(\bm{y})],
		\end{split}
	\end{align}
	where the second equality follows from the fact that $L_n$ is an algebra homomorphism. This then gives us
	\begin{align}
		\begin{split}
			\partial(\Lambda_nf)=&\Lambda_n[\bm{x}\mapsto A_x\rho\bm{y}]+\sum_{i=1}^k\Lambda_n[\bm{x}_i\mapsto A_i\rho\bm{y}]\\&+\Lambda_n[\bm{x}\mapsto\rho A_y\bm{y}]+\sum_{j=1}^\ell\Lambda_n[\bm{x}\mapsto \rho B_j\bm{y}_j]\\&+[\lambda_a(\bm{x})\mapsto \rho_{1,3}L_n(\rho)\lambda_b(\bm{y})]+[\lambda_a(\bm{x})\mapsto L_n(\rho)\rho_{1,3}\lambda_b(\bm{y})]\\=&\Lambda_n(\partial f)+[\lambda_a(\bm{x})\mapsto [\rho_{1,3},L_n(\rho)]\lambda_b(\bm{y})],
		\end{split}
	\end{align}
	where $[\rho_{1,3},L_n(\rho)]=\rho_{1,3}L_n(\rho)+L_n(\rho)\rho_{1,3}$ is the commutator of $\rho_{1,3}$ and $L_n(\rho)$. However, by construction of $L_n$, we have that $[\rho_{1,3},L_n(\rho)]=0$ for all $\rho\in\mathcal{A}_n$ so $\partial(\Lambda_nf)=\Lambda_n(\partial f)$ and $\Lambda_n$ is an injective differential algebra homomorphism. Now note that, again by construction of $L_n$, no element of $\mathrm{im}(\Lambda_n)$ is of the form $[\lambda_a(\bm{x})\mapsto\rho\lambda_b(\bm{y})]$, where $\rho\in\rho_{1,3}\mathcal{A}_{n+1}$ so $\mathrm{im}(\Lambda_n)\cap\mathrm{im}(\rho_{1,3}\Lambda_n)=\{0\}$. Moreover, since $I_D(\lambda_a(\bm{x}))=L_n(I_D(\bm{x}))$ for any $\bm{x}\in\widehat{\CFD}(a_+)$ and any generator of $\widehat{\CFD}(a_{++})$ is of the form $\lambda_a(\bm{x})$, the only algebra elements acting nontrivially on the module $\bigoplus_{a\in\mathfrak{C}_n}\widehat{\CFD}(a_{++})$ are those in $L_n(\mathcal{A}_n)\oplus\rho_{1,3}L_n(\mathcal{A}_n)\subset\mathcal{A}_{n+1}$.
	
	Therefore, if $f=[\lambda_a(\bm{x})\mapsto\rho\lambda_b(\bm{y})]\in\End^{\mathcal{A}_{n+1}}\left(\,\bigoplus_{a\in\mathfrak{C}_n}\widehat{\CFD}(a_{++})\right)$ is a basic morphism, then either $\rho\in L_n(\mathcal{A}_n)$ or $\rho\in\rho_{1,3}L_n(\mathcal{A}_n)$. Since the basic morphisms form a basis for $\End^{\mathcal{A}_{n+1}}\left(\,\bigoplus_{a\in\mathfrak{C}_n}\widehat{\CFD}(a_{++})\right)$, this shows that
	\begin{align}
		\End^{\mathcal{A}_{n+1}}\left(\,\bigoplus_{a\in\mathfrak{C}_n}\widehat{\CFD}(a_{++})\right)=\mathrm{im}(\Lambda_n)\oplus\mathrm{im}(\rho_{1,3}\Lambda_n).
	\end{align}
	Since we have shown that $\Lambda_n$ is a chain map, to show that the restriction of the differential is block diagonal with respect to this decomposition, it remains to show that $\mathrm{im}(\rho_{1,3}\Lambda_n)$ is closed under the differential. However, the computation showing that $\Lambda_n$ is a chain map can be readily adapted, mutatis mutandis, to show that $\partial(\rho_{1,3}\Lambda_nf)=\rho_{1,3}\Lambda_n(\partial f)$. Lastly, note that the morphism spaces $\Mor^{\mathcal{A}_{n+1}}(c_{+},d_{+})$ are closed under the differential for all $c,d\in\mathfrak{C}_{n+1}$ so $g\in\mathfrak{h}_{n+1}$ is an element of $\mathrm{im}(\Lambda_n)\cap\mathrm{im}(\partial)$ if and only if $g=\partial f$ for some $f\in\mathrm{im}(\Lambda_n)$. Therefore, by Lemma \ref{Injective}, the map $(\Lambda_n)_*$ is injective.
\end{proof}
\begin{theorem}
	The differential algebras $\mathfrak{h}_n$ are not formal for $n>1$.
\end{theorem}
\begin{proof}
	We will show in Appendix \ref{Homology}, by a lengthy but straightforward computation, that $\mathfrak{h}_2$ is non-formal with nontrivial $m_3$ operation. Since $\mathfrak{h}_2$ embeds homologically injectively in $\mathfrak{h}_n$ for all $n>1$, this proves that $\mathfrak{h}_n$ is non-formal with nontrivial $m_3$ for all $n>1$.
\end{proof}
	\appendix
	\section{The $A_\infty$-Algebra $H_*\mathfrak{h}_2$}
	\subsection{The algebra $\mathcal{A}_2$}
Before we proceed, we will need a complete description of the algebra $\mathcal{A}_2$. Consider the genus 2 linear pointed matched circle
\begin{align}
	\mathcal{Z}_2=\,\,\,\,\raisebox{-2.15cm}{\includegraphics[scale=1]{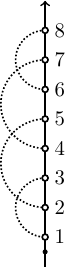}}.
\end{align}
The matching $M:[8]\to[4]$ determining $\mathcal{Z}$ is given by $M(1)=M(3)=1$, $M(2)=M(5)=2$, $M(4)=M(7)=3$, and $M(6)=M(8)=4$.
The algebra $\mathcal{A}_2$ contains six orthogonal idempotents $\iota_0=I(\{1,2\})$, $\iota_1=I(\{1,3\})$, $\iota_2=I(\{1,4\})$, $\iota_3=I(\{2,3\})$, $\iota_4=I(\{2,4\})$, and $\iota_5=I(\{3,4\})$, which are depicted below.

\begin{align*}
	\iota_0=\raisebox{-0.875cm}{\includegraphics[scale=0.5]{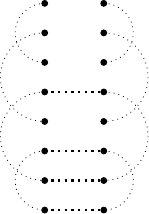}}\,=\,\raisebox{-0.875cm}{\includegraphics[scale=0.5]{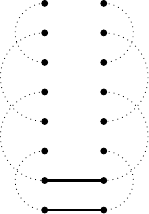}}+\raisebox{-0.875cm}{\includegraphics[scale=0.5]{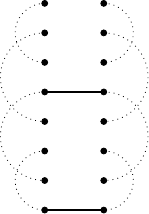}}+\raisebox{-0.875cm}{\includegraphics[scale=0.5]{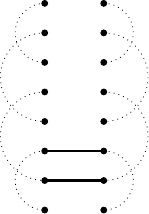}}+\raisebox{-0.875cm}{\includegraphics[scale=0.5]{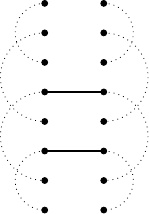}}
\end{align*}
\begin{align*}
	\iota_1=\raisebox{-0.875cm}{\includegraphics[scale=0.5]{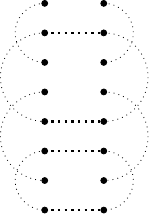}}\,=\,\raisebox{-0.875cm}{\includegraphics[scale=0.5]{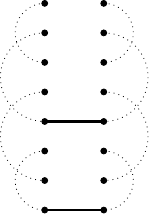}}+\raisebox{-0.875cm}{\includegraphics[scale=0.5]{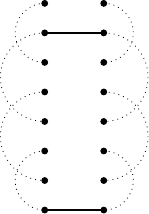}}+\raisebox{-0.875cm}{\includegraphics[scale=0.5]{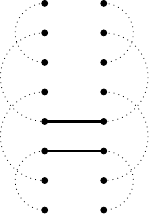}}+\raisebox{-0.875cm}{\includegraphics[scale=0.5]{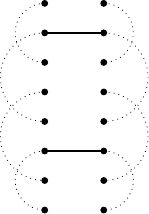}}
\end{align*}
\begin{align*}
	\iota_2=\raisebox{-0.875cm}{\includegraphics[scale=0.5]{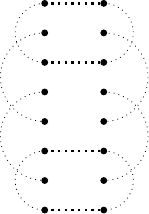}}\,=\,\raisebox{-0.875cm}{\includegraphics[scale=0.5]{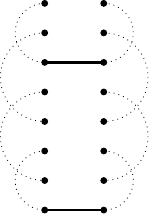}}+\raisebox{-0.875cm}{\includegraphics[scale=0.5]{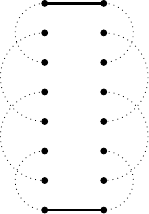}}+\raisebox{-0.875cm}{\includegraphics[scale=0.5]{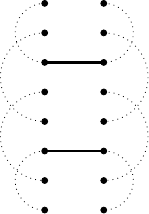}}+\raisebox{-0.875cm}{\includegraphics[scale=0.5]{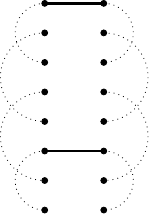}}
\end{align*}
\begin{align*}
	\iota_3=\raisebox{-0.875cm}{\includegraphics[scale=0.5]{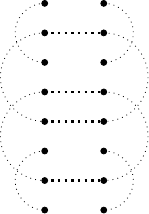}}\,=\,\raisebox{-0.875cm}{\includegraphics[scale=0.5]{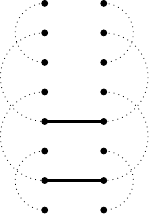}}+\raisebox{-0.875cm}{\includegraphics[scale=0.5]{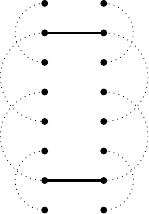}}+\raisebox{-0.875cm}{\includegraphics[scale=0.5]{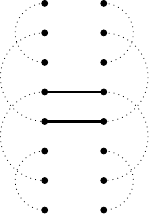}}+\raisebox{-0.875cm}{\includegraphics[scale=0.5]{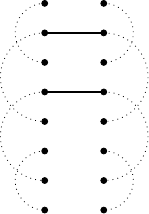}}
\end{align*}
\begin{align*}
	\iota_4=\raisebox{-0.875cm}{\includegraphics[scale=0.5]{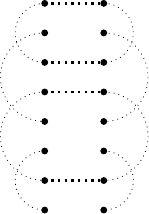}}\,=\,\raisebox{-0.875cm}{\includegraphics[scale=0.5]{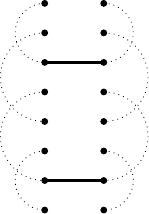}}+\raisebox{-0.875cm}{\includegraphics[scale=0.5]{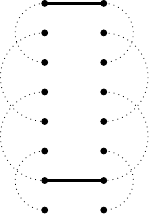}}+\raisebox{-0.875cm}{\includegraphics[scale=0.5]{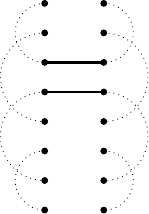}}+\raisebox{-0.875cm}{\includegraphics[scale=0.5]{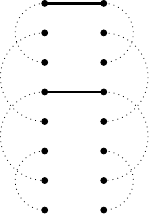}}
\end{align*}
\begin{align*}
	\iota_5=\raisebox{-0.875cm}{\includegraphics[scale=0.5]{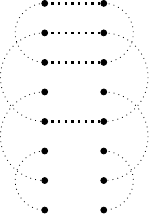}}\,=\,\raisebox{-0.875cm}{\includegraphics[scale=0.5]{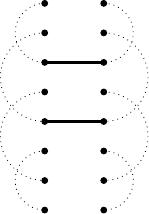}}+\raisebox{-0.875cm}{\includegraphics[scale=0.5]{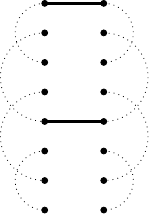}}+\raisebox{-0.875cm}{\includegraphics[scale=0.5]{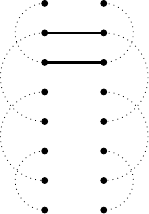}}+\raisebox{-0.875cm}{\includegraphics[scale=0.5]{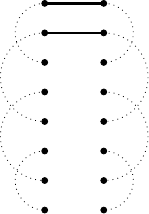}}
\end{align*}
For a string $0\leq a_1<a_2<\cdots<a_k\leq 5$, define an idempotent $\iota_{a_1a_2\cdots a_k}$ by
\begin{align}
	\iota_{a_1a_2\cdots a_k}=\sum_{i=1}^k\iota_{a_i}
\end{align}
and, for $0\leq i<j\leq 7$, let $\reeb{i,j}$ be the strands algebra element determined by the Reeb chord in $\mathcal{Z}$ from $i$ to $j$. In this notation, $\mathcal{A}_2$ has 28 single Reeb chord generators.
$\mathcal{A}_2$ also has 179 double Reeb chord generators $\reebII{i,j}{k,\ell}=\iota_a\reebII{i,j}{k,\ell}\iota_b$, for $i<k$, corresponding to the sets of Reeb chords $\{[i,j],[k,\ell]\}$. However, many of these are redundant as they are products of single chord generators. For the sake of completeness, we list all of these generators, their idempotents, and their differentials below in Figures \ref{fig:SingleReeb}, \ref{fig:DoubleReeb1}, \ref{fig:DoubleReeb2}, \ref{fig:DoubleReeb3}, and \ref{fig:DoubleReeb4}.
\begin{figure}
	\begin{center}
		\includegraphics{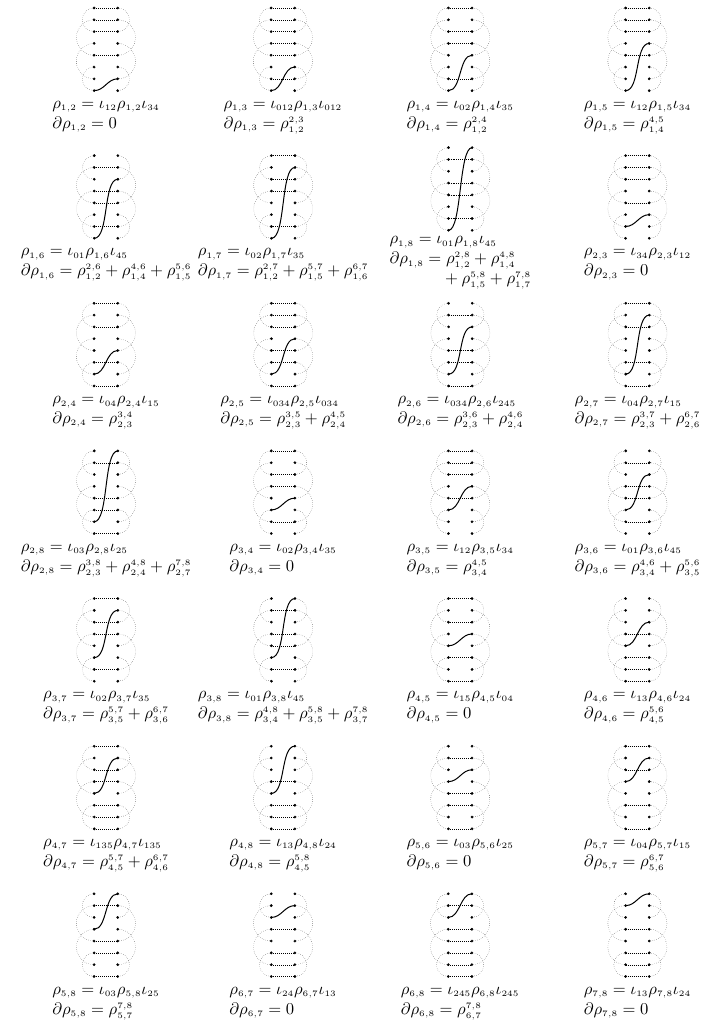}
	\end{center}
	\caption[Single Reeb chord generators of $\mathcal{A}_2$.]{Single Reeb chord generators of $\mathcal{A}_2$. Dotted horizontal strands indicate that we sum over all ways of inserting a single horizontal strand at each corresponding height.}
	\label{fig:SingleReeb}
\end{figure}
\begin{figure}
	\begin{center}
		\includegraphics{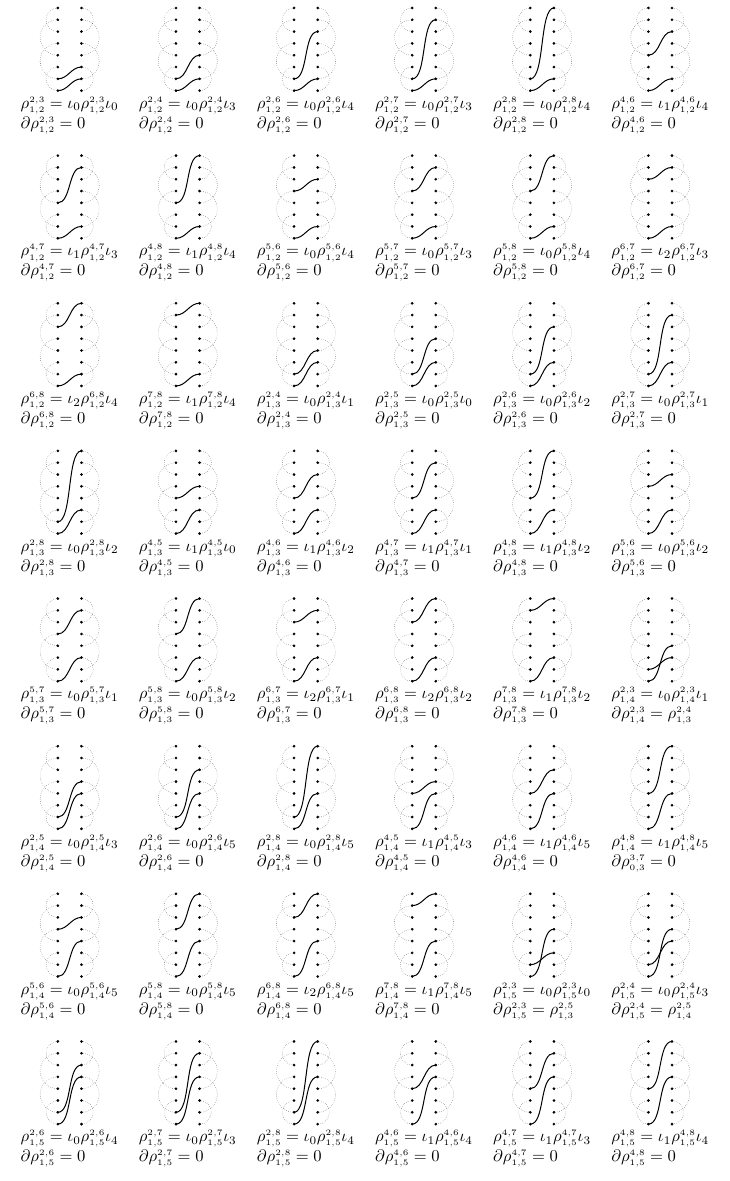}
	\end{center}
	\caption{Double Reeb chord generators of $\mathcal{A}_2$ (Part I).}
	\label{fig:DoubleReeb1}
\end{figure}
\begin{figure}
	\begin{center}
		\includegraphics{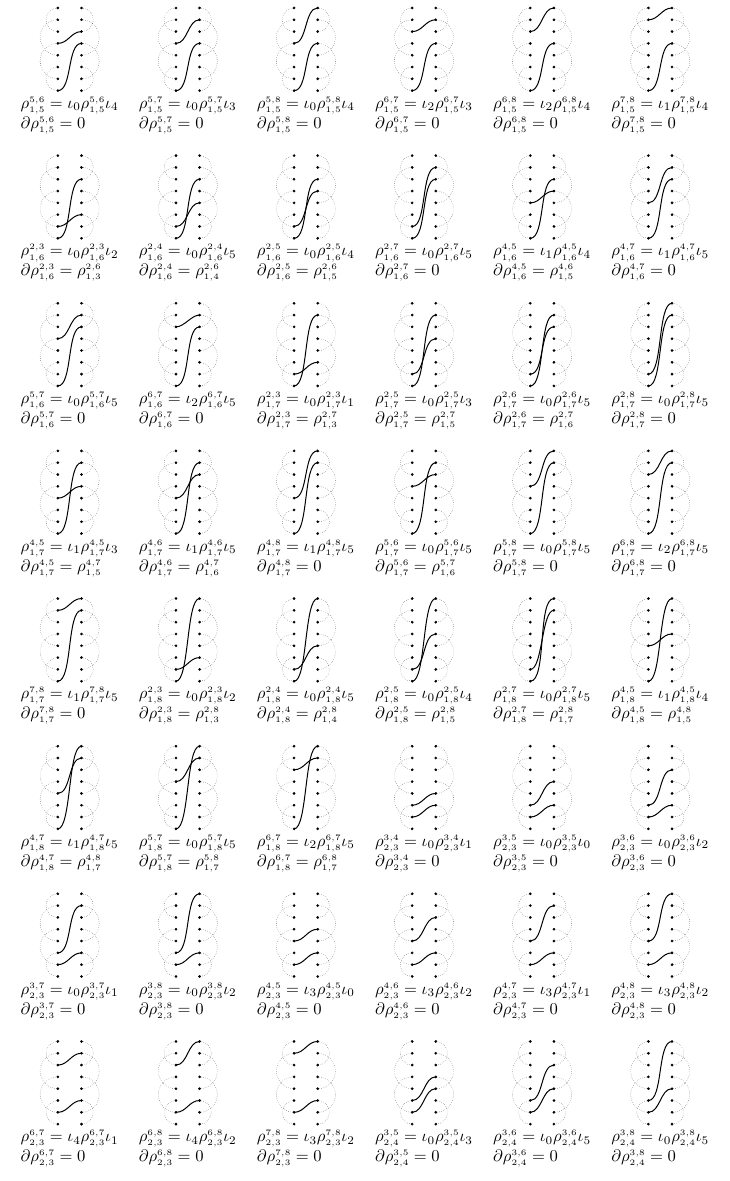}
	\end{center}
	\caption{Double Reeb chord generators of $\mathcal{A}_2$ (Part II).}
	\label{fig:DoubleReeb2}
\end{figure}
\begin{figure}
	\begin{center}
		\includegraphics{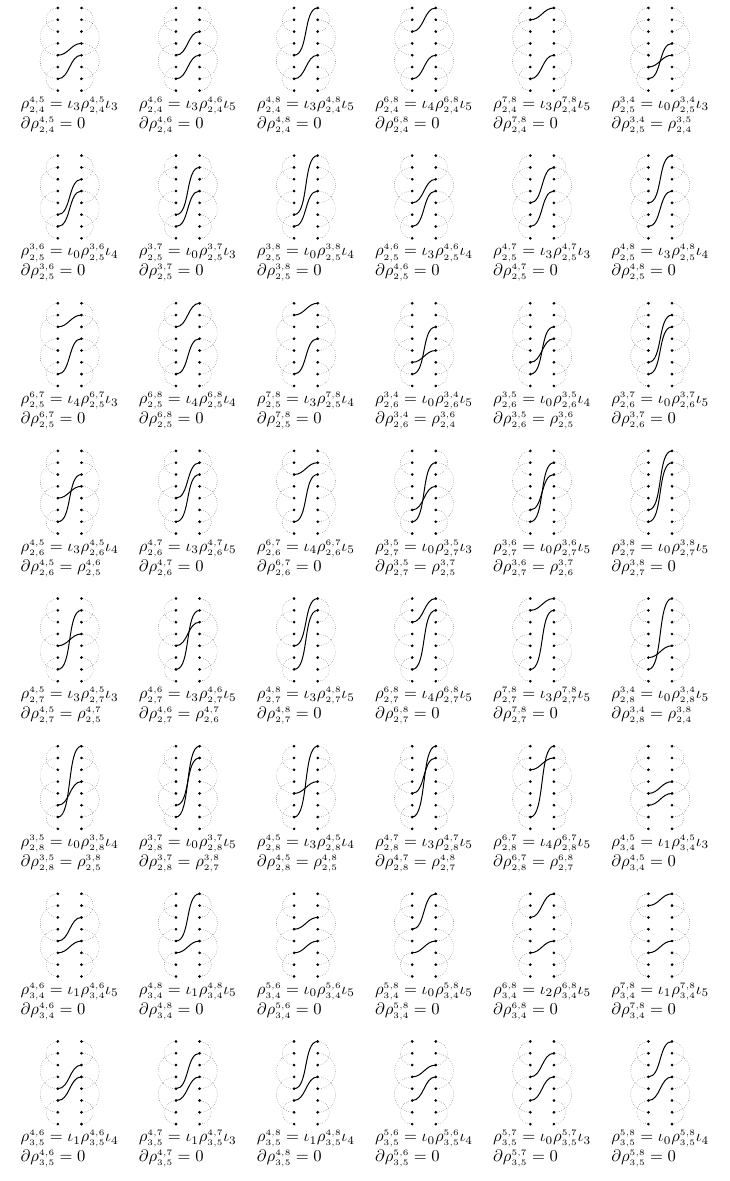}
	\end{center}
	\caption{Double Reeb chord generators of $\mathcal{A}_2$ (Part III).}
	\label{fig:DoubleReeb3}
\end{figure}
\begin{figure}
	\begin{center}
		\includegraphics{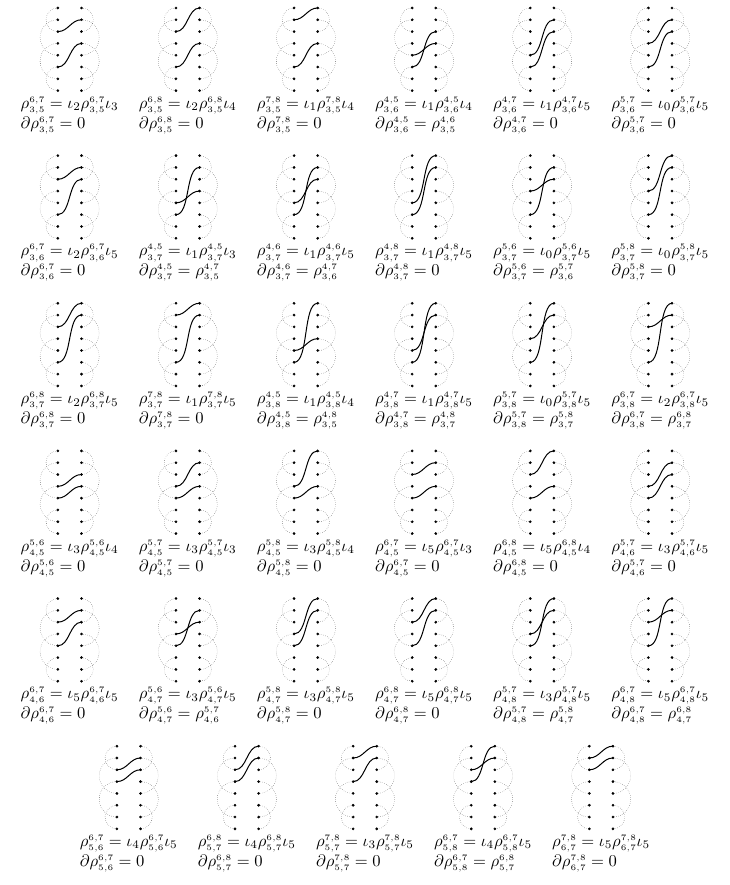}
	\end{center}
	\caption{Double Reeb chord generators of $\mathcal{A}_2$ (Part IV).}
	\label{fig:DoubleReeb4}
\end{figure}
\newpage
\section{The branched arc algebra $\mathfrak{h}_2$}\label{Homology}
The branched arc algebra $\mathfrak{h}_2$ is the endomorphism algebra
\begin{align}
	\mathfrak{h}_2=\End^{\mathcal{A}_2}\left(\,\bigoplus_{a\in\mathfrak{B}_2}\widehat{\CFD}(a)\right).
\end{align}
The set $\mathfrak{B}_2$ of crossingless matchings on six points of the form $c_+$ consists of the two planar diagrams
\begin{align}
	\begin{split}
		&\raisebox{-1.5cm}{\includegraphics[scale=0.6]{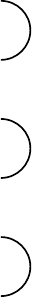}}\,\,\,\,\hspace{0.5cm}\textup{and}\hspace{0.75cm}\raisebox{-1.5cm}{\includegraphics[scale=0.6]{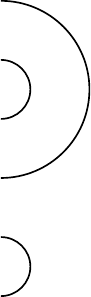}}
	\end{split}
\end{align}
which we denote by $a_1$ and $a_2$, respectively. As $a_1$ is the one-ended plat closure of the six stranded identity braid, the first part of the algorithm given on page \pageref{fig:Plat} furnishes us with the bordered Heegaard diagram $\mathcal{H}_1$ for $\Sigma(a_1)$ shown below.
\begin{align}
	\mathcal{H}_1=\,\,\,\,\raisebox{-2.25cm}{\includegraphics[scale=1]{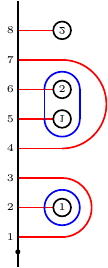}}
\end{align}
Isotope $a_2$ to obtain its minimal plat closure-form as follows.
\begin{align}
	\begin{split}
		\raisebox{-1.5cm}{\includegraphics[scale=0.6]{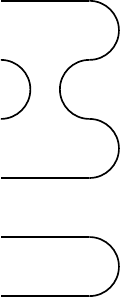}}
	\end{split}
\end{align}
Inserting a new handle and $\beta$-curve into $\mathcal{H}_1$ for the cap-cup pair in this diagram, then simplifying using the destabilization procedure detailed on page \pageref{CapCupHandle}, gives us the following Heegaard diagram
\begin{align}
	\raisebox{-2.25cm}{\includegraphics[scale=1]{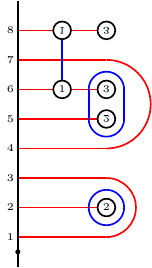}}\hspace{0.25cm}\rightsquigarrow\hspace{0.25cm}\raisebox{-2.25cm}{\includegraphics[scale=1]{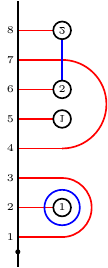}}\,\,\,\,=:\mathcal{H}_2
\end{align}
for $\Sigma(a_i)$. We now compute $\widehat{\CFD}(a_i)$ for $i=1,3$.
\subsubsection*{$\widehat{\CFD}(a_1)$:}
It is not hard to see that
\begin{align}
	\mathcal{H}_1=\hspace{0.25cm}\raisebox{-2.25cm}{\includegraphics[scale=1]{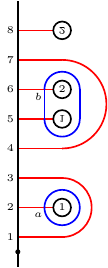}}
\end{align}
has a single generator $\bm{t}=\{a,b\}$ with $\iota_1\bm{t}=\bm{t}$ and supports the following index 1 domains from $\bm{t}$ to itself:

\begin{align}
	\begin{split}
		&\raisebox{-2.25cm}{\includegraphics[scale=1]{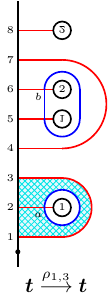}}\hspace{1cm}\raisebox{-2.25cm}{\includegraphics[scale=1]{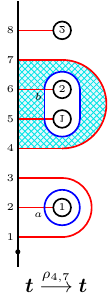}}
	\end{split}
\end{align}
giving us
\begin{align}
	\widehat{\CFD}(a_1)=\begin{tikzcd}[ampersand replacement=\&]
		\bm{t}\arrow[loop right,looseness=8,out=35,in=-35,"\reeb{1,3}+\reeb{4,7}"]
	\end{tikzcd}
\end{align}
which is to say that $\widehat{\CFD}(\mathcal{H}_1)=\F\langle\bm{t}\rangle$ with $\delta^1(\bm{t})=(\reeb{1,3}+\reeb{4,7})\otimes\bm{t}$. This coincides with the computation in \S 5.2 of \cite{LOTSpectral1}.
\subsubsection*{$\widehat{\CFD}(a_2)$:}
By inspection, the diagram
\begin{align}
	\mathcal{H}_2=\hspace{0.25cm}\raisebox{-2.25cm}{\includegraphics[scale=1]{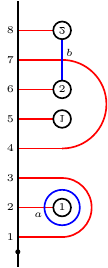}}
\end{align}
has a single generator $\bm{w}=\{a,b\}$ with $\iota_2\bm{w}=\bm{w}$ and supports the domains

\begin{align}
	\begin{split}
		&\raisebox{-2.25cm}{\includegraphics[scale=1]{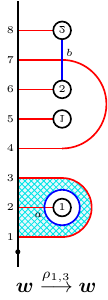}}\hspace{1cm}\raisebox{-2.25cm}{\includegraphics[scale=1]{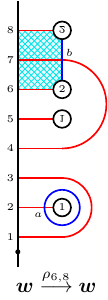}}
	\end{split}
\end{align}
giving us
\begin{align}
	\widehat{\CFD}(a_2)=\begin{tikzcd}[ampersand replacement=\&]
		\bm{w}\arrow[loop right,looseness=8,out=35,in=-35,"\reeb{1,3}+\reeb{6,8}"]
	\end{tikzcd}
\end{align}
i.e. $\widehat{\CFD}(\mathcal{H}_2)=\F\langle\bm{v}\rangle$ with $\delta^1(\bm{w})=(\reeb{1,3}+\reeb{6,8})\otimes\bm{w}$. Strictly speaking, the structure coefficients for these type-$D$ structures should be of the form $\rho\iota_i$ but if $\bm{\xi}$ is some generator with $\iota_i\bm{\xi}=\bm{\xi}$, then we have that $\rho\otimes\bm{\xi}=\rho\otimes(\iota_i\bm{\xi})=(\rho\iota_i)\otimes\bm{\xi}$ so this distinction is essentially cosmetic.

\subsection{The morphism spaces $\Mor(i,j)$}
Given $i,j\in\{1,2\}$, let
\begin{align}
	\Mor(i,j)=\Mor^{\mathcal{A}_2}\left(\widehat{\CFD}(a_i),\widehat{\CFD}(a_j)\right)
\end{align}
be the space of $\mathcal{A}_2$-module homomorphisms $f:\widehat{\CFD}(\mathcal{H}_i)\to\widehat{\CFD}(\mathcal{H}_j)$. Then
\begin{align}
	\mathfrak{h}_2=\Mor(1,1)\oplus\Mor(1,2)\oplus\Mor(2,1)\oplus\Mor(2,2).
\end{align}
We compute each summand separately.
\subsubsection{$\Mor^{\mathcal{A}_2}(1,1)$}
Since
\begin{align}
	\widehat{\CFD}(a_1)=\begin{tikzcd}[ampersand replacement=\&]
		\bm{t}\arrow[loop right,looseness=8,out=35,in=-35,"\reeb{1,3}+\reeb{4,7}"]
	\end{tikzcd}
\end{align}
and $\bm{t}=\iota_1\bm{t}$, a basic $\mathcal{A}_2$-module homomorphism $f:\widehat{\CFD}(a_1)\to\widehat{\CFD}(a_1)$ is determined by $f(\bm{t})=\rho\bm{t}$ where $\rho\in\mathcal{A}_2$ satisfies $\rho=\iota_1\rho\iota_1$. One may verify that the possible values of $\rho$ are $\iota_1$, $\iota_1\reeb{1,3}\iota_1$, $\iota_1\reeb{4,7}\iota_1$, and $\reebII{1,3}{4,7}$. Therefore, we have
\begin{align}
	\Mor(1,1)=\F\langle\fn{f}{1,1}{1},\fn{f}{1,1}{2},\fn{f}{1,1}{3},\fn{f}{1,1}{4}\rangle,
\end{align}
where
\begin{align}
	\begin{split}
		&\fn{f}{1,1}{1}(\bm{t})=\bm{t}\\
		&\fn{f}{1,1}{2}(\bm{t})=\iota_1\reeb{1,3}\iota_1\bm{t}\\
		&\fn{f}{1,1}{3}(\bm{t})=\iota_1\reeb{4,7}\iota_1\bm{t}\\
		&\fn{f}{1,1}{4}(\bm{t})=\reebII{1,3}{4,7}\bm{t}
	\end{split}
\end{align}
and $\dim_\F H_*\Mor(1,1)=\dim_\F\widehat{\HF}(\#^2S^2\times S^1)=4$ so $\partial f=0$ for every generator $f\in\Mor(1,1)$ and $H_*\Mor(1,1)=\F\langle[\fn{f}{1,1}{1}],[\fn{f}{1,1}{2}],[\fn{f}{1,1}{3}],[\fn{f}{1,1}{4}]\rangle$.
\subsubsection{$\Mor(1,2)$}
Here, we have
\begin{align}
	\widehat{\CFD}(a_2)=\begin{tikzcd}[ampersand replacement=\&]
		\bm{w}\arrow[loop right,looseness=8,out=35,in=-35,"\reeb{1,3}+\reeb{6,8}"]
	\end{tikzcd}
\end{align}
with $\bm{w}=\iota_2\bm{w}$ so a basic $\mathcal{A}_2$-module homomorphism $f:\widehat{\CFD}(a_1)\to\widehat{\CFD}(a_2)$ is determined by $f(\bm{t})=\rho\bm{w}$ where $\rho=\iota_1\rho\iota_2$. We then have that
\begin{align}
	\Mor(1,2)=\F\langle\fn{f}{1,2}{1},\fn{f}{1,2}{2},\fn{f}{1,2}{3},\fn{f}{1,2}{4},\fn{f}{1,2}{5},\fn{f}{1,2}{6}\rangle
\end{align}
where
\begin{align}
	\begin{alignedat}{2}
		&\fn{f}{1,2}{1}(\bm{t})=\iota_1\reeb{4,6}\iota_2\bm{w}\quad\quad\quad && \fn{f}{1,2}{4}(\bm{t})=\reebII{1,3}{4,6}\bm{w}\\
		&\fn{f}{1,2}{2}(\bm{t})=\iota_1\reeb{4,8}\iota_2\bm{w}\quad\quad && \fn{f}{1,2}{5}(\bm{t})=\reebII{1,3}{4,8}\bm{w}\\
		&\fn{f}{1,2}{3}(\bm{t})=\iota_1\reeb{7,8}\iota_2\bm{w}\quad\quad && \fn{f}{1,2}{6}(\bm{t})=\reebII{1,3}{7,8}\bm{w}\\
	\end{alignedat}
\end{align}
and one may verify that
\begin{align}
	\begin{alignedat}{2}
		&\partial\fn{f}{1,2}{1}=\fn{f}{1,2}{2} \quad\quad\quad && \partial\fn{f}{1,2}{4}=\fn{f}{1,2}{5} \\
		&\partial\fn{f}{1,2}{2}=0 \quad\quad && \partial\fn{f}{1,2}{5}=0 \\
		&\partial\fn{f}{1,2}{3}=\fn{f}{1,2}{2} \quad\quad && \partial\fn{f}{1,2}{6}=\fn{f}{1,2}{5} \\
	\end{alignedat}
\end{align}
so, as a chain complex, $\Mor(1,2)$ is given graphically by
\begin{align}
	\begin{tikzcd}[ampersand replacement=\&]
		\fn{f}{1,2}{1}\arrow[r] \& \fn{f}{1,2}{2} \& \fn{f}{1,2}{3}\arrow[l]\\
		\fn{f}{1,2}{4}\arrow[r] \& \fn{f}{1,2}{5} \& \fn{f}{1,2}{6}\arrow[l]
	\end{tikzcd},
\end{align}
where an arrow $\fn{f}{1,2}{i}\to\fn{f}{1,2}{j}$ means that $\fn{f}{1,2}{j}$ has coefficient 1 in $\partial\fn{f}{1,2}{i}$. This complex has 2-dimensional homology with basis consisting of the classes $[\fn{f}{1,2}{1}+\fn{f}{1,2}{3}]$ and $[\fn{f}{1,2}{4}+\fn{f}{1,2}{6}]$.

\subsubsection{$\Mor(2,1)$}
Since
\begin{align}
	\widehat{\CFD}(a_3)=\begin{tikzcd}[ampersand replacement=\&]
		\bm{w}\arrow[loop right,looseness=8,out=35,in=-35,"\reeb{1,3}+\reeb{6,8}"]
	\end{tikzcd}
\end{align}
with $\iota_2\bm{w}=\bm{w}$ and
\begin{align}
	\widehat{\CFD}(a_1)=\begin{tikzcd}[ampersand replacement=\&]
		\bm{t}\arrow[loop right,looseness=8,out=35,in=-35,"\reeb{1,3}+\reeb{4,7}"]
	\end{tikzcd}
\end{align}
with $\iota_1\bm{t}=\bm{t}$, a basic morphism $f\in\Mor(2,1)$ is determined by $f(\bm{w})=\rho\bm{t}$, where $\rho=\iota_2\rho\iota_1$ so
\begin{align}
	\Mor(2,1)=\F\langle\fn{f}{2,1}{1},\fn{f}{2,1}{2}\rangle
\end{align}
where
\begin{align}
	\fn{f}{2,1}{1}(\bm{w})=\iota_2\reeb{6,7}\iota_1\bm{t}\quad\quad\quad \fn{f}{2,1}{2}(\bm{w})=\reebII{1,3}{6,7}\bm{t}
\end{align}
and $\dim_\F H_*\Mor(2,1)=\dim_\F\widehat{\HF}(S^2\times S^1)=2$ so it follows that $\partial\fn{f}{2,1}{1}=\partial\fn{f}{2,1}{2}=0$ and $H_*\Mor(2,1)=\F\langle[\fn{f}{2,1}{1}],[\fn{f}{2,1}{2}]\rangle$.

\subsubsection{$\Mor(2,2)$}
Lastly, since
\begin{align}
	\widehat{\CFD}(\mathcal{H}_3)=\begin{tikzcd}[ampersand replacement=\&]
		\bm{w}\arrow[loop right,looseness=8,out=35,in=-35,"\reeb{1,3}+\reeb{6,8}"]
	\end{tikzcd}
\end{align}
with $\iota_2\bm{w}=\bm{w}$, a basic morphism $f\in\Mor(2,2)$ is given by $f(\bm{w})=\rho\bm{w}$, where $\rho=\iota_2\rho\iota_2$. Therefore,
\begin{align}
	\Mor(2,2)=\F\langle\fn{f}{2,2}{1},\fn{f}{2,2}{2},\fn{f}{2,2}{3},\fn{f}{2,2}{4}\rangle,
\end{align}
where
\begin{align}
	\begin{alignedat}{2}
		& \fn{f}{2,2}{1}(\bm{w})=\bm{w} \quad\quad\quad && \fn{f}{2,2}{3}(\bm{w})=\iota_2\reeb{6,8}\iota_2\bm{w}\\
		& \fn{f}{2,2}{2}(\bm{w})=\iota_2\reeb{1,3}\iota_2\bm{w} \quad\quad\quad && \fn{f}{2,2}{4}(\bm{w})=\reebII{1,3}{6,8}\bm{w}
	\end{alignedat}
\end{align}
and $\dim_\F H_*\Mor(2,2)=\dim_\F\widehat{\HF}(\#^2S^2\times S^1)=4$ so the differential on $\Mor(2,2)$ vanishes and $H_*\Mor(2,2)=\F\langle[\fn{f}{2,2}{1}],[\fn{f}{2,2}{2}],[\fn{f}{2,2}{3}],[\fn{f}{2,2}{4}]\rangle$.
\subsubsection{Summary}
To recap, we have
\begin{align}
	\mathfrak{h}_2=\Mor(1,1)\oplus\Mor(1,2)\oplus\Mor(2,1)\oplus\Mor(2,2).
\end{align}
Here,
\begin{align}
	\Mor(1,1)=\F\langle\fn{f}{1,1}{1},\fn{f}{1,1}{2},\fn{f}{1,1}{3},\fn{f}{1,1}{4}\rangle,
\end{align}
is given by
\begin{align}
	\begin{split}
		&\fn{f}{1,1}{1}(\bm{t})=\bm{t}\\
		&\fn{f}{1,1}{2}(\bm{t})=\iota_1\reeb{1,3}\iota_1\bm{t}\\
		&\fn{f}{1,1}{3}(\bm{t})=\iota_1\reeb{4,7}\iota_1\bm{t}\\
		&\fn{f}{1,1}{4}(\bm{t})=\reebII{1,3}{4,7}\bm{t}
	\end{split}
\end{align}
with vanishing differential, so $H_*\Mor(1,1)=\F\langle[\fn{f}{1,1}{1}],[\fn{f}{1,1}{2}],[\fn{f}{1,1}{3}],[\fn{f}{1,1}{4}]\rangle$. Next,
\begin{align}
	\Mor(1,2)=\F\langle\fn{f}{1,2}{1},\fn{f}{1,2}{2},\fn{f}{1,2}{3},\fn{f}{1,2}{4},\fn{f}{1,2}{5},\fn{f}{1,2}{6}\rangle
\end{align}
where
\begin{align}
	\begin{alignedat}{2}
		&\fn{f}{1,2}{1}(\bm{t})=\iota_1\reeb{4,6}\iota_2\bm{w}\quad\quad\quad && \fn{f}{1,2}{4}(\bm{t})=\reebII{1,3}{4,6}\bm{w}\\
		&\fn{f}{1,2}{2}(\bm{t})=\iota_1\reeb{4,8}\iota_2\bm{w}\quad\quad && \fn{f}{1,2}{5}(\bm{t})=\reebII{1,3}{4,8}\bm{w}\\
		&\fn{f}{1,2}{3}(\bm{t})=\iota_1\reeb{7,8}\iota_2\bm{w}\quad\quad && \fn{f}{1,2}{6}(\bm{t})=\reebII{1,3}{7,8}\bm{w}\\
	\end{alignedat}
\end{align}
and
\begin{align}
	\begin{alignedat}{2}
		&\partial\fn{f}{1,2}{1}=\fn{f}{1,2}{2} \quad\quad\quad && \partial\fn{f}{1,2}{4}=\fn{f}{1,2}{5} \\
		&\partial\fn{f}{1,2}{2}=0 \quad\quad && \partial\fn{f}{1,2}{5}=0 \\
		&\partial\fn{f}{1,2}{3}=\fn{f}{1,2}{2} \quad\quad && \partial\fn{f}{1,2}{6}=\fn{f}{1,2}{5} \\
	\end{alignedat}
\end{align}
so this complex has 2-dimensional homology with basis consisting of the classes $[\fn{f}{1,2}{1}+\fn{f}{1,2}{3}]$ and $[\fn{f}{1,2}{4}+\fn{f}{1,2}{6}]$. Next, we have
\begin{align}
	\Mor(2,1)=\F\langle\fn{f}{2,1}{1},\fn{f}{2,1}{2}\rangle
\end{align}
where
\begin{align}
	\fn{f}{2,1}{1}(\bm{w})=\iota_2\reeb{6,7}\iota_1\bm{t}\quad\quad\quad \fn{f}{2,1}{2}(\bm{w})=\reebII{1,3}{6,7}\bm{t}
\end{align}
with vanishing differential, so $H_*\Mor(2,1)=\F\langle[\fn{f}{2,1}{1}],[\fn{f}{2,1}{2}]\rangle$. Lastly, we have
\begin{align}
	\Mor(2,2)=\F\langle\fn{f}{2,2}{1},\fn{f}{2,2}{2},\fn{f}{2,2}{3},\fn{f}{2,2}{4}\rangle,
\end{align}
where
\begin{align}
	\begin{alignedat}{2}
		& \fn{f}{2,2}{1}(\bm{w})=\bm{w} \quad\quad\quad && \fn{f}{2,2}{3}(\bm{w})=\iota_2\reeb{6,8}\iota_2\bm{w}\\
		& \fn{f}{2,2}{2}(\bm{w})=\iota_2\reeb{1,3}\iota_2\bm{w} \quad\quad\quad && \fn{f}{2,2}{4}(\bm{w})=\reebII{1,3}{6,8}\bm{w},
	\end{alignedat}
\end{align}
also with vanishing differential, so $H_*\Mor(2,2)=\F\langle[\fn{f}{2,2}{1}],[\fn{f}{2,2}{2}],[\fn{f}{2,2}{3}],[\fn{f}{2,2}{4}]\rangle$.
\subsection{$\mathfrak{h}_2$ and its homology}
We now describe $(\mathfrak{h}_2,\circ_{2})$ and its homology algebra $(H_*\mathfrak{h}_2,m_2=\overline{\circ}_2)$ explicitly, as associative algebras, and give a partial description of $H_*\mathfrak{h}_2$ as an $A_\infty$-algebra. One may verify using the above computations that $\mathfrak{h}_2$ has multiplication table with respect to the basis of basic morphisms as in Figure \ref{fig:MultiplicationChain}.
\begin{sidewaysfigure}[h]
	\begin{center}
		\begin{align*}
			\begin{array}{|c||c|c|c|c|c|c|c|c|c|c|c|c|c|c|c|c|}
				\hline \circ_{2}(\textup{row},\textup{col})\vphantom{\frac{A}{A^B}} & \fn{f}{1,1}{1} & \fn{f}{1,1}{2} & \fn{f}{1,1}{3} & \fn{f}{1,1}{4} & \fn{f}{1,2}{1} & \fn{f}{1,2}{2} & \fn{f}{1,2}{3} & \fn{f}{1,2}{4} & \fn{f}{1,2}{5} & \fn{f}{1,2}{6} & \fn{f}{2,1}{1} & \fn{f}{2,1}{2} & \fn{f}{2,2}{1} & \fn{f}{2,2}{2} & \fn{f}{2,2}{3} & \fn{f}{2,2}{4} \\
				\hhline{|=#=|=|=|=|=|=|=|=|=|=|=|=|=|=|=|=|} \fn{f}{1,1}{1}\vphantom{\frac{A}{A^B}} & \fn{f}{1,1}{1} & \fn{f}{1,1}{2} & \fn{f}{1,1}{3} & \fn{f}{1,1}{4} & \fn{f}{1,2}{1} & \fn{f}{1,2}{2} & \fn{f}{1,2}{3} & \fn{f}{1,2}{4} & \fn{f}{1,2}{5} & \fn{f}{1,2}{6} &  &  &  &  &  &  \\
				\hline \fn{f}{1,1}{2}\vphantom{\frac{A}{A^B}} & \fn{f}{1,1}{2} &  & \fn{f}{1,1}{4} &  & \fn{f}{1,2}{4} & \fn{f}{1,2}{5} & \fn{f}{1,2}{6} &  &  &  &  &  &  &  &  &  \\
				\hline \fn{f}{1,1}{3}\vphantom{\frac{A}{A^B}} & \fn{f}{1,1}{3} & \fn{f}{1,1}{4} &  &  &  &  & \fn{f}{1,2}{2} &  &  & \fn{f}{1,2}{5} &  &  &  &  &  &  \\
				\hline \fn{f}{1,1}{4}\vphantom{\frac{A}{A^B}} & \fn{f}{1,1}{4} &  &  &  &  &  & \fn{f}{1,2}{5} &  &  &  &  &  &  &  &  &  \\
				\hline \fn{f}{1,2}{1}\vphantom{\frac{A}{A^B}} &  &  &  &  &  &  &  &  &  &  & \fn{f}{1,1}{3} & \fn{f}{1,1}{4} & \fn{f}{1,2}{1} & \fn{f}{1,2}{4} & \fn{f}{1,2}{2} & \fn{f}{1,2}{5} \\
				\hline \fn{f}{1,2}{2}\vphantom{\frac{A}{A^B}} &  &  &  &  &  &  &  &  &  &  &  &  & \fn{f}{1,2}{2} & \fn{f}{1,2}{5} &  &  \\
				\hline \fn{f}{1,2}{3}\vphantom{\frac{A}{A^B}} &  &  &  &  &  &  &  &  &  &  &  &  & \fn{f}{1,2}{3} & \fn{f}{1,2}{6} &  &  \\
				\hline \fn{f}{1,2}{4}\vphantom{\frac{A}{A^B}} &  &  &  &  &  &  &  &  &  &  & \fn{f}{1,1}{4} &  & \fn{f}{1,2}{4} &  & \fn{f}{1,2}{5} &  \\
				\hline 
				\fn{f}{1,2}{5}\vphantom{\frac{A}{A^B}} &  &  &  &  &  &  &  &  &  &  &  &  & \fn{f}{1,2}{5} &  &  &  \\
				\hline 
				\fn{f}{1,2}{6}\vphantom{\frac{A}{A^B}} &  &  &  &  &  &  &  &  &  &  &  &  & \fn{f}{1,2}{6} &  &  &  \\
				\hline 
				\fn{f}{2,1}{1}\vphantom{\frac{A}{A^B}} & \fn{f}{2,1}{1} & \fn{f}{2,1}{2} &  &  &  &  & \fn{f}{2,2}{3} &  &  & \fn{f}{2,2}{4} &  &  &  &  &  &  \\
				\hline 
				\fn{f}{2,1}{2}\vphantom{\frac{A}{A^B}} & \fn{f}{2,1}{2} &  &  &  &  &  & \fn{f}{2,2}{4} &  &  &  &  &  &  &  &  &  \\
				\hline 
				\fn{f}{2,2}{1}\vphantom{\frac{A}{A^B}} &  &  &  &  &  &  &  &  &  &  & \fn{f}{2,1}{1} & \fn{f}{2,1}{2} & \fn{f}{2,2}{1} & \fn{f}{2,2}{2} & \fn{f}{2,2}{3} & \fn{f}{2,2}{4} \\
				\hline 
				\fn{f}{2,2}{2}\vphantom{\frac{A}{A^B}} &  &  &  &  &  &  &  &  &  &  & \fn{f}{2,1}{2} &  & \fn{f}{2,2}{2} &  & \fn{f}{2,2}{4} &  \\
				\hline 
				\fn{f}{2,2}{3}\vphantom{\frac{A}{A^B}} &  &  &  &  &  &  &  &  &  &  &  &  & \fn{f}{2,2}{3} & \fn{f}{2,2}{4} &  &  \\
				\hline 
				\fn{f}{2,2}{4}\vphantom{\frac{A}{A^B}} &  &  &  &  &  &  &  &  &  &  &  &  & \fn{f}{2,2}{4} &  &  &  \\
				\hline 
			\end{array}
		\end{align*}
	\end{center}
	\caption[A multiplication table for $\mathfrak{h}_2$ given by the basis of basic morphisms.]{A multiplication table for $\mathfrak{h}_2$ given by the basis of basic morphisms. Entries containing 0 are left blank for readability.}
	\label{fig:MultiplicationChain}
\end{sidewaysfigure}
The algebra $H_*\mathfrak{h}_2$ has a basis given by the homology classes $[\fn{f}{1,1}{1}]$, $[\fn{f}{1,1}{2}]$, $[\fn{f}{1,1}{3}]$, $[\fn{f}{1,1}{4}]$, $[\fn{f}{1,2}{1}+\fn{f}{1,2}{3}]$, $[\fn{f}{1,2}{4}+\fn{f}{1,2}{6}]$, $[\fn{f}{2,1}{1}]$, $[\fn{f}{2,1}{2}]$, $[\fn{f}{2,2}{1}]$, $[\fn{f}{2,2}{2}]$, $[\fn{f}{2,2}{3}]$, and $[\fn{f}{2,2}{4}]$.
We define maps $p:\mathfrak{h}_2\to H_*\mathfrak{h}_2$ and $\iota:H_*\mathfrak{h}_2\to\mathfrak{h}_2$ by
\begin{align}
	\begin{alignedat}{2}
		&\fn{f}{1,1}{1}\mapsto[\fn{f}{1,1}{1}] \quad\quad\quad && \fn{f}{1,2}{5}\mapsto 0\\
		&\fn{f}{1,1}{2}\mapsto[\fn{f}{1,1}{2}] \quad\quad\quad && \fn{f}{1,2}{6}\mapsto 0\\
		&\fn{f}{1,1}{3}\mapsto[\fn{f}{1,1}{3}] \quad\quad\quad && \fn{f}{2,1}{1}\mapsto[\fn{f}{2,1}{1}]\\
		&\fn{f}{1,1}{4}\mapsto[\fn{f}{1,1}{4}] \quad\quad\quad && \fn{f}{2,1}{2}\mapsto[\fn{f}{2,1}{2}]\\
		&\fn{f}{1,2}{1}\mapsto [\fn{f}{1,2}{1}+\fn{f}{1,2}{3}] \quad\quad\quad && \fn{f}{2,2}{1}\mapsto[\fn{f}{2,2}{1}]\\
		&\fn{f}{1,2}{2}\mapsto 0 \quad\quad\quad && \fn{f}{2,2}{2}\mapsto[\fn{f}{2,2}{2}]\\
		&\fn{f}{1,2}{3}\mapsto 0 \quad\quad\quad && \fn{f}{2,2}{3}\mapsto[\fn{f}{2,2}{3}]\\
		&\fn{f}{1,2}{4}\mapsto [\fn{f}{1,2}{4}+\fn{f}{1,2}{6}] \quad\quad\quad && \fn{f}{2,2}{4}\mapsto[\fn{f}{2,2}{4}]\\
	\end{alignedat}
\end{align}
and
\begin{align}
	\begin{alignedat}{2}
		&[\fn{f}{1,1}{1}]\mapsto\fn{f}{1,1}{1} \quad\quad\quad && [\fn{f}{2,1}{1}]\mapsto\fn{f}{2,1}{1}\\
		&[\fn{f}{1,1}{2}]\mapsto\fn{f}{1,1}{2} \quad\quad\quad && [\fn{f}{2,1}{2}]\mapsto\fn{f}{2,1}{2}\\
		&[\fn{f}{1,1}{3}]\mapsto\fn{f}{1,1}{3} \quad\quad\quad && [\fn{f}{2,2}{1}]\mapsto\fn{f}{2,2}{1}\\
		&[\fn{f}{1,1}{4}]\mapsto\fn{f}{1,1}{4} \quad\quad\quad && [\fn{f}{2,2}{2}]\mapsto\fn{f}{2,2}{2}\\
		&[\fn{f}{1,2}{1}+\fn{f}{1,2}{3}]\mapsto\fn{f}{1,2}{1}+\fn{f}{1,2}{3} \quad\quad\quad && [\fn{f}{2,2}{3}]\mapsto\fn{f}{2,2}{3}\\
		&[\fn{f}{1,2}{4}+\fn{f}{1,2}{6}]\mapsto\fn{f}{1,2}{4}+\fn{f}{1,2}{6} \quad\quad\quad && [\fn{f}{2,2}{4}]\mapsto\fn{f}{2,2}{4},\\
	\end{alignedat}
\end{align}
respectively, so that $\iota p=\id$ on $\Mor(1,2)^\perp$ by construction. Now define $h:\mathfrak{h}_2\to\mathfrak{h}_2$ by $h(\fn{f}{1,2}{2})=\fn{f}{1,2}{3}$ and $h(\fn{f}{1,2}{5})=\fn{f}{1,2}{6}$ and by zero on all other generators. One may then check that $p\iota=\id$ and $\iota p=\id+\partial h+h\partial$ so that $p$, $\iota$, and $h$ define a retract as in the statement of the homological perturbation lemma. Using this retract, the homology algebra $H_*\mathfrak{h}_2$ has multiplication table given in terms of the above basis of homology classes as in Figure \ref{fig:MultiplicationHomology} --- compare this with the table for multiplication in $H_2$ given in Figure \ref{fig:MultiplicationArc2}.

Now we have $m_2([\fn{f}{2,1}{1}],[\fn{f}{1,1}{3}])=0$ and $m_2([\fn{f}{1,1}{3}],[\fn{f}{1,2}{1}+\fn{f}{1,2}{3}])=0$, so the sequence $([\fn{f}{2,1}{1}],[\fn{f}{1,1}{3}],[\fn{f}{1,2}{1}+\fn{f}{1,2}{3}])$ in $H_*\mathfrak{h}_2$ is Massey admissible in the sense of \cite[Definition 2.1.21]{LOTBimodules2015}, and one may verify directly that
\begin{align}
	m_3([\fn{f}{2,1}{1}],[\fn{f}{1,1}{3}],[\fn{f}{1,2}{1}+\fn{f}{1,2}{3}])=[\fn{f}{2,2}{3}].
\end{align}
Letting a hollow dot on a planar circle represent a label of 1 and a filled dot represent a label of $x$, as in \cite{CohenSplitting}, this $m_3$ may be realized graphically in terms of elements of $H_2$ as
\begin{align}
	m_3\scalebox{2}{$($}\!\bai,\aaix+\aaxi,\abi\!\scalebox{2}{$)$}=\bbix+\bbxi.
\end{align}
Letting $\alpha_1=[\fn{f}{2,1}{1}]$, $\alpha_2=[\fn{f}{1,1}{3}]$, and $\alpha_3=[\fn{f}{1,2}{1}+\fn{f}{1,2}{3}]$, we now compute the cycles $\xi_{i,j}=q_{j-i}(\alpha_{i+1},\dots,\alpha_j)$, where the $q_k$ are defined as in Proposition \ref{CanonicalAinfQuasiIso}, which contribute to the cycle
\begin{align}
	\sum_{0<k<3}\xi_{0,k}\xi_{k,3}
\end{align}
representing $m_3(\alpha_1,\alpha_2,\alpha_3)$. These are
\begin{align}
	\begin{split}
		&\xi_{0,1}=\fn{f}{2,1}{1}\\
		&\xi_{0,2}=0\\
		&\xi_{1,3}=\fn{f}{1,2}{3}\\
		&\xi_{2,3}=\fn{f}{1,2}{1}+\fn{f}{1,2}{3}
	\end{split}
\end{align}
so we get
\begin{align}
	\sum_{0<k<3}\xi_{0,k}\xi_{k,3}=\circ_2(\fn{f}{2,1}{1},\fn{f}{1,2}{3})=\fn{f}{2,2}{3},
\end{align}
as expected. This representing cycle is independent of the choices of the $\xi_{i,j}$ by \cite[Lemma 2.1.22]{LOTBimodules2015}, so $m_3$ is nontrivial, independent of our choice of retract, and this shows that $\mathfrak{h}_2$ is not formal. This finishes the proof of Theorem \ref{Nonformal}.

One may further interpret this $m_3$ operation as follows: if we regard $H_2$ as an endomorphism algebra
\begin{align}
	\End_{\mathrm{BN}}\left(\amatch\oplus\bmatch\right)
\end{align}
in Bar-Natan's dotted cobordism category \cite{Bar-Natan2005}, then the arc algebra element $\bai$ is an undotted saddle cobordism ${\color{red}S}:\bmatch\to\amatch$ and $\abi$ is an undotted saddle cobordism ${\color{blue}S}:\amatch\to\bmatch$, so $\aaix+\aaxi={\color{blue}S}{\color{red}S}$ and $\bbix+\bbxi={\color{red}S}{\color{blue}S}$. In these terms, we have $m_3({\color{red}S},{\color{blue}S}{\color{red}S},{\color{blue}S})={\color{red}S}{\color{blue}S}$, which one can view as measuring the fact that ${\color{red}S}{\color{blue}S}{\color{red}S}{\color{blue}S}=0$ in two different ways, corresponding to ${\color{blue}S}{\color{red}S}{\color{blue}S}={\color{red}S}{\color{blue}S}{\color{red}S}=0$, in the sense that ${\color{red}S}{\color{blue}S}$ annihilates ${\color{red}S}$ and ${\color{blue}S}$ on the left and right, respectively. Indeed, all of the nontrivial $m_3$ operations on $H_*\mathfrak{h}_2$ --- which we list below, along with their interpretations in terms of elements of $H_2$ --- can be interpreted as measuring some version of this fact, or the analogous fact when the cobordism ${\color{red}S}{\color{blue}S}{\color{red}S}{\color{blue}S}$ has a dot.
\begin{center}
	\begin{align*}
		\begin{array}{|c|l|}
			\hline H_*\mathfrak{h}_2 &  \multicolumn{1}{c|}{H_2} \\
			\hline m_3([\fn{f}{2,1}{1}],[\fn{f}{1,1}{3}],[\fn{f}{1,2}{1}+\fn{f}{1,2}{3}])=[\fn{f}{2,2}{3}] & m_3\scalebox{2}{$($}\!\bai,\aaix+\aaxi,\abi\!\scalebox{2}{$)$}=\bbix+\bbxi \\
			\hline m_3([\fn{f}{2,1}{2}],[\fn{f}{1,1}{3}],[\fn{f}{1,2}{1}+\fn{f}{1,2}{3}])=[\fn{f}{2,2}{4}] & m_3\scalebox{2}{$($}\!\bax,\aaix+\aaxi,\abi\!\scalebox{2}{$)$}=\bbxx \\
			\hline m_3([\fn{f}{2,1}{1}],[\fn{f}{1,1}{3}],[\fn{f}{1,2}{4}+\fn{f}{1,2}{6}])=[\fn{f}{2,2}{4}] & m_3\scalebox{2}{$($}\!\bai,\aaix+\aaxi,\abx\!\scalebox{2}{$)$}=\bbxx \\
			\hline m_3([\fn{f}{2,1}{1}],[\fn{f}{1,1}{4}],[\fn{f}{1,2}{1}+\fn{f}{1,2}{3}])=[\fn{f}{2,2}{4}] & m_3\scalebox{2}{$($}\!\bai,\aaxx,\abi\!\scalebox{2}{$)$}=\bbxx \\
			\hline m_3([\fn{f}{2,1}{1}],[\fn{f}{1,2}{1}+\fn{f}{1,2}{3}],[\fn{f}{2,2}{3}])=[\fn{f}{2,2}{3}] & m_3\scalebox{2}{$($}\!\bai,\abi,\bbix+\bbxi\!\scalebox{2}{$)$}=\bbix+\bbxi \\
			\hline m_3([\fn{f}{2,1}{2}],[\fn{f}{1,2}{1}+\fn{f}{1,2}{3}],[\fn{f}{2,2}{3}])=[\fn{f}{2,2}{4}] & m_3\scalebox{2}{$($}\!\bax,\abi,\bbix+\bbxi\!\scalebox{2}{$)$}=\bbxx \\
			\hline m_3([\fn{f}{2,1}{1}],[\fn{f}{1,2}{4}+\fn{f}{1,2}{6}],[\fn{f}{2,2}{3}])=[\fn{f}{2,2}{4}] & m_3\scalebox{2}{$($}\!\bai,\abx,\bbix+\bbxi\!\scalebox{2}{$)$}=\bbxx \\
			\hline m_3([\fn{f}{2,1}{1}],[\fn{f}{1,2}{1}+\fn{f}{1,2}{3}],[\fn{f}{2,2}{4}])=[\fn{f}{2,2}{4}] & m_3\scalebox{2}{$($}\!\bai,\abi,\bbxx\!\scalebox{2}{$)$}=\bbxx \\
			\hline
		\end{array}
	\end{align*}
\end{center}
The first and fifth lines correspond to $m_3({\color{red}S},{\color{blue}S}{\color{red}S},{\color{blue}S})={\color{red}S}{\color{blue}S}$ and $m_3({\color{red}S},{\color{blue}S},{\color{red}S}{\color{blue}S})={\color{red}S}{\color{blue}S}$, respectively, and the remaining lines can be obtained from one of these by placing a dot on one argument and on the output.
\begin{proposition}
	For all $n>1$, $H_*\mathfrak{h}_n$ is unbounded as an $A_\infty$-algebra.
\end{proposition}
\begin{proof}
	We claim that $m_{2+k}([\fn{f}{2,1}{1}],[\fn{f}{1,1}{3}],\stackrel{k}{\dots},[\fn{f}{1,1}{3}],[\fn{f}{1,2}{1}+\fn{f}{1,2}{3}])=[\fn{f}{2,2}{3}]$ for $k>0$ in $H_*\mathfrak{h}_2$. To see this, note that $m_2(\fn{f}{2,1}{1},\fn{f}{1,1}{3})=m_2(\fn{f}{1,1}{3},\fn{f}{1,1}{3})=0$ so we have
	\begin{align}
		\begin{split}
			&m_{2+k}([\fn{f}{2,1}{1}],[\fn{f}{1,1}{3}],\stackrel{k}{\dots},[\fn{f}{1,1}{3}],[\fn{f}{1,2}{1}+\fn{f}{1,2}{3}])\\&=m_2([\fn{f}{2,1}{1}],q_{1+k}([\fn{f}{1,1}{3}],\stackrel{k}{\dots},[\fn{f}{1,1}{3}],[\fn{f}{1,2}{1}+\fn{f}{1,2}{3}]))
		\end{split}
	\end{align}
	because all other trees of operations contributing to $m_{2+k}$ necessarily involve a multiplication of the form $m_2(\fn{f}{2,1}{1},\fn{f}{1,1}{3})$ or $m_2(\fn{f}{1,1}{3},\fn{f}{1,1}{3})$. Since $m_2(\fn{f}{1,1}{3},\fn{f}{1,1}{3})=0$, one may further show by induction that
	\begin{align}
		q_{1+k}([\fn{f}{1,1}{3}],\stackrel{k}{\dots},[\fn{f}{1,1}{3}],[\fn{f}{1,2}{1}+\fn{f}{1,2}{3}])=\fn{f}{1,2}{3}
	\end{align}
	so $m_{2+k}([\fn{f}{2,1}{1}],[\fn{f}{1,1}{3}],\stackrel{k}{\dots},[\fn{f}{1,1}{3}],[\fn{f}{1,2}{1}+\fn{f}{1,2}{3}])=pm_2(\fn{f}{2,1}{1},\fn{f}{1,2}{3})=p(\fn{f}{2,2}{3})=[\fn{f}{2,2}{3}]$, as desired. One may check that the sequence $([\fn{f}{2,1}{1}],[\fn{f}{1,1}{3}],\stackrel{k}{\dots},[\fn{f}{1,1}{3}],[\fn{f}{1,2}{1}+\fn{f}{1,2}{3}])$ is Massey admissible and that $m_{2+k}([\fn{f}{2,1}{1}],[\fn{f}{1,1}{3}],\stackrel{k}{\dots},[\fn{f}{1,1}{3}],[\fn{f}{1,2}{1}+\fn{f}{1,2}{3}])$ is indeed represented by the cycle $\fn{f}{2,2}{3}$ so $m_{2+k}$ is nontrivial for all $k>0$. Since $\mathfrak{h}_2\hookrightarrow\mathfrak{h}_n$ homologically injectively for all $n>2$, this then implies that $H_*\mathfrak{h}_n$ is unbounded for all $n>1$.
\end{proof}
This completes the proof of Theorem \ref{Nonformal}.
\begin{sidewaysfigure}
	\begin{center}
		\begin{align*}
			\begin{array}{|c||c|c|c|c|c|c|c|c|c|c|c|c|}
				\hline m_2([\textup{row}],[\textup{col}])\vphantom{\frac{A}{A^B}} & [\fn{f}{1,1}{1}] & [\fn{f}{1,1}{2}] & [\fn{f}{1,1}{3}] & [\fn{f}{1,1}{4}] & \![\fn{f}{1,2}{1}+\fn{f}{1,2}{3}]\! & \![\fn{f}{1,2}{4}+\fn{f}{1,2}{6}]\! & [\fn{f}{2,1}{1}] & [\fn{f}{2,1}{2}] & [\fn{f}{2,2}{1}] & [\fn{f}{2,2}{2}] & [\fn{f}{2,2}{3}] & [\fn{f}{2,2}{4}]\\
				\hhline{|=#=|=|=|=|=|=|=|=|=|=|=|=|} [\fn{f}{1,1}{1}]\vphantom{\frac{A}{A^B}} & [\fn{f}{1,1}{1}] & [\fn{f}{1,1}{2}] & [\fn{f}{1,1}{3}] & [\fn{f}{1,1}{4}] & \![\fn{f}{1,2}{1}+\fn{f}{1,2}{3}]\! & \![\fn{f}{1,2}{4}+\fn{f}{1,2}{6}]\! &  &  &  &  &  &  \\
				\hline [\fn{f}{1,1}{2}]\vphantom{\frac{A}{A^B}} & [\fn{f}{1,1}{2}] &  & [\fn{f}{1,1}{4}] &  & \![\fn{f}{1,2}{4}+\fn{f}{1,2}{6}]\! &  &  &  &  &  &  &  \\
				\hline 
				[\fn{f}{1,1}{3}]\vphantom{\frac{A}{A^B}} & [\fn{f}{1,1}{3}] & [\fn{f}{1,1}{4}] &  &  &  &  &  &  &  &  &  &  \\
				\hline 
				[\fn{f}{1,1}{4}]\vphantom{\frac{A}{A^B}} & [\fn{f}{1,1}{4}] &  &  &  &  &  &  &  &  &  &  &  \\
				\hline 
				\![\fn{f}{1,2}{1}+\fn{f}{1,2}{3}]\!\vphantom{\frac{A}{A^B}} &  &  &  &  &  &  & [\fn{f}{1,1}{3}] & [\fn{f}{1,1}{4}] & \![\fn{f}{1,2}{1}+\fn{f}{1,2}{3}]\! & \![\fn{f}{1,2}{4}+\fn{f}{1,2}{6}]\! &  &  \\
				\hline 
				\![\fn{f}{1,2}{4}+\fn{f}{1,2}{6}]\!\vphantom{\frac{A}{A^B}} &  &  &  &  &  &  & [\fn{f}{1,1}{4}] &  & \![\fn{f}{1,2}{4}+\fn{f}{1,2}{6}]\! &  &  &  \\
				\hline 
				[\fn{f}{2,1}{1}]\vphantom{\frac{A}{A^B}} & [\fn{f}{2,1}{1}] & [\fn{f}{2,1}{2}] &  &  & [\fn{f}{2,2}{3}] & [\fn{f}{2,2}{4}] &  &  &  &  &  &  \\
				\hline 
				[\fn{f}{2,1}{2}]\vphantom{\frac{A}{A^B}} & [\fn{f}{2,1}{2}] &  &  &  & [\fn{f}{2,2}{4}] &  &  &  &  &  &  &  \\
				\hline 
				[\fn{f}{2,2}{1}]\vphantom{\frac{A}{A^B}} &  &  &  &  &  &  & [\fn{f}{2,1}{1}] & [\fn{f}{2,1}{2}] & [\fn{f}{2,2}{1}] & [\fn{f}{2,2}{2}] & [\fn{f}{2,2}{3}] & [\fn{f}{2,2}{4}] \\
				\hline 
				[\fn{f}{2,2}{2}]\vphantom{\frac{A}{A^B}} &  &  &  &  &  &  & [\fn{f}{2,1}{2}] &  & [\fn{f}{2,2}{2}] &  & [\fn{f}{2,2}{4}] &  \\
				\hline 
				[\fn{f}{2,2}{3}]\vphantom{\frac{A}{A^B}} &  &  &  &  &  &  &  &  & [\fn{f}{2,2}{3}] & [\fn{f}{2,2}{4}] &  &  \\
				\hline 
				[\fn{f}{2,2}{4}]\vphantom{\frac{A}{A^B}} &  &  &  &  &  &  &  &  & [\fn{f}{2,2}{4}] &  &  &  \\
				\hline 
			\end{array}
		\end{align*}
	\end{center}
	\caption[A multiplication table for $H_*\mathfrak{h}_2$ given by a retract.]{The multiplication table for $H_*\mathfrak{h}_2$ given by the retract defined above.}
	\label{fig:MultiplicationHomology}
\end{sidewaysfigure}
\begin{sidewaysfigure}
	\begin{center}
		\begin{align*}
			\begin{array}{|c||c|c|c|c|c|c|c|c|c|c|c|c|}
				\hline \textup{row}\cdot\textup{col}\vphantom{\raisebox{-0.1cm}{\scalemath{2.4}{\bigcirc}}} & \aaii & \aaxi & \!\!\aaix+\aaxi\!\! & \aaxx & \abi & \abx & \bai & \bax & \bbii & \bbxi & \!\!\bbix+\bbxi\!\! & \bbxx\\
				\hhline{|=#=|=|=|=|=|=|=|=|=|=|=|=|} \aaii\vphantom{\raisebox{-0.1cm}{\scalemath{2.4}{\bigcirc}}} & \aaii & \aaxi & \!\!\aaix+\aaxi\!\! & \aaxx & \abi & \abx &  &  &  &  &  &  \\
				\hline
				\aaxi\vphantom{\raisebox{-0.1cm}{\scalemath{2.4}{\bigcirc}}} & \aaxi &  & \aaxx &  & \abx &  &  &  &  &  &  &  \\
				\hline 
				\aaix+\aaxi\vphantom{\raisebox{-0.1cm}{\scalemath{2.4}{\bigcirc}}} & \aaix & \aaxx &  &  &  &  &  &  &  &  &  &  \\
				\hline 
				\aaxx\vphantom{\raisebox{-0.1cm}{\scalemath{2.4}{\bigcirc}}} & \aaxx &  &  &  &  &  &  &  &  &  &  &  \\
				\hline 
				\abi\vphantom{\raisebox{-0.1cm}{\scalemath{2.4}{\bigcirc}}} &  &  &  &  &  &  & \!\!\aaix+\aaxi\!\! & \aaxx & \abi & \abx &  &  \\
				\hline
				\abx\vphantom{\raisebox{-0.1cm}{\scalemath{2.4}{\bigcirc}}} &  &  &  &  &  &  & \aaxx &  & \abx &  &  &  \\
				\hline
				\bai\vphantom{\raisebox{-0.1cm}{\scalemath{2.4}{\bigcirc}}} & \bai & \bax &  &  & \!\!\bbxi+\bbix\!\! & \bbxx &  &  &  &  &  &  \\
				\hline
				\bax\vphantom{\raisebox{-0.1cm}{\scalemath{2.4}{\bigcirc}}} & \bax &  &  &  & \bbxx &  &  &  &  &  &  &  \\
				\hline
				\bbii\vphantom{\raisebox{-0.1cm}{\scalemath{2.4}{\bigcirc}}} &  &  &  &  &  &  & \bai & \bax & \bbii & \bbxi & \!\!\bbix+\bbxi\!\! & \bbxx \\
				\hline
				\bbxi\vphantom{\raisebox{-0.1cm}{\scalemath{2.4}{\bigcirc}}} &  &  &  &  &  &  & \bax &  & \bbxi &  & \bbxx &  \\
				\hline
				\!\!\bbix+\bbxi\vphantom{\raisebox{-0.1cm}{\scalemath{2.4}{\bigcirc}}}\!\! &  &  &  &  &  &  &  &  & \bbix+\bbxi & \bbxx &  &  \\
				\hline
				\bbxx\vphantom{\raisebox{-0.1cm}{\scalemath{2.4}{\bigcirc}}} &  &  &  &  &  &  &  &  & \bbxx &  &  &  \\
				\hline        
			\end{array}
		\end{align*}
	\end{center}
	\caption[A multiplication table for $H_2$.]{A multiplication table for $H_2$. Note that we are using a nonstandard basis for $H_2$.}
	\label{fig:MultiplicationArc2}
\end{sidewaysfigure}
	\bibliographystyle{alpha}
	\bibliography{bib}
\end{document}